\documentclass[moor]{informs1}              

\usepackage{natbib}
 \NatBibNumeric
 \def\bibfont{\small}%
 \def\BIBand{and}%
 \bibpunct[, ]{[}{]}{,}{n}{}{,}%

\usepackage[colorlinks=true,breaklinks=true,bookmarks=true,urlcolor=blue,
     citecolor=blue,linkcolor=blue,bookmarksopen=false,draft=false]{hyperref}

\def\EMAIL#1{\href{mailto:#1}{#1}}  

\TheoremsNumberedThrough     

\EquationsNumberedThrough    

\usepackage{booktabs} 
\usepackage{multicol, multirow} 
\usepackage{float} 
\usepackage{caption} 
\usepackage{amsmath,amssymb} 
\usepackage{url,hyperref}
\usepackage{mathtools}
\usepackage{multirow}
\usepackage{graphicx}
\usepackage{color}
\usepackage{url,hyperref}
\usepackage{algorithm}
\usepackage{paralist}
\usepackage{algorithmicx,algpseudocode,algcompatible}

\newcommand{\Id}{\mathbb{I}}
\newcommand{\R}{\mathbb{R}}

\newcommand{\set}[1]{\left\{#1\right\}}
\newcommand{\sets}[1]{\{#1\}}

\newcommand{\norms}[1]{\Vert#1\Vert}

\newcommand{\iprods}[1]{\langle #1\rangle}
\newcommand{\Eproof}{\hfill $\square$}

\newcommand{\dom}[1]{\mathrm{dom}(#1)}
\newcommand{\zero}[1]{{\boldsymbol{0}}}

\newcommand{\Expsk}[2]{\mathbb{E}_{#1}\big[#2\big]}
\newcommand{\Expsn}[2]{\mathbb{E}_{#1}[#2]}
\newcommand{\Expn}[1]{\mathbb{E}[#1]}
\newcommand{\Expk}[1]{\mathbb{E}\big[#1\big]}

\newcommand{\zer}[1]{\mathrm{zer}(#1)}
\newcommand{\gra}[1]{\mathrm{gra}(#1)}

\newcommand{\Hc}{\mathcal{H}}
\newcommand{\Bc}{\mathcal{B}}
\newcommand{\Xc}{\mathcal{X}}
\newcommand{\Sc}{\mathcal{S}}
\newcommand{\Gc}{\mathcal{G}}
\newcommand{\Lc}{\mathcal{L}}
\newcommand{\Qc}{\mathcal{Q}}
\newcommand{\Tc}{\mathcal{T}}
\newcommand{\Fc}{\mathcal{F}}
\newcommand{\Ec}{\mathcal{E}}

\newcommand{\Rc}{\mathcal{R}}
\newcommand{\BigO}[1]{\mathcal{O}\left(#1\right)}
\newcommand{\BigOs}[1]{\mathcal{O}\big(#1\big)}
\newcommand{\SmallO}[1]{o\left(#1\right)}
\newcommand{\SmallOs}[1]{o\big(#1\big)}

\newcommand{\mbf}[1]{\mathbf{#1}}

\newcommand{\myeq}[2]{\vspace{-0.25ex}\begin{equation}\label{#1}#2\vspace{-0.25ex}\end{equation}}
\newcommand{\myeqn}[1]{\vspace{-0.25ex}\begin{equation*}#1\vspace{-0.25ex}\end{equation*}}

\newcommand{\myeqc}[1]{ {\scriptsize\textcircled{#1}} }

\newcommand{\beforesec}{\vspace{-0.5ex}}
\newcommand{\aftersec}{\vspace{-1ex}}
\newcommand{\beforesubsec}{\vspace{-1ex}}
\newcommand{\aftersubsec}{\vspace{-1ex}}

\begin{document}

\RUNAUTHOR{Quoc Tran-Dinh}

\RUNTITLE{Variance-Reduced Fast Krasnosel'ski\v{i}-Mann Methods for Finite-Sum Root-Finding Problems}

\TITLE{Variance-Reduced Fast Krasnosel'ski\v{i}-Mann Methods for Finite-Sum Root-Finding Problems}

\ARTICLEAUTHORS{%
\AUTHOR{Quoc Tran-Dinh}
\AFF{Department of Statistics and Operations Research \\ The University of North Carolina at Chapel Hill (UNC), 318-Hanes Hall, Chapel Hill, NC 27599-3260, USA.\\ 
\EMAIL{quoctd@email.unc.edu}}
} 

\ABSTRACT{
We develop a new class of fast Krasnosel'ski\v{i}-Mann methods with variance reduction to solve a finite-sum cocoercive equation $Gx = 0$. 
Our algorithm is single-loop and leverages a new family of unbiased variance-reduced estimators specifically designed for a wider class of root-finding algorithms.
Our method achieves both $\mathcal{O}(1/k^2)$ and $o(1/k^2)$ last-iterate convergence rates in terms of $\mathbb{E}[ \Vert{Gx^k}\Vert^2]$, where $k$ is the iteration counter and $\mathbb{E}[\cdot]$ is the total expectation.
We also establish almost-sure $o(1/k^2)$ convergence rates and the almost-sure convergence of iterates $\{x^k\}$ to a solution of $Gx=0$.
We instantiate our framework for two prominent estimators: SVRG and SAGA.
By an appropriate choice of parameters, both variants attain an oracle complexity of $\mathcal{O}\big({n + n^{2/3}\epsilon^{-1}}\big)$ to reach an $\epsilon$-solution, where $n$ denotes the number of summands in the finite-sum operator $G$.
Furthermore, under $\sigma$-strong quasi-monotonicity, our method achieves a linear convergence rate and an oracle complexity of $\mathcal{O}\big({n+ \max\sets{n, n^{2/3}\kappa} \log(\frac{1}{\epsilon})}\big)$, where $\kappa := L/\sigma$.
We extend our approach to solve a class of finite-sum inclusions (possibly nonmonotone), demonstrating that our schemes retain the same theoretical guarantees as in the equation setting. 
Finally, numerical experiments demonstrate that our methods exhibit promising performance compared with state-of-the-art methods.
}

\MSCCLASS{90C25, 90C06, 90-08}
\ORMSCLASS{Operations research; mathematical programming; game theory}

\maketitle

\vspace{-2ex}
\beforesec
\section{Introduction.}\label{sec:intro}
\aftersec
This paper develops and analyzes a new class of \textit{\textbf{stochastic fast}}  Krasnosel'ski\v{i}--Mann (KM) methods with \textbf{\textit{variance reduction}} for approximating solutions of finite-sum ``cocoercive'' equations, as well as their application to finite-sum composite inclusions.

\beforesubsec
\subsection{Problem statement.}\label{subsec:prob_stats}
\aftersubsec
Our first goal in this paper is to solve the following  finite-sum ``cocoercive" equation using novel stochastic variance reduction methods:
\myeq{eq:ME}{
\arraycolsep=0.2em
\begin{array}{lcl}
\textrm{Find $x^{\star}\in \R^p$ such that:}~ Gx^{\star} = 0, \quad\textrm{where}\quad Gx := \frac{1}{n}\sum_{i=1}^n G_ix.
\tag{CE}
\end{array}
}
Here, $G_i : \R^p \to \R^p$ are single-valued mappings for all $i \in [n] := \sets{1,\cdots, n}$ and $n \geq 1$.
Following the existing literature, such as \cite{davis2022variance}, we call \eqref{eq:ME} a \textit{root-finding problem}.

Equation \eqref{eq:ME} can be rewritten equivalently as the following \textit{fixed-point problem}:
\myeq{eq:FP}{
\textrm{Find $x^{\star} \in \R^p$ such that:}~ x^{\star} = Fx^{\star},
\tag{FP}
}
where $Fx := x - \lambda Gx$ for any given $\lambda > 0$. 
Note that when $\lambda = \frac{2}{L}$, $G$ in \eqref{eq:ME} is $\frac{1}{L}$-cocoercive (see Assumption~\ref{as:A1} below) if and only if $F$ is nonexpansive, i.e., $\norms{Fx - Fy} \leq \norms{x - y}$ for all $x, y \in \R^p$ \cite[Proposition 4.11]{Bauschke2011}. 
Since the two formulations are equivalent, methods for solving \eqref{eq:FP} can be applied to solve \eqref{eq:ME} and vice versa. In this work, we focus on \textit{the root-finding problem \eqref{eq:ME}}.

Throughout this paper, we make the following assumption on \eqref{eq:ME}.

\begin{assumption}\label{as:A1}
\begin{itemize}
\item[\hspace{-1ex}$\mathrm{(a)}$] 
Equation \eqref{eq:ME} has a solution, i.e., $\zer{G} := \sets{x^{\star} \in \R^p : Gx^{\star} = 0} \neq\emptyset$.

\item[$\mathrm{(b)}$]
 The mappings $G_i$, $i \in [n]$, satisfy the following ``cocoercivity-type'' condition:
There exists $L \in (0, +\infty)$ such that for all $x, y \in \R^p$, we have
\myeq{eq:co-coerciveness_of_G}{
\arraycolsep=0.2em
\begin{array}{lcl}
\frac{1}{n}\sum_{i=1}^n\iprods{G_ix - G_iy, x - y} & \geq & \frac{1}{L} \big[ \frac{1}{n} \sum_{i=1}^n \norms{G_ix - G_iy}^2 \big], \quad 
\end{array}
}
\end{itemize}
\end{assumption}
Note that \eqref{eq:co-coerciveness_of_G} implies the $\frac{1}{L}$-cocoercivity of $G$, i.e., $\iprods{Gx - Gy, x - y} \geq \frac{1}{L}\norms{Gx - Gy}^2$.
Conversely, if $G$ is $\frac{1}{L}$-cocoercive, then it may not satisfy \eqref{eq:co-coerciveness_of_G}.
However, if every $G_i$ is $\frac{1}{L}$-cocoercive for $i \in [n]$, then $G$ satisfies  \eqref{eq:co-coerciveness_of_G}.
Thus \eqref{eq:co-coerciveness_of_G} is stronger than the $\frac{1}{L}$-cocoercivity of $G$ but weaker than the $\frac{1}{L}$-cocoercivity of all the summands $G_i$.

Assumption~\ref{as:A1}(b) is not new. 
It can be viewed as a generalization of the $\frac{1}{L}$-average cocoercive property in finite-sum convex optimization when $G_i = \nabla{f_i}$, which is equivalent to the $L$-average smoothness: $\frac{1}{n}\sum_{i=1}^n\norms{\nabla{f}_i(x) - \nabla{f}_i(y)}^2 \leq L^2\norms{x - y}^2$ for all $x, y \in \R^p$. 
This average smoothness condition has been widely used in optimization and is slightly different from the expected smoothness condition considered in \cite{gower2019sgd,loizou2021stochastic}. 
If we fix $y := x^{\star} \in \zer{G}$, then \eqref{eq:co-coerciveness_of_G} reduces to a special case of \cite[Assumption 3.3]{loizou2021stochastic} for the finite-sum setting.

\beforesubsec
\subsection{Proposed method.}\label{subsec:proposed_method}
\aftersubsec
The classical Krasnosel'ski\v{i}-Mann (KM) fixed-point iteration was independently developed by M. A. Krasnosel'ski\v{i} \cite{Krasnoselskii1955} and W. R. Mann \cite{mann1953mean} to approximately solve \eqref{eq:FP} when $F$ is nonexpansive. 
More specifically, it can be expressed as follows:
\textit{Starting from an initial point $x^0 \in \R^p$, at each iteration $k \geq 0$, it updates
\myeq{eq:KM_scheme}{
x^{k+1} := (1- \alpha_k)x^k + \alpha_kFx^k,
\tag{KM}
}
where $\alpha_k \in (0, 1)$ is a given averaging parameter $($e.g., $\alpha = \frac{1}{2}$$)$ satisfying appropriate conditions.}

If we substitute $F = \Id - \frac{2}{L}G$ into \eqref{eq:KM_scheme} to solve \eqref{eq:ME},  we obtain $x^{k+1} = x^k - \frac{2\alpha_k}{L}Gx^k$, which can be viewed as a forward/gradient-type scheme, see, e.g., \cite{Facchinei2003}. 
The best convergence rate for the squared residual norm $\norms{x^k - Fx^k}^2$ in \eqref{eq:KM_scheme} (or, equivalently, $\norms{Gx^k}^2$ in the forward scheme) is $\BigOs{1/k}$, see, e.g., \cite{cominetti2014KMRate}. 
This rate is suboptimal and significantly slower than the $\BigOs{1/k^2}$ rate, which was proven to be optimal  in \cite{park2022exact} for first-order methods to solve \eqref{eq:ME} under Assumption~\ref{as:A1}.

To boost convergence rates of \eqref{eq:KM_scheme}, one can either apply Halpern's fixed-point iteration, as in \cite{diakonikolas2020halpern,halpern1967fixed,lieder2021convergence}, or Nesterov's acceleration, as in \cite{attouch2019convergence,bot2022bfast,tran2022connection}. 
The authors in \cite{bot2022bfast} proposed discretizing a second-order ordinary differential equation (ODE) associated with $F$, leading to the following scheme: \textit{Starting from $x^0 \in \R^p$, at each iteration $k \geq 0$, it updates
\myeq{eq:FKM_scheme}{\tag{FKM}
x^{k+1} := (1 - \tfrac{s\tau_k}{2})x^k + (1-s)\tau_k (x^k - x^{k-1}) + \tfrac{s\tau_k}{2} Fx^k + s\tau_k(Fx^k - Fx^{k-1}),
}
where $s \in (0, 1]$ and $\tau_k := \frac{r}{k+r}$ with $r > 2$ are given parameters.}
This scheme is called the \textit{fast KM iteration.} Note that \eqref{eq:FKM_scheme} is equivalent to Nesterov's accelerated algorithm in \cite{tran2022connection} applied to $F = \Id - \frac{2}{L}G$. Both \cite{tran2022connection} and \cite{bot2022bfast} proved faster $\BigOs{1/k^2}$ and $\SmallOs{1/k^2}$ convergence rates for deterministic fast KM methods compared to the non-accelerated scheme \eqref{eq:KM_scheme}.

Our method in this paper builds on deterministic Nesterov's accelerated KM methods  from our recent work \cite{tran2022connection} for the cocoercive equation \eqref{eq:ME} and \eqref{eq:FKM_scheme} from \cite{bot2022bfast} for the fixed-point problem \eqref{eq:FP}. 
More specifically, the method in \cite[Theorem 2]{tran2022connection} can be written in the following form after eliminating the iterate $y^k$:
\textit{Given $x^0 \in \R^p$, set $x^{-1} := x^0$, and at each iteration  $k \geq 0$, it updates
\myeq{eq:KM_method}{
x^{k+1} := x^k + \theta_k(x^k - x^{k-1}) - \eta_kS^k, \quad \textrm{where} \quad S^k := Gx^k - \gamma_kGx^{k-1},
}
where $\theta_k \in (0, 1)$ is an inertial or a momentum  parameter, $\eta_k > 0$ is a given stepsize, and $\gamma_k \in (0, 1)$ is a given correction parameter.
}

Since we aim to solve large finite-sum instances of \eqref{eq:ME} (i.e., when $n \gg 1$), for which evaluating $S^k$ exactly is often costly or even infeasible, we approximate $S^k$ using a stochastic unbiased estimator $\widetilde{S}^k$ with variance reduction.

\textit{\textbf{Our method.}}
Having $\widetilde{S}^k$, we propose the following method to solve \eqref{eq:ME}:
\textit{
Given $x^0 \in \dom{G}$, we set $x^{-1} := x^0$ and at each iteration $k \geq 0$, we construct an unbiased variance-reduced estimator $\widetilde{S}^k$ of $S^k := Gx^k - \gamma_kGx^{k-1}$ and then update
\myeq{eq:VRKM4ME}{
x^{k+1} := x^k + \theta_k(x^k - x^{k-1}) - \eta_k\widetilde{S}^k,
\tag{VFKM}
}
where $\theta_k \in (0, 1)$, $\eta_k > 0$, and $\gamma_k \in (0, 1)$ are given, determined later, and $\widetilde{S}^0 := S^0$.
}

Our method \eqref{eq:VRKM4ME} uses three parameters $\theta_k$, $\eta_k$, and $\gamma_k$.
However, we do not explicitly specify their forms here or the construction of $\widetilde{S}^k$ to avoid repetition.
Section~\ref{sec:VRS_Estimator} presents the construction of $\widetilde{S}^k$.
Sections~\ref{sec:AVFR4NE} and \ref{sec:star_monotone_convergence} provide explicit update rules for $\theta_k$, $\eta_k$, and $\gamma_k$, and then analyze the convergence of \eqref{eq:VRKM4ME} for two different cases, respectively: without \textit{strong quasi-monotonicity} and with \textit{strong quasi-monotonicity}.

\beforesubsec
\subsection{Applications to finite-sum inclusions.}\label{subsec:applications_to_NI}
\aftersubsec
Our second goal in this paper is to extend our method \eqref{eq:VRKM4ME} to solve the following finite-sum inclusion:
\myeq{eq:MI}{
\textrm{Find $u^{\star}\in\dom{\Psi}$ such that:}~ 0 \in \Psi{u^{\star}} :=  Gu^{\star} + Tu^{\star},
\tag{NI}
}
where $Gu := \frac{1}{n}\sum_{i=1}^nG_iu$ as in \eqref{eq:ME} and each $G_i$ is $\frac{1}{L_g}$-cocoercive for $i \in [n]$, and $T : \R^p \rightrightarrows 2^{\R^p}$ is a maximally $\nu$-co-hypomonotone operator, i.e.,  there exists $\nu \geq 0$ such that $\iprods{x - y, u - v} \geq -\nu\norms{x - y}^2$ for all $(u, x), (v, y) \in \gra{T}$, where $2^{\R^p}$ is the set of all subsets of $\R^p$ and $\gra{T}$ is the graph of $T$ (see Section~\ref{sec:AVFR4NI} for more details).
Note that the operator defining \eqref{eq:MI}  is not necessarily monotone as shown in Section~\ref{sec:AVFR4NI}.

\vspace{0.75ex}
\noindent\textbf{Special cases of \eqref{eq:MI}.}
If $T$ is monotone, then problem \eqref{eq:MI} becomes a classical monotone inclusion, which has been widely studied in the literature; see, e.g., \cite{Bauschke2011,ryu2022large}. 
This formulation \eqref{eq:MI} also covers the optimality conditions of finite-sum composite minimization and minimax optimization, as well as [mixed] variational inequality problems (VIPs), including Nash equilibrium models; see, e.g., \cite{Facchinei2003}. 
We provide three examples below to illustrate the generality of \eqref{eq:MI}.

\vspace{0.75ex}
\textit{Example 1.} Consider the following finite-sum composite minimization problem, which commonly arises in machine learning applications such as linear and logistic regression (see, e.g., \cite{friedman2001elements}):
\myeq{eq:com_cvx}{
\begin{array}{l}
{\displaystyle\min_{u \in \R^p}}\big\{ \phi(u) := f(u) + g(u) \equiv \frac{1}{n}\sum_{i=1}^nf_i(u) + g(u) \big\},
\end{array}
}
where $f_i : \R^p \to \R$ is convex and differentiable and $g$ is proper, closed, and possibly nonconvex.
The optimality condition of \eqref{eq:com_cvx} is $0 \in \nabla{f}(u^{\star}) + \partial{g}(u^{\star})$, which is a special case of \eqref{eq:MI} with $G_iu = \nabla{f_i}(u)$ and $Tu = \partial{g}(u)$, the ``abstract'' subdifferential of $g$ (see \cite{Rockafellar2004}).

\vspace{0.75ex}
\textit{Example 2.}
We consider the following finite-sum minimax optimization model, which often arises in robust and distributionally robust optimization, two-person zero-sum games, adversarial training, and generative adversarial networks; see, e.g., \cite{goodfellow2014generative,kuhn2025distributionally,madry2018towards,vonneumann1928}:
\myeq{eq:saddle_minimax}{
\begin{array}{l}
{\displaystyle\min_{x\in \R^{p_1}}\max_{y\in \R^{p_2}}} \big\{ \Lc(x, y) := f(x) + \Phi(x, y) - g(y) \equiv f(x) + \frac{1}{n}\sum_{i=1}^n\Phi_i(x, y) - g(y) \big\},
\end{array}
}
where $f$ and $g$ are proper, closed, and possibly nonconvex functions, and $\Phi$ is convex-concave and differentiable.
Let $u := [x, y]$ be a concatenation of $x$ and $y$.
Then, the optimality condition of \eqref{eq:saddle_minimax} is a special case of \eqref{eq:MI} with $G_iu := [\nabla_x\Phi_i(x, y), -\nabla_y{\Phi_i}(x, y)]$ and $Tu := [\partial{f}(x), \partial{g}(y)]$.

\vspace{0.75ex}
\textit{Example 3.}
We consider the following mixed variational inequality problem (MVIP), which has a number of applications in mechanics, traffic networks, and Nash equilibria; see, e.g., \cite{Facchinei2003}:
\myeq{eq:MVIP}{
\text{Find $u^{\star} \in \R^p$ such that:} ~\iprods{Gu^{\star}, u - u^{\star}} + g(u) - g(u^{\star}) \geq 0,
}
where $G$ is given in \eqref{eq:MI} and $g$ is a proper, closed, and convex function.
In particular, if $g = \delta_{\Xc}$, the indicator of a convex set $\Xc$, then \eqref{eq:MVIP} reduces to the classical Stampacchia variational inequality (VIP), see \cite{Facchinei2003}.
Clearly, \eqref{eq:MVIP} is a special case of \eqref{eq:MI} with $Tu = \partial{g}(u)$.

\vspace{0.75ex}
\noindent\textbf{Our approach for \eqref{eq:MI}.}
Our approach to solve \eqref{eq:MI} consists of the following two steps:
\begin{compactitem}
\item \textit{Step 1:} We reformulate \eqref{eq:MI} as \eqref{eq:ME} via a backward-forward splitting (BFS) equation, similar to \cite{davis2022variance}.
However, since $T$ is co-hypomonotone, we need to establish the relationship between \eqref{eq:MI} and this equation and verify Condition~\eqref{eq:co-coerciveness_of_G}.
\item \textit{Step 2:} We apply \eqref{eq:VRKM4ME} to find an approximate solution of the BFS equation and then recover the corresponding approximate solution of \eqref{eq:MI}. 
\end{compactitem}

\beforesubsec
\subsection{Our  contribution.}\label{subsec:contribution_related_work}
\aftersubsec
Our main theoretical contributions, together with the proposed method (VFKM), are as follows:
\begin{compactitem}
\item[(a)] 
In Section~\ref{sec:VRS_Estimator}, we introduce a new class of unbiased variance-reduced estimators $\widetilde{S}^k$ for $S^k$ in \eqref{eq:VRKM4ME}, as defined in Definition~\ref{de:ub_SG_estimator}.
Within this class, we construct two specific estimators: one based on the loopless SVRG method from \cite{SVRG,kovalev2019don} and the other leveraging the SAGA estimator from \cite{Defazio2014}.
However, any estimator satisfying Definition~\ref{de:ub_SG_estimator} can be used in our method  \eqref{eq:VRKM4ME}.

\item[(b)] 
In Section~\ref{sec:AVFR4NE}, we establish that our method  \eqref{eq:VRKM4ME} achieves  an $\BigOs{1/k^2}$ last-iterate convergence rate for $\Expk{\norms{Gx^k}^2}$ and  $\Expk{\norms{x^{k+1} - x^k}^2}$ under Assumption~\ref{as:A1}, where $k$ is the iteration counter.
We further prove $\SmallOs{1/k^2}$ convergence rates for these metrics.
Additionally, we establish an almost-sure $\SmallO{1/k^2}$  convergence rate for $\norms{Gx^k}^2$ and $\norms{x^{k+1} - x^k}^2$, along with the almost-sure convergence of the iterate sequence $\sets{x^k}$ to a solution  of \eqref{eq:ME}.

\item[(c)] 
In Subsection~\ref{subsec:SVRG_SAGA}, we present two variants of \eqref{eq:VRKM4ME} using our SVRG and SAGA estimators.
We show that both variants require $\mathcal{O}( n + n^{2/3}\epsilon^{-1} )$ evaluations of $G_i$ to achieve $\Expk{\norms{Gx^k}^2} \leq \epsilon^2$ for a given tolerance $\epsilon > 0$.
This yields an $n^{1/3}$ factor improvement over deterministic accelerated methods and a $1/\epsilon$ factor improvement over non-accelerated stochastic methods.

\item[(d)] 
In Section~\ref{sec:star_monotone_convergence},  we establish a linear convergence rate for \eqref{eq:VRKM4ME} under an additional $\sigma$-strong quasi-monotonicity on \eqref{eq:ME}.
When using our SVRG and SAGA estimators, the method attains an oracle complexity of $\BigOs{n+ \max\sets{n, n^{2/3}\kappa}\log(\epsilon^{-1})}$, improving over deterministic counterparts by a factor of $n^{1/3}$ whenever $\kappa \geq \BigOs{n^{1/3}}$, where $\kappa := \frac{L}{\sigma}$.

\item[(e)] 
We apply our method \eqref{eq:VRKM4ME} to solve a class of finite-sum inclusions \eqref{eq:MI}, which may be nonmonotone, and derive a new variant with the same theoretical convergence rates and oracle complexity bounds as in Section~\ref{sec:AVFR4NE}.
\end{compactitem}

\noindent\textbf{Comparison with prior works.}
We now briefly compare our contributions with existing works to highlight key differences.
\begin{compactitem}
\item First,  our unbiased variance-reduced estimator $\widetilde{S}^k$, including  SVRG and SAGA, is specifically designed  for $S^k$, not for $Gx^k$.
This differs from existing works such as \cite{alacaoglu2022beyond,alacaoglu2021stochastic,beznosikov2023stochastic,bot2019forward,davis2022variance,driggs2019accelerating,gorbunov2020unified,khaled2023unified}.
Note that \cite{alacaoglu2022beyond,alacaoglu2021stochastic,bot2019forward,davis2022variance} only consider specific SVRG and/or SAGA estimators to directly approximate $Gx^k$, while \cite{beznosikov2023stochastic,driggs2019accelerating,gorbunov2020unified,khaled2023unified} cover a wide class of estimators, possibly including SVRG and SAGA as special instances.
In Subsection~\ref{subsec:SG_estimators} below, we will further compare our class of estimators with these related works.

\item Second, our method \eqref{eq:VRKM4ME} is single-loop and uses simple operations, making it easier to implement than double-loop or catalyst-based algorithms in \cite{khalafi2023accelerated,yang2020catalyst}.
\item Third, our convergence rates and oracle complexity estimates rely on the metric $\mathbb{E}\big[ \norms{Gx^k}^2 \big]$ or $\norms{Gx^k}^2$, which can be viewed as a type of Fukushima's regularized gap function \cite{Fukushima1992a} for \eqref{eq:ME}.
This metric differs from recent results, e.g., \cite{alacaoglu2022beyond,alacaoglu2021stochastic}, that use a Nesterov's gap or restricted gap function from \cite{Nesterov2007a}. 
Note that for \eqref{eq:ME}, the standard Auslender gap function is degenerate.

\item Fourth, our method achieves an $\BigO{1/k}$-factor improvement over non-accelerated methods in \cite{alacaoglu2022beyond,alacaoglu2021stochastic,davis2022variance}.
Additionally, we also establish almost-sure $\SmallOs{1/k^2}$ convergence rates and show that the iterate sequence $\sets{x^k}$ converges almost surely to a solution  of \eqref{eq:ME}.
To the best of our knowledge, this is the first result on stochastic accelerated methods for root-finding problems.
\item Finally, our complexity matches the best-known results in convex optimization using SAGA or SVRG, without requiring enhancements such as restarting or nested strategies as in, e.g., \cite{Allen-Zhu2016,zhou2018stochastic}.
This complexity is better than deterministic accelerated methods by a factor $n^{1/3}$. 
\end{compactitem}

\beforesubsec
\subsection{Related work.}\label{subsec:related_work}
\aftersubsec
Both \eqref{eq:ME} and \eqref{eq:MI} are well-studied in the literature (e.g., \cite{Bauschke2011,reginaset2008,Facchinei2003,phelps2009convex,ryu2022large,ryu2016primer}). 
We focus on the most recent works relevant to our methods.

\textit{Accelerated methods.}
Nesterov's accelerated methods have been applied to solve both \eqref{eq:ME} and \eqref{eq:MI} in \cite{he2016accelerated,kolossoski2017accelerated} and \cite{attouch2019convergence}, followed by \cite{adly2021first,attouch2020convergence,attouch2022ravine,attouch2019convergence,chen2017accelerated,he2016accelerated,kim2021accelerated,mainge2021accelerated,park2022exact,tran2022connection}. 
Unlike convex optimization, extending these methods to \eqref{eq:ME} and \eqref{eq:MI} faces a fundamental challenge in constructing a suitable Lyapunov function. 
In convex optimization, the objective function serves this purpose, but it is absent in \eqref{eq:ME} and \eqref{eq:MI}. 
This necessitates a different approach for \eqref{eq:ME} and \eqref{eq:MI} (see, e.g., \cite{attouch2019convergence,mainge2021accelerated}). 
Our approach leverages similar ideas from dynamical systems as in \cite{attouch2019convergence,mainge2021accelerated}.

Alternatively, Halpern's fixed-point iteration  \cite{halpern1967fixed}  has recently been proven to achieve better convergence rates (see \cite{diakonikolas2020halpern,lieder2021convergence,sabach2017first}), matching Nesterov's schemes. 
Starting with the seminal work \cite{yoon2021accelerated}, which extended Halpern's iteration to the extragradient method, many subsequent works have exploited this idea for other methods  (e.g., \cite{cai2022accelerated,cai2022baccelerated,lee2021fast,park2022exact,tran2023extragradient,tran2022connection,tran2021halpern}). 
Recently, \cite{tran2022connection} established a connection between Nesterov's and Halpern's accelerations for solving \eqref{eq:ME} with different schemes.
Nevertheless, all of these results are primarily developed for deterministic methods.

\textit{Stochastic approximation methods.}
Stochastic methods for both \eqref{eq:ME} and \eqref{eq:MI} and their special cases have been extensively studied, e.g., in \cite{juditsky2011solving,kotsalis2022simple,pethick2023solving}. 
Some methods exploit mirror-prox and averaging techniques such as \cite{juditsky2011solving,kotsalis2022simple}, while others rely on projection or extragradient-type schemes (e.g., \cite{bohm2022two,cui2021analysis,iusem2017extragradient,kannan2019optimal,mishchenko2020revisiting,pethick2023solving,yousefian2018stochastic}). 
Many algorithms use standard Robbins-Monro stochastic approximation with large or increasing batch sizes to reduce the variance. 
Other works generalize the analysis to a broader class of algorithms (e.g., \cite{beznosikov2023stochastic,gorbunov2022stochastic,loizou2021stochastic}), covering both standard stochastic and variance reduction methods.
However, their complexity is generally worse than that of methods based on control variate techniques as in our work.

\textit{Variance reduction methods.}	
Control variate variance reduction techniques are widely used in optimization, leading to the development of various estimators, including SAGA \cite{Defazio2014}, SVRG \cite{SVRG}, SARAH \cite{nguyen2017sarah}, and Hybrid-SGD \cite{Tran-Dinh2019a}. 
These estimators have been adopted in methods for solving \eqref{eq:ME} and \eqref{eq:MI}.
For instance,  \cite{davis2016smart,davis2022variance} proposed a SAGA-type method for \eqref{eq:ME} and \eqref{eq:MI}, under ``star'' cocoercivity and strong quasi-monotonicity --- assumptions closely related to our work. 
Nevertheless, our focus is on accelerated methods that achieve improved convergence rates and complexity.
Similarly, \cite{alacaoglu2022beyond,alacaoglu2021stochastic} employed SVRG estimators for methods related to \eqref{eq:MI}, but these algorithms are non-accelerated.
Other relevant works include \cite{bot2019forward,carmon2019variance,chavdarova2019reducing,huang2022accelerated,palaniappan2016stochastic,yu2022fast}, some of which focused on minimax problems or bilinear matrix games.
More recently, variance-reduced methods based on Halpern’s fixed-point iteration have been explored (e.g., \cite{cai2023variance,cai2022stochastic}), leveraging SARAH to achieve improved complexity. 
While this approach offers advantages in certain settings, our work takes a different path by focusing on Nesterov's  acceleration combined with variance reduction techniques.
In addition, we prove new and stronger convergence results compared to these papers.
  
\beforesubsec
\subsection{Notation.}\label{subsec:back_ground}
\aftersubsec
We work with a finite-dimensional space $\R^p$ equipped with the standard inner product $\iprods{\cdot,\cdot}$ and norm $\norms{\cdot}$.
For a single-valued or multivalued operator $F$ on $\R^p$, $\dom{F} := \sets{x \in \R^p : Fx\neq\emptyset}$ denotes its domain, $\gra{F} := \sets{(x, z) \in \dom{F}\times\R^p : z \in Fx}$ denotes its graph, and $J_Fx := \sets{u \in \dom{F} : x \in u + Fu}$ denotes its resolvent.
For a symmetric matrix $\mbf{X}$, $\lambda_{\min}(\mbf{X})$ and $\lambda_{\max}(\mbf{X})$ stand for the smallest and largest eigenvalues of $\mbf{X}$, respectively.

We use $\Fc_k$ to denote the $\sigma$-algebra generated by all the randomness arising from the algorithm, including $x^0,\cdots, x^k$, up to the iteration $k$ before constructing $\widetilde{S}^k$.
$\Expsn{k}{\cdot} = \Expn{\cdot \mid \Fc_k}$ denotes the conditional expectation with respect to $\Fc_k$, and $\Expn{\cdot}$ is the total expectation. 
We also use $\BigO{\cdot}$, $\SmallO{\cdot}$, and $\Omega(\cdot)$ to characterize convergence rates and oracle complexity as usual.
 
\beforesubsec
\subsection{Paper organization.}\label{subsec:paper_organization}
\aftersubsec
The rest of this paper is organized as follows.
Section~\ref{sec:VRS_Estimator} introduces a new class of  unbiased variance-reduced estimators for $S^k$.
We also construct two instances, based on SVRG and SAGA, and prove their key properties.
Section~\ref{sec:AVFR4NE} establishes sublinear  convergence rates and oracle complexity of \eqref{eq:VRKM4ME}.
Section~\ref{sec:star_monotone_convergence} proves a linear convergence rate of \eqref{eq:VRKM4ME} under strong quasi-monotonicity.
Section~\ref{sec:AVFR4NI} derives a new variant to solve \eqref{eq:MI} and proves its convergence.
Section~\ref{sec:num_experiments} presents numerical experiments.
Technical proofs  are deferred to the appendix.
  
\section{Unbiased Variance-Reduced Estimators for $S^k$.}\label{sec:VRS_Estimator}
Most existing variance reduction methods in the literature directly construct an estimator for $Gx^k$.
In this paper, we depart from those methods and construct a class of  estimators $\widetilde{S}^k$ for the intermediate quantity $S^k = Gx^k - \gamma_kGx^{k-1}$ defined in \eqref{eq:VRKM4ME}.

\beforesubsec
\subsection{A class of stochastic variance-reduced estimators for $S^k$.}\label{subsec:SG_estimators}
\aftersubsec
We introduce the following class of unbiased variance-reduced estimators for $S^k$ in \eqref{eq:VRKM4ME}.

\begin{definition}
\label{de:ub_SG_estimator}
For given $\sets{\gamma_k}$, $\sets{x^k}$, and $S^k$ in \eqref{eq:VRKM4ME}, let $\Fc_k$ denote the $\sigma$-algebra containing all the information generated by the algorithm up to iteration $k$ before constructing $\widetilde{S}^k$.
A stochastic estimator $\widetilde{S}^k$ of $S^k$, constructed using fresh randomness conditionally on $\Fc_k$, is called a \textit{variance-reduced unbiased estimator} of $S^k$ if there exist $\rho \in (0, 1)$, $\Theta > 0$, $\hat{\Theta} \geq 0$, and a nonnegative random sequence $\sets{\Delta_k}$ such that, for all $k\geq 1$, almost surely:
\myeq{eq:ub_SG_estimator}{
\arraycolsep=0.2em
\left\{ \begin{array}{lcl}
\Expsn{k}{ \widetilde{S}^k - S^k } & = & 0,   \vspace{1ex}\\ 
\Expsn{k}{ \norms{ \widetilde{S}^k - S^k}^2  } & \leq & \Expsn{k}{\Delta_k },   \vspace{1ex}\\ 
\frac{1}{(1-\gamma_k)^2}\Expsn{k}{ \Delta_k  } &\leq & \frac{1 - \rho}{(1-\gamma_{k-1})^2}\Delta_{k-1} +  \frac{\Theta}{(1-\gamma_k)^2} U_k +  \frac{\hat{\Theta}}{(1-\gamma_{k-1})^2} U_{k-1},
\end{array}\right.
}
where $U_k := \frac{1}{n} \sum_{i=1}^n \norms{G_ix^k - G_ix^{k-1}}^2$, $x^{-2} = x^{-1} = x^0$, and $\gamma_{-1} = \gamma_0 \in [0, 1)$.
\end{definition}
%
\vspace{1ex}
\noindent\textbf{Discussion.}
Let us further discuss Definition~\ref{de:ub_SG_estimator} and compare it with related existing works.
\begin{compactenum}
\item
The first condition in \eqref{eq:ub_SG_estimator} shows that $\widetilde{S}^k$ is an unbiased estimator of $S^k$. The second line of \eqref{eq:ub_SG_estimator} provides an upper bound on the variance of $\widetilde{S}^k$, which depends on both the parameter $\gamma_k$ in $S^k$ and the random variable $\Delta_k$. The last condition in \eqref{eq:ub_SG_estimator} establishes a ``quasi-Fej\'{e}r property'' for $(1-\gamma_k)^{-2}\Delta_k$ with contraction factor $1 - \rho$ and two vanishing terms $ U_k$ and $U_{k-1}$. 
This condition involves evaluations of $G_i$ at three consecutive iterates $x^{k-2}$, $x^{k-1}$, and $x^k$.

\item
Since Definition~\ref{de:ub_SG_estimator} defines a class of estimators $\widetilde{S}^k$ for $S^k = Gx^k - \gamma_kGx^{k-1}$, rather than for $Gx^k$ alone, it differs from the definitions in \cite{beznosikov2023stochastic,driggs2019accelerating,gorbunov2020unified,khaled2023unified}, which are formulated for $Gx^k$ or $\nabla{f}(x^k)$. 
More specifically, both \cite{gorbunov2020unified} and \cite{khaled2023unified} use the same definition of variance-reduced unbiased estimators. 
While their definition has a similar pattern to \eqref{eq:ub_SG_estimator}, it involves quantities different from $S^k$ and $U_k$ in Definition~\ref{de:ub_SG_estimator}. 
In fact, our Definition~\ref{de:ub_SG_estimator} does not reduce to a special case of those in \cite{gorbunov2020unified} and \cite{khaled2023unified} when applied to finite-sum convex optimization. 
Along similar lines, \cite{driggs2019accelerating} introduced a class of estimators that covers both unbiased and biased estimators to develop accelerated methods. 
Although it shares a similar pattern with \eqref{eq:ub_SG_estimator}, the quantities involved also differ when specialized to convex optimization. 
The authors in \cite{beznosikov2023stochastic} adopted the definition in \cite{gorbunov2020unified,khaled2023unified} and extended it to VIPs, which is not the same as \eqref{eq:ub_SG_estimator} since it measures $\norms{\widetilde{G}^k - Gx^{\star}}^2$ (where $\widetilde{G}^k$ is an estimator of $Gx^k$)  instead of $\norms{\widetilde{S}^k - S^k}^2$ and uses different terms on the right-hand sides of the corresponding inequalities.
\end{compactenum}

In the following subsections, we will construct two estimators that satisfy Definition~\ref{de:ub_SG_estimator} using SVRG \cite{SVRG,kovalev2019don} and SAGA \cite{Defazio2014}, respectively.
However, any other estimator satisfying Definition~\ref{de:ub_SG_estimator} can be used in our method \eqref{eq:VRKM4ME}.

\beforesubsec
\subsection{The loopless SVRG estimator for $S^k$.}\label{subsec:SVRG_estimator}
\aftersubsec
Consider an i.i.d. mini-batch $\Bc_k \subseteq [n]$ with a fixed size $b := \vert \Bc_k\vert$ (sampling with replacement).
Let $G_{\Bc_k}x := \frac{1}{b}\sum_{i\in\Bc_k}G_ix$ be a standard mini-batch estimator of $Gx$ for a given $x \in \dom{G}$.
Given $\sets{x^k}$ from \eqref{eq:VRKM4ME}, we define 
\myeq{eq:SVRG_estimator}{
\begin{array}{lcl}
\widetilde{S}^k & := &   (1-\gamma_k)\big( Gw^k  - G_{\Bc_k}w^k \big) + G_{\Bc_k}x^k - \gamma_k G_{\Bc_k}x^{k-1}, 
\end{array}
}
where, for all $k\geq 1$, the reference or snapshot point $w^k$ is updated as follows:
\myeq{eq:SG_estimator_w}{
\arraycolsep=0.2em
w^k := \left\{\begin{array}{ll}
x^{k-1} &\text{with probability $\mbf{p}$}  \vspace{1ex}\\ 
w^{k-1} & \text{with probability $1 - \mbf{p}$}.
\end{array}\right.
}
Here, $w^0 := x^0$, and the probability $\mbf{p} \in (0, 1)$ is typically chosen to be small.
Similar to  \cite{kovalev2019don}, we call $\widetilde{S}^k$ in \eqref{eq:SVRG_estimator} a loopless SVRG estimator, but it is still different from  those in \cite{davis2022variance,SVRG,kovalev2019don}.

Next, we state the following key property, whose proof is given in Appendix~\ref{apdx:le:SVRG_estimator}.

\begin{lemma}\label{le:SVRG_estimator}
Let $\sets{x^k}$ and $S^k = Gx^k - \gamma_k Gx^{k-1}$ be defined in \eqref{eq:VRKM4ME}, $\widetilde{S}^k$ be constructed by the loopless SVRG estimator \eqref{eq:SVRG_estimator}, and $w^k$ be given in \eqref{eq:SG_estimator_w}.
Define
\myeq{eq:SVRG_Deltak}{
\arraycolsep=0.2em
\begin{array}{lcl}
\Delta_k &:= &  \frac{1}{nb} \sum_{i=1}^n \norms{ G_ix^k - \gamma_k G_ix^{k-1} - (1-\gamma_k) G_{i}w^k }^2.
\end{array}
}
Then, for all $k\geq 1$, we have 
\myeq{eq:SVRG_vr_properties}{
\arraycolsep=0.2em
\left\{\begin{array}{lcl}
\Expsn{k}{\widetilde{S}^k - S^k} & = & 0,   \vspace{1ex}\\ 
\Expsn{k}{ \norms{\widetilde{S}^k - S^k }^2 } & \leq &  \Expsn{k}{ \Delta_k }, \vspace{1ex}\\
\frac{1}{(1-\gamma_k)^2} \Expsn{k}{  \Delta_k } & \leq & \big( 1 - \frac{\mbf{p}}{2} \big) \frac{1}{(1-\gamma_{k-1})^2}  \Delta_{k-1}   +  \frac{4-6\mbf{p} + 3\mbf{p}^2}{ b \mbf{p}(1-\gamma_k)^2 }  U_k +  \frac{2(2-3\mbf{p} + \mbf{p}^2) \gamma_{k-1}^2}{b \mbf{p} (1-\gamma_{k-1})^2 }  U_{k-1}.
\end{array}\right.
\hspace{-3.5ex}
}
Consequently, the estimator $\widetilde{S}^k$ satisfies Definition~\ref{de:ub_SG_estimator} with  $\rho := \frac{\mbf{p}}{2} \in (0, 1)$, $\Theta := \frac{ 4 - 6\mbf{p} + 3\mbf{p}^2 }{b\mbf{p}}$, $\hat{\Theta} := \frac{2(2 - 3\mbf{p} + \mbf{p}^2)}{b \mbf{p} }$, and $\Delta_k$ in \eqref{eq:SVRG_Deltak}.
\end{lemma}

\beforesubsec
\subsection{The SAGA estimator for $S^k$.}\label{subsec:SAGA_estimator}
\aftersubsec
At each iteration $k \geq 1$,  we first sample an i.i.d. mini-batch $\Bc_k  \subseteq [n] $ of a fixed size $b$  with replacement.
Next, we sample a refresh mini-batch $\hat{B}_k\subseteq [n]$ of size $b$ without replacement and independent of $\Bc_k$.
Let $G_{\Bc_k}x := \frac{1}{b}\sum_{i\in\Bc_k}G_ix$ be a mini-batch estimator of $Gx$ for a given $x \in \dom{G}$ and $\widehat{G}^k_{\Bc_k} := \frac{1}{b}\sum_{i \in \Bc_k}\widehat{G}^k_i$ for given $\sets{\widehat{G}^k_i}_{i=1}^n$.
Given $\sets{x^k}$ and $S^k = Gx^k - \gamma_kGx^{k-1}$ in \eqref{eq:VRKM4ME}, we define the following SAGA estimator for $S^k$:  
\myeq{eq:SAGA_estimator}{
\arraycolsep=0.2em
\begin{array}{lcl}
\widetilde{S}^k & := & \frac{1-\gamma_k}{n}\sum_{i=1}^n\widehat{G}^k_i  +  \big[ G_{\Bc_k}x^k - \gamma_kG_{\Bc_k}x^{k-1} - (1-\gamma_k)\widehat{G}^k_{\Bc_k} \big], 
\end{array}
}
where $\widehat{G}^k_i$ is initialized at $\widehat{G}^0_i := G_ix^0$ for $i \in [n]$, and then, for $k\geq 1$, is updated by
\myeq{eq:SAGA_ref_points}{
\arraycolsep=0.3em
\widehat{G}^k_i := \left\{\begin{array}{lcl}
\widehat{G}^{k-1}_i & \text{if}~i \notin \hat{\Bc}_k,   \vspace{1ex}\\ 
G_ix^{k-1} &\text{if}~ i \in \hat{\Bc}_k.
\end{array}\right.
}
For this SAGA estimator, we need to store all $n$ components $\widehat{G}_i^k$ computed so far for $i \in [n]$ in a table $\Tc_k := [\widehat{G}^k_1, \widehat{G}^k_2, \cdots, \widehat{G}^k_n]$.
At each iteration $k$, we will evaluate $G_ix^{k-1}$ for $i \in \hat{\Bc}_k$, and update the table $\Tc_k$ for all $i \in \hat{\Bc}_k$.
Hence, the SAGA estimator requires substantial memory to store $\Tc_k$ if $n$ and $p$ are both large.

The following lemma shows that the estimator constructed in \eqref{eq:SAGA_estimator} satisfies Definition~\ref{de:ub_SG_estimator}. 
Its proof is deferred to Appendix~\ref{apdx:le:SAGA_estimator}.

\begin{lemma}\label{le:SAGA_estimator}
Let $\sets{x^k}$ and $S^k = Gx^k - \gamma_k Gx^{k-1}$ be given in \eqref{eq:VRKM4ME} and $\widetilde{S}^k$ be constructed by the SAGA estimator \eqref{eq:SAGA_estimator}.
We define
\myeq{eq:SAGA_Deltak}{
\arraycolsep=0.2em
\begin{array}{lcl}
\Delta_k &:= &  \frac{1}{nb} \sum_{i=1}^n  \norms{ G_ix^k - \gamma_k G_ix^{k-1} - (1-\gamma_k) \widehat{G}_i^k }^2.
\end{array}
}
If we denote $\Theta :=  \frac{2(n - b)(2n + b ) + b^2}{ n b^2 } $ and $\hat{\Theta} := \frac{2(n-b)(2n + b ) }{ nb^2  } $, then for all $k\geq 1$, we have
\myeq{eq:SAGA_vr_properties}{
\arraycolsep=0.2em
\left\{\begin{array}{lcl}
\Expsn{k}{\widetilde{S}^k - S^k} & = & 0,   \vspace{1ex}\\ 
\Expsn{k}{ \norms{\widetilde{S}^k - S^k }^2 } & \leq &   \Expsn{k}{ \Delta_k },   \vspace{1ex}\\ 
\frac{1}{(1-\gamma_k)^2} \Expsn{k}{ \Delta_k } & \leq &  \big(1 -  \frac{b}{2n} \big) \frac{1}{(1-\gamma_{k-1})^2}  \Delta_{k-1}  +   \frac{\Theta}{(1-\gamma_k)^2} U_k   +  \frac{\hat{\Theta}}{(1-\gamma_{k-1})^2 } U_{k-1}.
\end{array}\right.
}
Consequently, $\widetilde{S}^k$ satisfies Definition~\ref{de:ub_SG_estimator} with $\rho := \frac{b}{2n} \in (0, 1)$, $\Theta$ and $\hat{\Theta}$ given above, and $\Delta_k$ in \eqref{eq:SAGA_Deltak}.
\end{lemma}

\beforesec
\section{Convergence  and Complexity Analysis of \eqref{eq:VRKM4ME}.}\label{sec:AVFR4NE}
\aftersec
In this section, we first establish sublinear convergence rates of \eqref{eq:VRKM4ME}.
Then, we prove its almost sure convergence rates and the almost sure convergence of $\sets{x^k}$.
Finally, we estimate the complexity for two variants of \eqref{eq:VRKM4ME} using our SVRG and SAGA estimators, respectively.

\beforesubsec
\subsection{Lyapunov function and descent lemma.}\label{subsec:one_iteration_analysis}
\aftersubsec
The following lemma serves as a key step in proving the convergence of  \eqref{eq:VRKM4ME}, whose proof is given in Appendix~\ref{apdx:le:AVFR4NE_key_estimate}.

\begin{lemma}\label{le:AVFR4NE_key_estimate}
Suppose that Condition~\eqref{eq:co-coerciveness_of_G} holds for \eqref{eq:ME}.
Let $\sets{x^k}$ be generated by \eqref{eq:VRKM4ME} using a stochastic estimator $\widetilde{S}^k$ of $S^k$ satisfying Definition~\ref{de:ub_SG_estimator} and the following parameters:
\myeq{eq:AVFR4NE_para_update}{
\arraycolsep=0.2em
\begin{array}{ll}
\theta_k := \frac{t_k - r - \mu}{t_{k+1}}, \quad \gamma_k := \frac{t_k - r - \mu}{t_k - \mu}, \quad \text{and} \quad \eta_k := \frac{2\beta(t_k - \mu)}{t_{k+1}},
\end{array}
}
where $\beta > 0$, $r > 0$, and $\mu > 0$ are given and $t_k \geq r + \mu$ for all $k\geq 0$.

Consider the following \textbf{Lyapunov function}:
\myeq{eq:AVFR4NE_Lfunc}{
\hspace{-4ex}
\arraycolsep=0.2em
\begin{array}{lcl}
\Ec_k & := & 4r\beta(t_k-\mu) \big[ \iprods{Gx^k, x^k - x^{\star}} - \beta\norms{ Gx^k }^2 \big] + \norms{r(x^k - x^{\star}) + t_{k+1}(x^{k+1} - x^k ) }^2   \vspace{1ex}\\ 
&& + {~} \mu r \norms{x^k - x^{\star}}^2  +  \frac{4 \beta^2 \hat{\Theta}(t_k - \mu)^2 }{\rho }  U_k   +  \frac{4  \beta^2(1-\rho) (t_k  - \mu)^2 }{ \rho } \Delta_k,
\end{array}
\hspace{-6ex}
}
where $U_k := \frac{1}{n}  \sum_{i=1}^n  \norms{G_ix^k - G_ix^{k-1}}^2$.

Then, for all $k\geq 1$, we have
\myeq{eq:AVFR4NE_key_estimate1}{
\hspace{-0.0ex}
\arraycolsep=0.2em
\begin{array}{lcl}
 \Ec_{k-1}  - \Expsn{k}{ \Ec_{k} } & \geq &  \Lambda_k(t_k - \mu)^2  U_k  +  \mu (2t_k - r - \mu)  \norms{x^k - x^{k-1}}^2 \vspace{1ex}\\
&& + {~} 4r\beta(t_{k-1} - t_k + r)\big[ \iprods{Gx^{k-1}, x^{k-1} - x^{\star}} - \beta \norms{Gx^{k-1}}^2 \big],
\end{array}
\hspace{-3.0ex}
}
where $\Lambda_k := 4\beta   \left[  \frac{(t_k - r - \mu)}{(t_k - \mu)}  \big(\frac{1}{L} - \beta\big) - \frac{\beta (\Theta + \hat{\Theta}) }{\rho } \right]$.
\end{lemma}

\beforesubsec
\subsection{Sublinear convergence rates of \eqref{eq:VRKM4ME}.}\label{subsec:AVROG_with_SVRG}
\aftersubsec
For simplicity of presentation, we define  the following constants:  
\myeq{eq:AVFR4NE_constants}{
\arraycolsep=0.1em
\left\{\begin{array}{lcllcl}
\bar{\beta} & := &  \min\left\{ \frac{\rho}{ [\rho + (r+1)(\Theta + \hat{\Theta})] L}, \ \frac{1}{(r+1)L} \right\},  &  \Lambda  & := &   \frac{4\beta \left( \rho - [\rho  + (r+1)(\Theta + \hat{\Theta})] L\beta \right) }{L\rho(r+1)},   \vspace{1ex}\\  
C_0 & := & r (1 + 3r + 8rL^2\beta^2),   &  C_1 & := &  \frac{4r^2(r-1)L^2\beta^2}{r+2} + \big(\frac{2}{r} +  \psi \big) C_0, \vspace{1ex}\\  
C_2 & := & 4r^2L^2\beta^2  + \big[ \psi +  \frac{2(r+2)^2}{r}\big] C_0 + \frac{4(r+2)^2C_1}{r(r-2)},     \qquad  & C_3 &:= & \frac{(r+2)^2(C_1 + C_2)}{2r^2 },
\end{array}\right.
}
where $\psi := \frac{4\beta^2(\Theta + \hat{\Theta})}{\rho\Lambda}$.
For a given $r > 2$, we have $C_s = \BigO{1}$ for $s \in \sets{0,1,2,3}$.

Note that we can always fix $r := 3$ in our method without affecting the convergence results. 
However, we keep $r$ as a free parameter to examine how it influences the practical performance of our method. 
For the constants in \eqref{eq:AVFR4NE_constants}, only $\bar{\beta}$ is used in our method \eqref{eq:VRKM4ME}. 
The other constants $\psi$, $\Lambda$, and $C_i$ for $i=0,1,2,3$ are used only in the theoretical convergence analysis in Theorem~\ref{th:AVFR4NE_convergence}.

Now, we establish sublinear convergence rates and iteration complexity of \eqref{eq:VRKM4ME}.

\begin{theorem}\label{th:AVFR4NE_convergence}
Suppose that $G$ in \eqref{eq:ME} satisfies Assumption~\ref{as:A1}.
Let $\sets{x^k}$ be generated by our method \eqref{eq:VRKM4ME} to solve \eqref{eq:ME} using a stochastic estimator $\widetilde{S}^k$ of $S^k$ satisfying Definition~\ref{de:ub_SG_estimator} with $\rho \in (0, 1)$ and $\Delta_0 = 0$ almost surely.
For $\bar{\beta}$ given in  \eqref{eq:AVFR4NE_constants}, assume that  $r > 2$ and $\beta \in (0, \bar{\beta})$ are given, and $\theta_k$, $\gamma_k$, and $\eta_k$ are updated by
\myeq{eq:AVFR4NE_para_update2}{
\theta_k := \frac{k}{k + r + 2}, \quad \gamma_k := \frac{k}{k + r}, \quad \text{and} \quad \eta_k := \frac{2\beta(k + r)}{k + r + 2}.
}
Let $\Lambda$ and $C_s$ for $s \in \sets{0,1,2,3}$ be given in \eqref{eq:AVFR4NE_constants}.
Then, the following statements hold.
\begin{compactitem}
\item[$\mathrm{(a)}$] \textbf{Summability bounds.}
The following bounds hold:
\myeq{eq:AVFR4NE_convergence}{
\arraycolsep=0.2em
\begin{array}{lcl}
\sum_{k=0}^{\infty} (2k + r + 3) \Expk{ \norms{x^{k+1} - x^k}^2 } & \leq &  C_0  \norms{x^0 - x^{\star}}^2, \vspace{1ex}\\ 
\sum_{k=0}^{\infty} (2k + r + 3)\Expk{\norms{Gx^k}^2 }  & \leq &  \frac{(r+2)^2}{2\beta^2r^2}\big(C_0 + \frac{2C_1}{r-2}\big) \norms{x^0 - x^{\star}}^2.
\end{array}
}
\item[$\mathrm{(b)}$]\textbf{Big-O convergence rates.}
For any $k \geq 0$,  we have
\myeq{eq:AVFR4NE_convergence2b}{
\hspace{-0.1ex}
\arraycolsep=0.2em
\begin{array}{lll}
\Expk{\norms{x^{k+1} - x^k}^2 } & \leq  & \dfrac{ C_2   \norms{x^0 - x^{\star}}^2 }{(k+r+2)^2},   \vspace{1ex}\\ 
\Expk{\norms{Gx^k}^2 }  & \leq & \dfrac{C_3   \norms{x^0 - x^{\star}}^2 }{\beta^2 (k+r-1)(k+r + 2) }.
\end{array}
\hspace{-4ex}
}
\item[$\mathrm{(c)}$]\textbf{Small-o convergence rates.}
The following $\SmallOs{1/k^2}$ rates also hold:
\myeq{eq:AVFR4NE_convergence2c}{
\arraycolsep=0.2em
\lim_{k\to\infty} k^2 \Expk{\norms{x^{k+1} - x^k}^2 } = 0 \quad \text{and} \quad \lim_{k\to\infty} k^2 \Expk{\norms{Gx^k}^2 } = 0.
}
\end{compactitem}
For a given $\epsilon > 0$, \eqref{eq:VRKM4ME} requires $\BigOs{\frac{1}{\epsilon}}$ iterations to achieve $\Expn{\norms{Gx^k}^2} \leq \epsilon^2$.
\end{theorem}

\noindent\textbf{Discussion.} Before proving Theorem \ref{th:AVFR4NE_convergence}, we make the following remarks on this result.
\begin{compactenum}
\item 
Theorem~\ref{th:AVFR4NE_convergence} establishes both $\BigO{1/k^2}$ and $\SmallO{1/k^2}$ convergence rates for $\Expk{\norms{Gx^k}^2 }$. 
These rates improve upon non-accelerated methods, e.g., \cite{alacaoglu2022beyond,alacaoglu2021stochastic,davis2022variance}, by a factor of $\BigOs{1/k}$, but they require Assumption~\ref{as:A1}, which may be stronger than the assumptions in \cite{alacaoglu2022beyond,alacaoglu2021stochastic}.

\item 
Our method \eqref{eq:VRKM4ME} uses three parameters $\theta_k$, $\gamma_k$, and $\eta_k$. 
Both $\gamma_k$ and $\theta_k$ are given explicitly and do not depend on any parameter of $G$ or of \eqref{eq:ub_SG_estimator} in  Definition~\ref{de:ub_SG_estimator}. 
Only $\eta_k$ depends on $\beta$, and thus $\bar{\beta}$ defined in \eqref{eq:AVFR4NE_constants}. 
Since our method covers a class of estimators in Definition~\ref{de:ub_SG_estimator}, it is natural that $\bar{\beta}$ depends on $L$ in Assumption~\ref{as:A1}, as well as $\rho$, $\Theta$, and $\hat{\Theta}$ in Definition~\ref{de:ub_SG_estimator}. 
However, once the estimator $\widetilde{S}^k$ is specified, $\rho$, $\Theta$, and $\hat{\Theta}$ are explicit, and thus $\bar{\beta}$ depends only on $L$.

\item 
The summability bounds in \eqref{eq:AVFR4NE_convergence} are well-known for deterministic Nesterov acceleration methods in convex optimization and monotone operators; see, e.g., \cite{attouch2016rate}. 
Since $\Delta{x}^k := x^{k+1} - x^k$ can be viewed as the ``velocity" associated with the ``state"  $x^k$, the first summability bound in \eqref{eq:AVFR4NE_convergence} describes the evolution of the corresponding ``kinetic energy" over the iterations and shows that the ``total weighted kinetic energy" is finite. 
Likewise, viewing $\norms{Gx^k}^2$ as a ``force magnitude", the second summability bound in \eqref{eq:AVFR4NE_convergence} tracks its evolution across iterations and shows that its total weighted sum is finite.

\item 
The limits in \eqref{eq:AVFR4NE_convergence2c} indicate a faster $\SmallOs{1/k^2}$ convergence rate of $\Expk{\norms{x^{k+1} - x^k}^2 }$ and $\Expk{\norms{Gx^k}^2 }$, respectively, in the regime $k \gg 1$.
\end{compactenum}
While the $\BigOs{1/k^2}$ convergence rates of the form \eqref{eq:AVFR4NE_convergence2b} have been widely presented in variance reduction methods for convex optimization, the $\SmallOs{1/k^2}$ convergence rates remain largely elusive in stochastic variance reduction methods for both optimization and root-finding problems.

\vspace{0.75ex}
\proof{\textbf{Proof of Theorem~\ref{th:AVFR4NE_convergence}.}}
For simplicity of our analysis, we choose $\mu := 1$ and $t_k := k + r + \mu = k + r + 1$.
Since $t_k = k + r + 1$ and $\mu = 1$, from \eqref{eq:AVFR4NE_para_update}, we get
\myeqn{
\arraycolsep=0.2em
\begin{array}{lcl}
\theta_k := \frac{t_k - r - \mu}{t_{k+1}} = \frac{k}{k + r + 2}, 
\quad \gamma_k := \frac{t_k - r - \mu}{t_k - \mu} = \frac{k}{k + r}, \quad \text{and} \quad \eta_k := \frac{2\beta(t_k - \mu)}{t_{k+1}} = \frac{2\beta(k + r)}{k + r + 2},
\end{array}
}
which are the updates in \eqref{eq:AVFR4NE_para_update2}.

Now, for clarity, we divide the proof into several steps as follows.

\vspace{0.75ex}
\noindent\textbf{Step 1.}~\textit{The first summability result in \eqref{eq:AVFR4NE_convergence}.}
Since $\beta \in (0, \bar{\beta})$, $t_k = k + r + 1$, and $\mu = 1$, for $k \geq 1$, the constant $\Lambda_k$ in Lemma~\ref{le:AVFR4NE_key_estimate} is lower bounded by 
\myeqn{
\arraycolsep=0.2em
\begin{array}{lcl}
\Lambda_k & := & 4\beta   \Big[  \frac{(t_k - r - \mu)}{(t_k - \mu)} \big(\frac{1}{L} - \beta\big) - \frac{ \beta (\Theta + \hat{\Theta}) }{\rho } \Big] = 4\beta \big[ \frac{k}{k+r} \big(\frac{1}{L} - \beta\big)  - \frac{\beta (\Theta + \hat{\Theta} ) }{\rho} \big]   \vspace{1ex}\\ 
& \geq &  \frac{4\beta}{r+1}  \left[ \frac{1}{L} -  \frac{[\rho  + (r+1)(\Theta + \hat{\Theta})] \beta}{\rho}   \right] =    \frac{4\beta \left( \rho  - [\rho  + (r+1) (\Theta + \hat{\Theta})] L\beta \right) }{L\rho(r+1)} \equiv \Lambda,
\end{array}
}
where $\Lambda$ is defined by \eqref{eq:AVFR4NE_constants}.
Clearly, since $\beta \in (0, \bar{\beta})$, we have $\Lambda_k \geq \Lambda > 0$ for $k \geq 1$.

Since $\beta \leq \frac{1}{L}$, by Condition \eqref{eq:co-coerciveness_of_G}, we have $\iprods{Gx^{k-1}, x^{k-1} - x^{\star}} - \beta \norms{Gx^{k-1}}^2 \geq 0$.
Substituting this inequality, the lower bound $\Lambda$ of $\Lambda_k$, and $t_k - \mu = k + r$ into \eqref{eq:AVFR4NE_key_estimate1}, and noting that $t_{k-1} - t_k + r = r - 1 > 0$ (since $r > 2$), we obtain
\myeq{eq:AVFR4NE_key_inequality100}{
\arraycolsep=0.2em
\begin{array}{lcl}
\Ec_{k-1} - \Expsn{k}{ \Ec_k } & \geq &    \Lambda (k+r)^2 U_k +  (2k+r + 1) \norms{x^k - x^{k-1}}^2, \quad \forall k \geq 1.
\end{array}
}
Taking the total expectation on both sides of this inequality, we get
\myeqn{
\arraycolsep=0.2em
\begin{array}{lcl}
\Expn{\Ec_{k-1}} - \Expn{ \Ec_k } & \geq &    \Lambda (k+r)^2 \Expn{ U_k } +  (2k+r + 1) \Expn{ \norms{x^k - x^{k-1}}^2 }.
\end{array}
}
Summing this inequality  from $k=1$ to $K$, and using \eqref{eq:Ec0_lower_bound} from Lemma~\ref{le:lower_bounding_Ec_0} (from Appendix~\ref{apdx:subsec:tech_lemmas}), and $\Expk{\Ec_K} \geq 0$, then letting $K\to\infty$, we can show that
\myeq{eq:AVFR4NE_key_inequality101}{
\arraycolsep=0.2em
\begin{array}{lcl}
\sum_{k=1}^{\infty} (2k+r + 1) \Expn{ \norms{x^k - x^{k-1}}^2 } &\leq & \Ec_0 \overset{\tiny\eqref{eq:Ec0_lower_bound}}{\leq} C_0  \norms{x^0 - x^{\star}}^2, \vspace{1ex}\\
\sum_{k=1}^{\infty}(k+r)^2 \Expn{ U_k } &\leq & \frac{1}{\Lambda}\Ec_0 \overset{\tiny\eqref{eq:Ec0_lower_bound}}{\leq}  \frac{C_0}{\Lambda}  \norms{x^0 - x^{\star}}^2.
\end{array}
}
The first line of \eqref{eq:AVFR4NE_key_inequality101} is exactly the first  summability bound in \eqref{eq:AVFR4NE_convergence}.

\vspace{0.75ex}
\noindent\textbf{Step 2.}~\textit{The second summability result in \eqref{eq:AVFR4NE_convergence}.}
For $W_k$ defined by \eqref{eq:AVFR4NE_Wk_def} and $v^k := x^{k+1} - x^k + \eta_kGx^k$, taking the total expectation of \eqref{eq:AVFR4NE_Wk_bound} in Lemma~\ref{le:AVFR4NE_key_bound_of_Wk}, we get
\myeqn{
\arraycolsep=0.2em
\begin{array}{lcl}
\Expn{W_k} & \leq & \Expn{ W_{k-1}} - (r-2)(k+r+1) \Expn{ \norms{v^{k-1}}^2 }  +  \frac{2(2k + r +1)}{r} \Expn{ \norms{x^k - x^{k-1}}^2 } + \frac{4\beta^2(\Theta + \hat{\Theta} ) (k + r)^2 }{\rho} \Expn{ U_k }.
\end{array}
}
Utilizing \eqref{eq:AVFR4NE_key_inequality101}, we can apply Lemma~\ref{le:A1_descent} to this inequality to prove that
\myeq{eq:AVFR4NE_proof15a}{
\arraycolsep=0.2em
\begin{array}{llcl}
& \lim_{k\to\infty}\Expn{ W_k } \ \textrm{exists},  && \vspace{1ex}\\
&\sum_{k=1}^{\infty}(k+r+1) \Expn{ \norms{v^{k-1} }^2 } & \leq & W_0 + \big[ \frac{2}{r} +  \frac{4\beta^2(\Theta + \hat{\Theta} )}{\rho\Lambda} \big]C_0 \norms{x^0 - x^{\star}}^2 \leq  C_1 \norms{x^0 - x^{\star}}^2, \vspace{1ex}\\
& \Expn{W_k} &\leq & C_1 \norms{x^0 - x^{\star}}^2,
\end{array}
}
where we have used $W_0  =  (r - 1)(r + 2)\norms{ x^1 - x^0 + \eta_0Gx^0 }^2 \leq (r - 1)(r + 2)\eta_0^2\norms{Gx^0}^2 \leq \frac{4L^2\beta^2(r-1)}{r+2}\norms{x^0 - x^{\star}}^2$ and $C_1 := \frac{4L^2\beta^2(r-1)}{r+2} + \frac{2C_0}{r} +  \frac{4\beta^2(\Theta + \hat{\Theta} )C_0}{\rho\Lambda}$ as in \eqref{eq:AVFR4NE_constants}.

Now, by the choice of $\eta_k = \frac{2\beta(k + r)}{k+r + 2}$ from \eqref{eq:AVFR4NE_para_update2}, we get $\eta_k \geq \frac{2\beta r}{r+2}$.
Applying this bound and Young's inequality, and using $v^k  = x^{k+1} - x^k + \eta_kGx^k$, we have 
\myeq{eq:AVFR4NE_proof15}{
\arraycolsep=0.2em
\begin{array}{lcl}
\frac{2\beta^2r^2}{(r+2)^2} \norms{Gx^k}^2 & \leq & \frac{\eta_k^2}{2}\norms{Gx^k}^2  \leq  \norms{v^k}^2 + \norms{x^{k+1} - x^k}^2.
\end{array}
}
Combining \eqref{eq:AVFR4NE_proof15}, the second bound of \eqref{eq:AVFR4NE_proof15a}, and the first bound of  \eqref{eq:AVFR4NE_convergence}, we obtain
\myeqn{
\arraycolsep=0.2em
\begin{array}{lcl}
\frac{2\beta^2r^2}{(r+2)^2} \sum_{k=0}^K (2k+r+3) \Expn{ \norms{Gx^k}^2 } & \leq &   2 \sum_{k=0}^K (k+r+2) \Expn{ \norms{v^k}^2 }    +   \sum_{k=0}^K (2k + r + 3) \Expn{ \norms{x^{k+1} - x^k}^2 }   \vspace{1ex}\\ 
&\leq & \Big(C_0 + \frac{2C_1}{r-2} \Big)  \norms{x^0 - x^{\star}}^2.
\end{array}
}
Taking the limit of this inequality as $K \to \infty$, we prove the second line of  \eqref{eq:AVFR4NE_convergence}.

\vspace{0.75ex}
\noindent\textbf{Step 3.}~\textit{The Big-O and small-o rates of $\Expk{ \norms{x^{k+1} - x^k + \eta_kGx^k}^2 }$.}
Since $(k+r-1)(k+r+2)\norms{x^{k+1} - x^k + \eta_kGx^k}^2 \leq W_k$, the last line of \eqref{eq:AVFR4NE_proof15a} implies that 
\myeq{eq:AVFR4NE_proof15h}{ 
\arraycolsep=0.2em
\begin{array}{lcl}
\Expk{\norms{x^{k+1} - x^k + \eta_kGx^k}^2 } \leq \frac{C_1}{(k+r-1)(k+r+2)}  \norms{x^0 - x^{\star}}^2. 
\end{array}
}
From \eqref{eq:AVFR4NE_Deltak_bound} of Lemma~\ref{le:AVFR4NE_key_bound_of_Wk} and the second line of \eqref{eq:AVFR4NE_key_inequality101}, we can show that 
\myeqn{ 
\arraycolsep=0.2em
\begin{array}{lcl}
\sum_{k=1}^{\infty}(k+r)^2\Expn{\Delta_k} \leq \frac{(\Theta + \hat{\Theta})C_0}{\rho\Lambda}\norms{x^0 - x^{\star}}^2.
\end{array}
}
This summability bound shows that $\lim_{k\to\infty}(k+r)^2\Expn{\Delta_k} = 0$.
Moreover, from the second line of \eqref{eq:AVFR4NE_key_inequality101}, we also have $\lim_{k\to\infty}(k+r)^2\Expn{U_k} = 0$.
Combining these limits and the definition \eqref{eq:AVFR4NE_Wk_def} of $W_k$ and the existence of $\lim_{k\to\infty}\Expn{ W_k }$ from \eqref{eq:AVFR4NE_proof15a}, we have
\myeqn{ 
\arraycolsep=0.2em
\begin{array}{lcl}
{\displaystyle\lim_{k\to\infty}}(k+r-1)(k+r+2) \Expn{ \norms{v^k}^2 } \ \textrm{exists}.
\end{array}
}
Applying Lemma~\ref{le:A2_sum} to this limit and using the second line of \eqref{eq:AVFR4NE_proof15a}, we conclude that 
\myeqn{ 
\arraycolsep=0.2em
\begin{array}{lcl}
{\displaystyle\lim_{k\to\infty}}(k+r-1)(k+r+2) \Expk{ \norms{x^{k+1} - x^k + \eta_kGx^k}^2 } = 0.
\end{array}
}
This also implies that $\lim_{k\to\infty} k^2 \Expk{ \norms{x^{k+1} - x^k + \eta_kGx^k}^2 } = 0$.

\vspace{0.75ex}
\noindent\textbf{Step 4.}~\textit{The Big-O and small-o rates of $\Expk{ \norms{x^{k+1} - x^k}^2 }$.}
For $Z_k$ defined by \eqref{eq:AVFR4NE_Zk_def}, taking the total expectation on both sides of \eqref{eq:AVFR4NE_Zk_bound} from Lemma~\ref{le:AVFR4NE_auxiliary_lemma2}, we get
\myeq{eq:AVFR4NE_proof15f}{ 
\arraycolsep=0.2em
\begin{array}{lcl}
\Expn{Z_k} & \leq & \Expn{ Z_{k-1} } + 4r\beta^2(2k + r + 3) \Expn{ \norms{Gx^k}^2 } + \frac{4\beta^2(\Theta + \hat{\Theta} ) (k + r)^2 }{\rho} \Expn{ U_k }.
\end{array}
}
Summing this inequality from $k =1$ to $K$, and using $\norms{x^1 - x^0}^2 \leq \frac{4r^2L^2\beta^2}{(r+2)^2}\norms{x^0 - x^{\star}}^2$, the second line of \eqref{eq:AVFR4NE_convergence}, and the second line of \eqref{eq:AVFR4NE_key_inequality101}, we get
\myeqn{ 
\arraycolsep=0.2em
\begin{array}{lcl}
\Expn{Z_K} & \leq & \big[ 4r^2L^2\beta^2 + \frac{2(r+2)^2}{r}\big(C_0 + \frac{2C_1}{r-2}\big) + \frac{4\beta^2(\Theta+\hat{\Theta})C_0}{\rho\Lambda}\big]\norms{x^0 - x^{\star}}^2 = C_2\norms{x^0 - x^{\star}}^2,
\end{array}
}
where $C_2 := 4r^2L^2\beta^2 + \frac{2(r+2)^2}{r}\big(C_0 + \frac{2C_1}{r-2}\big) + \frac{4\beta^2(\Theta+\hat{\Theta})C_0}{\rho\Lambda}$ as defined in \eqref{eq:AVFR4NE_constants}.
Since $(k+r+2)^2\norms{x^{k+1} - x^k}^2 \leq Z_k$ due to \eqref{eq:AVFR4NE_Zk_def}, the last inequality implies that $\Expn{\norms{x^{k+1} - x^k}^2} \leq \frac{C_2 \norms{x^0 - x^{\star}}^2 }{(k+r + 2)^2}$, proving the first bound of \eqref{eq:AVFR4NE_convergence2b}.

Utilizing the second summability bound of \eqref{eq:AVFR4NE_convergence} and the second line of \eqref{eq:AVFR4NE_key_inequality101}, we can apply Lemma~\ref{le:A1_descent} to \eqref{eq:AVFR4NE_proof15f} to conclude that $\lim_{k\to\infty}\Expn{Z_k}$ exists.
Since $\lim_{k\to\infty}(k+r)^2\Expn{\Delta_k} = 0$ and $\lim_{k\to\infty}(k+r)^2\Expn{U_k} = 0$ as proven above, using these limits and the definition \eqref{eq:AVFR4NE_Zk_def} of $Z_k$, we conclude that $\lim_{k\to\infty}(k+r+2)^2\Expn{\norms{x^{k+1} - x^k}^2}$ exists.
Applying Lemma~\ref{le:A2_sum} to this limit and using the first line of \eqref{eq:AVFR4NE_convergence}, we obtain $\lim_{k\to\infty} k^2 \Expk{ \norms{x^{k+1} - x^k}^2 } = 0$.
This proves the first limit of \eqref{eq:AVFR4NE_convergence2c}.

\vspace{0.75ex}
\noindent\textbf{Step 5.}~\textit{The Big-O and small-o rates of $\Expk{ \norms{Gx^k}^2 }$.}
Combining \eqref{eq:AVFR4NE_proof15}, \eqref{eq:AVFR4NE_proof15h}, the first bound of  \eqref{eq:AVFR4NE_convergence2b}, and noting that $k+r-1 \leq k+r+2$, we can show that
\myeqn{
\arraycolsep=0.2em
\begin{array}{lcl}
\Expk{ \norms{Gx^k}^2}  & \leq & \frac{(r+2)^2 C_1   \norms{x^0 - x^{\star}}^2  }{2\beta^2r^2 (k+r-1)(k+r+2)} + \frac{(r+2)^2 C_2  \norms{x^0 - x^{\star}}^2  }{2\beta^2r^2 (k+r+2)^2}   \leq  \frac{(r+2)^2(C_1 + C_2)}{2\beta^2r^2 (k+r-1)(k+r + 2 )}     \norms{x^0 - x^{\star}}^2  =  \frac{C_3  \norms{x^0 -  x^{\star}}^2}{\beta^2 (k + r - 1)(k+r+2)},
\end{array}
}
where $C_3 := \frac{(r+2)^2(C_1 + C_2)}{2r^2 }$ as in \eqref{eq:AVFR4NE_constants}.
This proves the second line of \eqref{eq:AVFR4NE_convergence2b}.

Alternatively, the second limit in \eqref{eq:AVFR4NE_convergence2c} follows from combining  \eqref{eq:AVFR4NE_proof15}, the limit $\lim_{k\to\infty} k^2 \Expk{ \norms{x^{k+1} - x^k + \eta_kGx^k}^2 } = 0$, and the first limit of \eqref{eq:AVFR4NE_convergence2c}.
Finally, the final iteration complexity statement is a direct consequence of line 2 in \eqref{eq:AVFR4NE_convergence2b}.
\Eproof
\endproof

\beforesubsec
\subsection{Almost-sure convergence analysis.}\label{subsec:almost_sure_convergence}
\aftersubsec
Our next step is  to prove the almost-sure convergence of \eqref{eq:VRKM4ME} as in the following theorem.

\begin{theorem}\label{th:AVFR4NE_as_convergence}
Under the same conditions and parameters as in Theorem~\ref{th:AVFR4NE_convergence}, $\sets{x^k}$ generated by \eqref{eq:VRKM4ME} satisfies the following almost-sure limits:
\myeq{eq:AVFR4NE_as_limits}{
\lim_{k\to\infty} k^2\norms{x^{k+1} - x^k}^2 = 0 \quad \textrm{and} \quad \lim_{k\to\infty}k^2\norms{Gx^k}^2 = 0 \quad\textrm{almost surely}.
}
Moreover, $\sets{x^k}$ converges almost surely to a $\zer{G}$-valued random variable $x^{\star}$ $($i.e., $x^{\star} \in \zer{G}$$)$.
\end{theorem}

\proof{\textbf{Proof.}}
First, since $\zer{G} \neq\emptyset$ by Assumption~\ref{as:A1}, we fix $x^{\star} \in \zer{G}$.
Then, $\Ec_k$ in \eqref{eq:AVFR4NE_Lfunc} depends on such an $x^{\star}$.
For $U_k = \frac{1}{n}\sum_{i=1}^n\norms{G_ix^k - G_ix^{k-1}}^2$, we can rewrite \eqref{eq:AVFR4NE_key_inequality100} as
\myeq{eq:AVFR4NE_as_proof1}{
\arraycolsep=0.2em
\begin{array}{lcl}
\Expsn{k}{ \Ec_k } & \leq & \Ec_{k-1} -  \big[ \Lambda (k+r)^2 U_k + (2k+r + 1) \norms{x^k - x^{k-1}}^2\big].
\end{array}
}
Next, applying the Supermartingale Theorem, i.e. Lemma~\ref{le:RS_lemma}, to \eqref{eq:AVFR4NE_as_proof1} with $X_k := \Ec_{k-1}$, $\alpha_k := 0$, $V_k := \Lambda (k+r)^2 U_k + (2k+r + 1) \norms{x^k - x^{k-1}}^2$, and $R_k := 0$, we get
\myeq{eq:AVFR4NE_as_proof2}{
\arraycolsep=0.2em
\begin{array}{ll}
& \lim_{k\to\infty}\Ec_k \quad \textrm{exists almost surely}, \vspace{1ex}\\
& \sum_{k=1}^{\infty} (2k+r + 1) \norms{x^k - x^{k-1}}^2  < +\infty \quad \textrm{almost surely}, \vspace{1ex}\\
& \sum_{k=1}^{\infty}  (k+r)^2 U_k  < +\infty \quad \textrm{almost surely}.
\end{array}
}
Indeed, $\Ec_{k-1}$ is $\Fc_k$-measurable and
$\Expsn{k}{X_{k+1}} = \Expsn{k}{\Ec_k}$, so Lemma~\ref{le:RS_lemma} applies with the above shifted indexing.

Now, to prove the limits \eqref{eq:AVFR4NE_as_limits}, we follow the same arguments from \textbf{Step 3} to \textbf{Step 5} in the proof of Theorem~\ref{th:AVFR4NE_convergence} but without taking the total expectation and repeatedly applying Lemma~\ref{le:RS_lemma} with the corresponding shifted sequences.
We skip their detailed derivations to avoid repetition.
The only additional point requiring justification is \textbf{Step~3}, where we need to prove $\lim_{k\to\infty}(k+r)^2\Delta_k = 0$ almost surely.
Indeed, from the last line of \eqref{eq:ub_SG_estimator} and $(1-\gamma_k)^2 = \frac{r^2}{(k+r)^2}$, we have
\myeq{eq:AVFR4NE_as_proof2b}{
\arraycolsep=0.2em
\begin{array}{lcl}
(k + r)^2 \Expsn{k}{\Delta_k} &\leq & (k  + r - 1)^2\Delta_{k-1}  - \rho(k + r-1)^2\Delta_{k-1}  + \Theta(k + r)^2  U_k  +   \hat{\Theta}(k+r-1)^2U_{k-1}.
\end{array} 
}
Since $\sum_{k=1}^{\infty}(k+r)^2U_k < +\infty$ almost surely due to \eqref{eq:AVFR4NE_as_proof2}, applying Lemma~\ref{le:RS_lemma} to \eqref{eq:AVFR4NE_as_proof2b} with $X_k := (k \! + \! r \! - \! 1)^2\Delta_{k-1}$, $\alpha_k := 0$, $V_k := \rho(k \! + \! r \! - \! 1)^2\Delta_{k-1}$, and
$R_k := \Theta(k \! + \! r)^2U_k + \hat{\Theta}(k+r-1)^2U_{k-1}$, we get
\myeq{eq:AVFR4NE_as_proof2c}{ 
\hspace{-1ex}
\arraycolsep=0.2em
\begin{array}{lcl}
\lim_{k\to\infty}(k+r)^2\Delta_k  \ \textrm{exists} \quad \textrm{and} \quad \sum_{k=1}^{\infty}(k+r-1)^2\Delta_{k-1} <+\infty~~\textrm{almost surely.}
\end{array}
\hspace{-3ex}
}
The last summability bound also implies that $\lim_{k\to\infty}(k+r)^2\Delta_k = 0$ almost surely.

Finally, we prove that $\sets{x^k}$ converges almost surely to a $\zer{G}$-valued random variable.
Indeed, we first substitute $t_k = k + r + 1$ and $\mu = 1$ into $\Ec_k$ from \eqref{eq:AVFR4NE_Lfunc} to rewrite it as follows:
\myeq{eq:AVFR4NE_Lfunc_as}{
\arraycolsep=0.2em
\begin{array}{lcl}
\Ec_k & = & 4r\beta(k + r) \big[ \iprods{Gx^k, x^k - x^{\star}} - \beta\norms{ Gx^k }^2 \big]  + (r^2 + r) \norms{x^k - x^{\star}}^2 + (k+r+2)^2\norms{x^{k+1} - x^k}^2  \vspace{1ex}\\ 
&& + {~}  2r(k+r+2)\iprods{x^{k+1} - x^k, x^k - x^{\star}}  +  \frac{4 \beta^2 \hat{\Theta}}{\rho }  (k+r)^2 U_k   +  \frac{4  \beta^2(1-\rho)}{ \rho } (k+r)^2 \Delta_k.
\end{array}
}
Next, since $\lim_{k\to\infty}\Ec_k$ exists almost surely from \eqref{eq:AVFR4NE_as_proof2}, and $\Ec_k \geq r\norms{x^k - x^{\star}}^2$ due to \eqref{eq:AVFR4NE_Lfunc}, there exists an almost surely finite random variable $M_0 > 0$ such that $\norms{x^k - x^{\star}} \leq M_0$ almost surely for all $k\geq 0$.
Using this fact and the Cauchy-Schwarz inequality, one can easily show that
\myeqn{ 
\arraycolsep=0.2em
\begin{array}{lcl}
(k+r) \vert  \iprods{Gx^k, x^k - x^{\star}}\vert &\leq &  M_0(k+r)\norms{Gx^k}, \vspace{1ex}\\
(k+r+2)\vert \iprods{x^{k+1} - x^k, x^k - x^{\star}} \vert &\leq &  M_0(k+r+2)\norms{x^{k+1} - x^k}.
\end{array}
}
Combining these inequalities and the limits in \eqref{eq:AVFR4NE_as_limits}, almost surely, we get
\myeqn{ 
\arraycolsep=0.2em
\begin{array}{lcl}
\lim_{k\to\infty} (k+r) \vert  \iprods{Gx^k, x^k - x^{\star}}\vert & = &  0, \vspace{1ex}\\
\lim_{k\to\infty} (k+r+2)\vert \iprods{x^{k+1} - x^k, x^k - x^{\star}} \vert & = & 0.
\end{array}
}
Now, we have $\lim_{k\to\infty} (k+r)^2\Delta_k = 0$ from \eqref{eq:AVFR4NE_as_proof2c}.
Similarly, the last line of \eqref{eq:AVFR4NE_as_proof2} also implies that $\lim_{k\to\infty}(k+r)^2U_k = 0$ almost surely.

Collecting all the limits proven above, \eqref{eq:AVFR4NE_as_limits}, and the existence of $\lim_{k\to\infty}\Ec_k$ from \eqref{eq:AVFR4NE_as_proof2}, we can show from \eqref{eq:AVFR4NE_Lfunc_as} that, for each fixed $x^{\star} \in \zer{G}$, $\lim_{k\to\infty}\norms{x^k - x^{\star}}^2$ exists almost surely.

To make this conclusion hold simultaneously for all $x^{\star}\in\zer{G}$, we use a standard separability argument.
Since $G$ is continuous by Condition~\eqref{eq:co-coerciveness_of_G}, $\zer{G}$ is closed and hence separable in the finite-dimensional space.
Let $\sets{x_j^{\star}}_{j\geq 1}$ be a countable dense subset of $\zer{G}$.
For each $j\geq 1$, the preceding argument shows that $\lim_{k\to\infty}\norms{x^k - x_j^{\star}}$ exists almost surely.
Taking the countable intersection of these probability-one events, we obtain a common probability-one event on which $\lim_{k\to\infty}\norms{x^k - x_j^{\star}}$ exists for every $j\geq 1$.
Since $\sets{x_j^{\star}}_{j\geq 1}$ is dense in $\zer{G}$, for any $x^{\star}\in\zer{G}$,  choose a subsequence $\sets{x^{\star}_{j_\ell}}_{\ell\geq 1}$ such that $x^{\star}_{j_\ell}\to x^{\star}$ as $\ell\to\infty$.
Then, by the triangle inequality, we get $\left\vert \norms{x^k - x^{\star}} - \norms{x^k - x^{\star}_{j_\ell}}\right\vert \leq \norms{x^{\star} - x^{\star}_{j_\ell}}$.
Letting $\ell\to\infty$, we conclude that $\lim_{k\to\infty}\norms{x^k-x^{\star}}$ exists almost surely for every $x^{\star}\in\zer{G}$ on the same probability-one event.

Finally, since $G$ is cocoercive by Condition~\eqref{eq:co-coerciveness_of_G}, it is continuous.
Moreover, we have $\lim_{k\to\infty}\norms{Gx^k} = 0$ almost surely by \eqref{eq:AVFR4NE_as_limits} and $\sets{x^k}$ is bounded almost surely.
Combining these facts and the existence of the almost sure limit $\lim_{k\to\infty}\norms{x^k - x^{\star}}^2$ for every $x^{\star}\in\zer{G}$ on a common probability-one event, we can apply Lemma~\ref{le:A3_lemma} in the appendix to show that $\sets{x^k}$ converges to a $\zer{G}$-valued random variable almost surely.
\Eproof
\endproof

\vspace{0.75ex}
\noindent\textbf{Discussion.}
The expression \eqref{eq:AVFR4NE_as_limits} shows an $\SmallOs{1/k^2}$ almost-sure convergence rate for both the ``kinetic energy term'' $\norms{x^{k+1} - x^k}^2$ and the ``force magnitude'' term $\norms{Gx^k}^2$.
The almost-sure convergence of the iterates $\sets{x^k}$ has been established for non-accelerated stochastic and randomized methods, see, e.g., \cite{combettes2015stochastic,davis2022variance}.
However, since our method is an accelerated variance reduction framework, to the best of our knowledge, the results in Theorem~\ref{th:AVFR4NE_as_convergence} are new to the literature.

\beforesubsec
\subsection{Complexity analysis for the loopless SVRG and SAGA variants.}\label{subsec:SVRG_SAGA}
\aftersubsec
Let us specialize Theorem~\ref{th:AVFR4NE_convergence} to analyze the expected oracle complexity of \eqref{eq:VRKM4ME}  using \eqref{eq:SVRG_estimator} and \eqref{eq:SAGA_estimator}.
The proofs of the following results are deferred to Appendix~\ref{apdx:co:SVRG_SAGA_complexity}.

\begin{corollary}\label{co:SVRG_complexity}
Let $\sets{x^k}$ be generated by \eqref{eq:VRKM4ME} to solve \eqref{eq:ME} using the \textbf{loopless SVRG estimator} \eqref{eq:SVRG_estimator} and the same parameters as in Theorem~\ref{th:AVFR4NE_convergence} with $r := 3$ such that $1 \leq b\mbf{p}^2 \leq 32$.
Then, under the same conditions as in Theorem \ref{th:AVFR4NE_convergence}, if we choose $\beta := \frac{b\mbf{p}^2}{2L(b\mbf{p}^2 + 64)}$, then $\beta \in \left[ \frac{1}{130L}, \frac{1}{6L}\right]$.
Moreover, we have
\myeq{eq:AVFR4NE_for_SVRG}{
\arraycolsep=0.2em
\begin{array}{lcl c l c l}
\Expn{\norms{x^{k+1} - x^k}^2 } & \leq  & \dfrac{ C_2 \norms{x^0 - x^{\star}}^2 }{(k+r+2)^2} \quad & \text{and}  & \quad {\displaystyle\lim_{k\to\infty}} k^2 \Expn{\norms{x^{k+1} - x^k}^2 } & = & 0,   \vspace{1ex}\\ 
\Expn{\norms{Gx^k}^2 }  & \leq & \dfrac{C_3 \norms{x^0 - x^{\star}}^2}{ \beta^2 (k+r-1)(k+r + 2) } \quad &  \text{and}  & \quad {\displaystyle\lim_{k\to\infty}}k^2 \Expn{\norms{Gx^k}^2 } & = & 0,
\end{array}
}
where $C_2 \leq 2360$ and $C_3 \leq 3353$ are explicitly given in \eqref{eq:AVFR4NE_constants}.

For a given $\epsilon > 0$, if we choose $\mbf{p} = \BigOs{\frac{1}{n^{1/3}}}$ and $b = \BigOs{n^{2/3}}$, then \eqref{eq:VRKM4ME} requires $\BigOs{ n + \frac{n^{2/3}L\Rc_0}{\epsilon} }$ evaluations of $G_i$ to achieve $\Expn{\norms{Gx^k}^2 } \leq \epsilon^2$, where $\Rc_0 := \norms{x^0 - x^{\star}}$.
\end{corollary}

\begin{corollary}\label{co:SAGA_complexity}
Let $\sets{x^k}$ be generated by \eqref{eq:VRKM4ME} to solve \eqref{eq:ME} using the \textbf{SAGA estimator} \eqref{eq:SAGA_estimator} and the same parameters as in Theorem~\ref{th:AVFR4NE_convergence} with $r := 3$ such that $1 \leq b \leq \min\sets{n, 4n^{2/3}}$.
Then, under the same conditions as in Theorem \ref{th:AVFR4NE_convergence}, if we choose $\beta :=  \frac{b^3}{2L(b^3 + 64n^2)}$, then $\beta \in \big[\frac{1}{2L(1+64n^2)}, \frac{1}{4L}\big]$.
Moreover, we have 
\myeq{eq:AVFR4NE_for_SAGA}{
\arraycolsep=0.2em
\begin{array}{lcl c l c l}
\Expn{\norms{x^{k+1} - x^k}^2 } & \leq  & \dfrac{ C_2 \norms{x^0 - x^{\star}}^2 }{(k+r+2)^2} \quad & \text{and}  & \quad {\displaystyle\lim_{k\to\infty}} k^2 \Expn{\norms{x^{k+1} - x^k}^2 } & = & 0,   \vspace{1ex}\\ 
\Expn{\norms{Gx^k}^2 }  & \leq & \dfrac{C_3 \norms{x^0 - x^{\star}}^2}{ \beta^2 (k+r-1)(k+r + 2) } \quad &  \text{and}  & \quad {\displaystyle\lim_{k\to\infty}} k^2 \Expn{\norms{Gx^k}^2 } & = & 0,
\end{array}
}
where $C_2 \leq 2559$ and $C_3 \leq 3636$ are explicitly given in \eqref{eq:AVFR4NE_constants}.

For a given $\epsilon > 0$, if we choose $b = \BigOs{n^{2/3}}$, then \eqref{eq:VRKM4ME}  requires $\BigOs{ n + \frac{ n^{2/3}L\Rc_0}{\epsilon} }$ evaluations of $G_i$  to achieve $\Expn{\norms{Gx^k}^2} \leq \epsilon^2$, where $\Rc_0 := \norms{x^0 - x^{\star}}$.
\end{corollary}

Both $C_2$ and $C_3$ in Corollaries \ref{co:SVRG_complexity} and \ref{co:SAGA_complexity} can be refined to obtain smaller upper bounds.
Nevertheless, we do not try to numerically optimize them here.
Corollaries \ref{co:SVRG_complexity} and \ref{co:SAGA_complexity} show that both the SVRG and SAGA variants have an oracle complexity of $\BigOs{n + \frac{n^{2/3}L\Rc_0}{\epsilon}}$ under appropriate choices of the mini-batch size $b$ and probability $\mbf{p}$.
This complexity is better than deterministic accelerated methods by a factor of $n^{1/3}$ and non-accelerated variance reduction methods such as  \cite{alacaoglu2022beyond,alacaoglu2021stochastic,davis2022variance} by a factor of $\frac{1}{\epsilon}$.
As suggested by our theory, one can choose $\beta \in [\frac{1}{6L}, \frac{1}{4L}]$ when implementing \eqref{eq:VRKM4ME}.

\vspace{1ex}
\beforesec
\section{Linear Convergence under Strong Quasi-monotonicity.}\label{sec:star_monotone_convergence}
\aftersec
In this section, we establish a linear convergence rate and estimate the oracle complexity of \eqref{eq:VRKM4ME} under Assumption~\ref{as:A1} and the following additional assumption:
\begin{assumption}\label{as:A3}
$G$ in \eqref{eq:ME} is $\sigma$-strongly quasi-monotone, i.e., there exist $x^{\star} \in \zer{G}$ and $\sigma > 0$ such that $\iprods{Gx, x - x^{\star}} \geq \sigma\norms{x - x^{\star}}^2$ for all $x \in \dom{G}$.
\end{assumption}
If $G$ is strongly monotone with a strong monotonicity modulus $\sigma > 0$, then it satisfies Assumption~\ref{as:A3}.
However, the converse is not true in general.
For example, if we consider $G : \R \to \R$ such that $Gx = x + \sin(x)$, then it is straightforward to verify that $x^{\star} = 0 \in \zer{G}$ and $G$ is $\sigma$-strongly quasi-monotone with $\sigma := 1 - \frac{1}{\pi} > 0$, but it is not strongly monotone.
Hence, this assumption is weaker than the strong monotonicity of $G$.
Note that we do not require each summand $G_i$ to be strongly monotone.
Assumption~\ref{as:A3} is not new.
It has been used in \cite{peng2016arock} and recently in \cite{davis2022variance,loizou2021stochastic}.

\beforesubsec
\subsection{Linear convergence and iteration complexity.}\label{subsec:linear_convergence}
\aftersubsec
For simplicity of presentation, we first define the following constants:
\myeq{eq:AVFR4SNE_constants}{
\hspace{-2ex}
\arraycolsep=0.2em
\begin{array}{lcl}
\kappa := \frac{L}{\sigma}, \ \  \Gamma := \rho + 2(\Theta + \hat{\Theta}),  \ \ M := 2(2\Gamma - 1) \kappa, \ \ \text{and} \ \  N := 3\Gamma\kappa + 2(1 - 2\rho).
\end{array}
\hspace{-2ex}
}
Suppose that $N^2 + 12\rho M \geq 0$ and $2\rho < 1$.
Then, we define
\myeq{eq:AVFR4SNE_par_conditions}{
\arraycolsep=0.2em
\begin{array}{lll}
\bar{\beta} := \frac{1}{\sigma}  \min\Big\{ \frac{3}{5}, \frac{3\rho}{2(1 - 2\rho)}, \frac{1}{2\kappa},  \frac{6 \rho }{N + \sqrt{N^2 + 12\rho M}} \Big\} > 0.
\end{array}
}
Theorem \ref{th:AVFR4SNE_convergence} states a linear convergence result for \eqref{eq:VRKM4ME} under Assumption~\ref{as:A3}.

\begin{theorem}\label{th:AVFR4SNE_convergence}
Suppose that $G$ in \eqref{eq:ME} satisfies Condition~\eqref{eq:co-coerciveness_of_G} and  Assumption~\ref{as:A3} such that $N^2 + 12\rho M \geq 0$ and $2\rho < 1$ for given constants defined in \eqref{eq:AVFR4SNE_constants}.
Let $\bar{\beta}$ be defined by \eqref{eq:AVFR4SNE_par_conditions}.
Let $\sets{x^k}$ be generated by \eqref{eq:VRKM4ME} to solve \eqref{eq:ME} using a stochastic estimator $\widetilde{S}^k$ for $S^k$ satisfying Definition~\ref{de:ub_SG_estimator} and the following fixed parameters:
\myeq{eq:AVFR4SNE_para_update}{
\arraycolsep=0.2em
\begin{array}{lcl}
\theta_k = \theta := \frac{1}{3}, \quad \gamma_k = \gamma := \frac{1}{2} , \quad \text{and} \quad \eta_k = \eta := \beta \in (0, \bar{\beta}),
\end{array}
}
Then, with $\omega := \frac{2\beta\sigma}{3 + 4\beta\sigma} \in (0, 1/2)$ we have
\myeq{eq:AVFR4SNE_convergence}{
\arraycolsep=0.2em
\begin{array}{lcl}
\Expk{ \norms{x^k - x^{\star}}^2 }  & \leq & 4(1 + 2L^2\beta^2) (1 - \omega)^k \norms{x^0 - x^{\star}}^2.
\end{array}
}
Consequently,  for a given tolerance $\epsilon > 0$,  \eqref{eq:VRKM4ME} requires $k =  \BigOs{\big(2 + \frac{3}{2\beta\sigma}\big)\log(\epsilon^{-1})}$ iterations to achieve an $\epsilon$-solution $x^k$ such that $\Expk{ \norms{x^k - x^{\star}}^2 } \leq \epsilon^2$.
\end{theorem}

\proof{\textbf{Proof.}} 
We first choose $r := 1$ and $s := 3$ in Lemma~\ref{le:AVFR4SNE_key_estimate} from the appendix.
Then, the update rule \eqref{eq:AVFR4SNE_para_update_with_r} reduces to \eqref{eq:AVFR4SNE_para_update}, and  $\omega := \frac{2r \beta\sigma}{2r + 1 + 2\beta\sigma(s-1)} = \frac{2\beta\sigma}{3 + 4\beta\sigma} \in (0, 1/2)$.

Next,  $\alpha$ in Lemma~\ref{le:AVFR4SNE_key_estimate} becomes $\alpha = 8\beta\big[\frac{1}{L} - \big(1 + \frac{2(\Theta+\hat{\Theta})}{\rho-\omega}\big)\beta \big]$.
Therefore, if 
\myeqn{
\arraycolsep=0.2em
\begin{array}{ll}
\frac{1}{\kappa\beta\sigma} = \frac{1}{L\beta} \geq  \frac{\rho + 2(\Theta+\hat{\Theta}) - \omega}{\rho-\omega} = \frac{\Gamma - \omega}{\rho-\omega} = \frac{3\Gamma + 2(2\Gamma - 1)\beta\sigma}{3\rho - 2(1 - 2\rho)\beta\sigma},
\end{array}
}
then $\alpha \geq 0$, where $\kappa := \frac{L}{\sigma}$ and $\Gamma := \rho + 2(\Theta + \hat{\Theta})$ are defined in \eqref{eq:AVFR4SNE_constants}.
The above condition holds if $2(2\Gamma-1)\kappa\beta^2\sigma^2 + [3\Gamma\kappa + 2(1-2\rho)] \beta\sigma - 3\rho \leq 0$.
Using $M := 2(2\Gamma-1)\kappa$ and $N := 3\Gamma\kappa + 2(1-2\rho)$ from \eqref{eq:AVFR4SNE_constants}, this inequality becomes $M(\beta\sigma)^2 + N\beta\sigma - 3\rho \leq 0$.
Solving this inequality in $\beta\sigma$, we get $\beta\sigma \leq \frac{6\rho}{N + \sqrt{N^2 + 12\rho M}}$, provided that  $N^2 + 12\rho M \geq 0$.
Combining this condition, $\beta \leq \frac{1}{2L} = \frac{1}{2\kappa\sigma}$, and \eqref{eq:AVFR4SNE_para_condition2}, we can show that $0 < \beta \leq \bar{\beta}$ as in \eqref{eq:AVFR4SNE_para_update} for $\bar{\beta}$ defined by \eqref{eq:AVFR4SNE_par_conditions}.

Now, since $\alpha \geq 0$ and $\frac{1}{2L} - \beta \geq 0$, it follows from \eqref{eq:AVFR4SNE_key_estimate} that $\Expsn{k}{\Ec_k} \leq (1-\omega)\Ec_{k-1}$.
Taking the total expectation on both sides of this inequality and applying induction, we have $\Expn{\Ec_k} \leq (1 - \omega)^k\Expn{\Ec_0}$.
Similar to Lemma~\ref{le:lower_bounding_Ec_0}, we can show that $\Ec_0 \leq 4(1 +  2 L^2\beta^2)  \norms{x^0 - x^{\star}}^2$.
Alternatively,  from \eqref{eq:AVFR4SNE_Lfunc} and $r = 1$, we also have $\Ec_k \geq  \norms{x^k - x^{\star}}^2$.
Therefore, combining these three results, we  get \eqref{eq:AVFR4SNE_convergence}.

Finally, for a given $\epsilon > 0$, to achieve $\Expn{ \norms{x^k - x^{\star}}^2 } \leq \epsilon^2$, we impose $(1 - \omega)^k C_0 \norms{x^0 - x^{\star}}^2 \leq \epsilon^2$, where $C_0 := 4(1 + 2L^2\beta^2)$.
Since $-\log(1 - \omega) \geq \omega = \frac{2\beta\sigma}{3+4\beta\sigma}$, the last condition leads to  $k \geq \frac{1}{\omega}\log\big(\frac{C_0\norms{x^0 - x^{\star}}^2}{\epsilon^2}\big)$.
Thus, we can choose $k := \BigOs{ \big( 2 + \frac{3}{2\beta\sigma}\big) \log( \frac{1}{\epsilon}) }$ to complete our proof.
\Eproof
\endproof

\vspace{1ex}
\begin{remark}\label{le:almost_sure_convergence}
If we ignore the dependence on $n$, then the worst-case iteration complexity stated in Theorem~\ref{th:AVFR4SNE_convergence} is $\BigOs{\kappa\log( \frac{1}{\epsilon})}$, where $\kappa := \frac{L}{\sigma}$ can be viewed as the condition number of $G$.
By the lower bound in \cite{zhang2022lower}, this complexity is optimal up to a constant factor.
This result is different from the $\BigOs{\sqrt{\kappa}\log( \frac{1}{\epsilon})}$ complexity known in convex optimization.
\end{remark}

\subsection{Oracle complexity.}\label{subsec:oracle_complexity_for_case2}
\aftersubsec
Theorem~\ref{th:AVFR4SNE_convergence} only states a linear convergence rate and iteration complexity in terms of $\mathbb{E}[ \norms{x^k - x^{\star}}^2 ]$ for \eqref{eq:VRKM4ME}, but does not specify  its oracle complexity.
We now specialize Theorem~\ref{th:AVFR4SNE_convergence} to our SVRG and SAGA estimators  to obtain the following complexity bounds in expectation,  see Appendix~\ref{apdx:th:AVFR4SNE_convergence} for the proof.

\begin{corollary}\label{co:complexity_for_AVFR4SNE}
Under the same conditions and parameters as in Theorem~\ref{th:AVFR4SNE_convergence}, the following statements hold:
\begin{itemize}
\item[$\mathrm{(a)}$] If we use the SVRG estimator $\widetilde{S}^k$ in \eqref{eq:SVRG_estimator} for \eqref{eq:VRKM4ME}, then by choosing $b = \BigO{ n^{2/3} }$, $\mbf{p} = \BigO{ \frac{1}{n^{1/3}} }$, and $\beta = \frac{\bar{\beta}}{2}$ in \eqref{eq:AVFR4SNE_par_conditions}, \eqref{eq:VRKM4ME} requires $\BigOs{n + \max\sets{n, n^{2/3}\kappa} \log(\frac{1}{\epsilon})}$ evaluations of $G_i$ to attain $\Expn{ \norms{x^k - x^{\star}}^2} \leq \epsilon^2$.

\item[$\mathrm{(b)}$]
 If we use the SAGA estimator $\widetilde{S}^k$ in \eqref{eq:SAGA_estimator} for \eqref{eq:VRKM4ME}, then by choosing   $b = \BigO{ n^{2/3} }$ and $\beta = \frac{\bar{\beta}}{2}$  in \eqref{eq:AVFR4SNE_par_conditions},  \eqref{eq:VRKM4ME} requires $\BigOs{n + \max\sets{n, n^{2/3}\kappa} \log(\frac{1}{\epsilon})}$ evaluations of $G_i$ to achieve $\Expn{ \norms{x^k - x^{\star}}^2} \leq \epsilon^2$.
\end{itemize}
\end{corollary}

The oracle complexity stated in Corollary~\ref{co:complexity_for_AVFR4SNE} is better than the deterministic counterparts by a factor of $n^{1/3}$ whenever $\kappa = \Omega(n^{1/3})$, similar to the result in  \cite{davis2022variance}, but our method is different.

\section{Application to Finite-Sum Inclusions.}\label{sec:AVFR4NI}
\aftersec
We now apply \eqref{eq:VRKM4ME} to solve \eqref{eq:MI} using a BFS operator  \cite{attouch2018backward,davis2022variance}.
We recall it here for convenience:
\myeqn{
\textrm{Find $u^{\star}\in\dom{\Psi}$ such that:}~ 0 \in \Psi{u^{\star}} :=  Gu^{\star} + Tu^{\star}.
\tag{\ref{eq:MI}}
}
We require the following assumption for \eqref{eq:MI}:

\begin{assumption}\label{as:A4}
$\zer{\Psi} := \sets{u^{\star} : 0 \in \Psi{u^{\star}} } \neq \emptyset$, 
each $G_i$ is $\frac{1}{L_g}$-cocoercive for $i \in [n]$ with $L_g > 0$, 
and $T$ is maximally $\nu$-co-hypomonotone such that $\nu \geq 0$ and $L_g\nu < 1$.
\end{assumption}

We recall from \cite{bauschke2020generalized} that $T$ is $\nu$-co-hypomonotone if there exists $\nu \geq 0$ such that $\iprods{x - y, u - v} \geq -\nu\norms{x - y}^2$ for all $(u, x), (v, y) \in \gra{T}$. If $\nu = 0$, then $T$ becomes monotone.
We say that $T$ is maximally $\nu$-co-hypomonotone if $\gra{T}$ is not properly contained in the graph of any other $\nu$-co-hypomonotone operator.

Note that under Assumption~\ref{as:A4}, $\Psi$ in \eqref{eq:MI} is not necessarily monotone.
For example, we take $n=1$ and $Gu = G_1u := \mbf{G}u + \mbf{g}$ for a symmetric positive semidefinite matrix $\mbf{G}$ in $\R^{p\times p}$ and a given $\mbf{g}\in\R^p$.
Then, we can view $G = \nabla{f}$ as the gradient of a convex quadratic and $L_g$-smooth function $f(u) := \frac{1}{2}u^T\mbf{G}u + \mbf{g}^Tu$ with $L_g := \lambda_{\max}(\mbf{G})$.
It is well-known that  $G$ is $\frac{1}{L_g}$-cocoercive (see \cite{Nesterov2004}). 

Now, we take $Tu := \mbf{T}u + \mbf{r}$ for a given symmetric and invertible matrix $\mbf{T}$ in $\R^{p\times p}$ and a given $\mbf{r} \in \R^p$ such that $\lambda_{\min}(\mbf{T}^{-1}) < 0$.
Then, we have $\iprods{Tu - Tv, u - v} \geq \lambda_{\min}(\mbf{T}^{-1})\norms{Tu - Tv}^2$ for all $u, v \in \R^p$, showing that $T$ is $\nu$-co-hypomonotone with $\nu := -\lambda_{\min}(\mbf{T}^{-1}) > 0$.
However, $\Psi = G+T$ is not necessarily monotone since $\mbf{G} + \mbf{T}$ is not necessarily positive semidefinite.

\beforesubsec
\subsection{Variance-reduced fast BFS method.}\label{subsec:BFS_method}
\aftersubsec
Given $\lambda > 0$, and $G$ and $T$ in \eqref{eq:MI}, we define the following backward-forward splitting (BFS) operators:
\myeq{eq:BFS_operator}{
\arraycolsep=0.2em
\begin{array}{ll}
\Gc_{i,\lambda}x := G_i(J_{\lambda T}x) + \frac{1}{\lambda}(x - J_{\lambda T}x)~~ \textrm{for $i \in [n]$} \quad \textrm{and} \quad \Gc_{\lambda}x := \frac{1}{n}\sum_{i=1}^n\Gc_{i, \lambda}x.
\end{array}
}
Given $\Gc_{\lambda}$ in \eqref{eq:BFS_operator}, we consider the following BFS equation:
\myeq{eq:BFS_equation}{
\arraycolsep=0.2em
\begin{array}{ll}
\textrm{Find $x^{\star} \in \dom{\Gc_{\lambda}}$ such that:} \quad \Gc_{\lambda}x^{\star} = 0.
\end{array}
}
The following lemma establishes the relation between  \eqref{eq:MI}, \eqref{eq:ME}, and \eqref{eq:BFS_equation}.
Its proof is given in Appendix~\ref{apdx:subsec:proof_of_Lemma11}.

\begin{lemma}\label{le:BFS_oper_properties}
Assume that Assumption~\ref{as:A4} holds for \eqref{eq:MI}.
Let $\Gc_{i,\lambda}$ and $\Gc_{\lambda}$ be defined by \eqref{eq:BFS_operator} for $i \in [n]$.
Then, the following statements hold.
\begin{itemize}
\item[$\mathrm{(a)}$] 
$\zer{G + T} = J_{\lambda T}(\zer{\Gc_{\lambda}})$, i.e., $x^{\star}$ solves \eqref{eq:BFS_equation} iff  $u^{\star} := J_{\lambda T}x^{\star}$ solves \eqref{eq:MI}.

\item[$\mathrm{(b)}$] 
If $\lambda \geq 2\nu$, then $\norms{J_{\lambda T}x - J_{\lambda T}y} \leq \norms{x - y}$ for all $x, y \in \dom{J_{\lambda T}}$, i.e. the resolvent $J_{\lambda T}$ is nonexpansive.

\item[$\mathrm{(c)}$] 
If  $\frac{2}{L_g}\big(1 - \sqrt{1 - L_g\nu}\big) < \lambda < \frac{2}{L_g}\big(1 + \sqrt{1 - L_g\nu}\big)$, then $\Gc_{i,\lambda}$ is $\frac{1}{L}$-cocoercive with $L := \frac{4(1- L_g\nu)}{\lambda(4 - L_g\lambda) - 4\nu} > 0$.
Consequently, $\Gc_{\lambda}$ satisfies Condition \eqref{eq:co-coerciveness_of_G} with the same constant $L$.

\end{itemize}
\end{lemma}

Since $\Gc_{\lambda}$ satisfies Condition~\eqref{eq:co-coerciveness_of_G} as in \eqref{eq:ME} by Lemma~\ref{le:BFS_oper_properties}(c), we can apply \eqref{eq:VRKM4ME} to solve \eqref{eq:BFS_equation}.
Hence, our proposed method for solving \eqref{eq:MI} is presented as follows.

\noindent{~~}$\bullet$~Starting from $x^0 \in \dom{\Gc_{\lambda}}$, we set $x^{-1} := x^0$ and compute $u^{-1} := J_{\lambda T}x^{-1}$.

\noindent{~~}$\bullet$~At each iteration $k \geq 0$, we perform the following steps:
\begin{itemize}
\item[(1)]  Compute $u^k := J_{\lambda T}x^k$.
\item[(2)]  Construct an estimator $\widetilde{S}_u^k$ of $S_u^k := Gu^k - \gamma_k Gu^{k-1}$ satisfying Definition~\ref{de:ub_SG_estimator}.
\item[(3)]  Update
\myeq{eq:VRKM4NI}{
\arraycolsep=0.2em
\left\{ \begin{array}{lcl}
\widetilde{\Sc}_{\lambda}^k &:= & \widetilde{S}_u^k +  \frac{1}{\lambda}(x^k - u^k) - \frac{\gamma_k}{\lambda}(x^{k-1} -  u^{k-1}), \vspace{1ex}\\
x^{k+1} & := &  x^k + \theta_k (x^k - x^{k-1})  - \eta_k\widetilde{\Sc}_{\lambda}^k.
\end{array}\right.
}
\end{itemize}
From \eqref{eq:BFS_operator} and \eqref{eq:VRKM4NI}, it follows that $\widetilde{\Sc}_{\lambda}^k$ constructed by \eqref{eq:VRKM4NI} is also an unbiased variance-reduced estimator of $\Sc_{\lambda}^k := \Gc_{\lambda}x^k - \gamma_k \Gc_{\lambda}x^{k-1}$ satisfying Definition~\ref{de:ub_SG_estimator}.

\subsection{Convergence analysis.}\label{subsec:convergence_of_VRKM4NI}
\aftersubsec
Let us define a forward-backward splitting (FBS) residual of \eqref{eq:MI} as $\Fc_{\lambda}u := \frac{1}{\lambda}(u - J_{\lambda T}(u - \lambda Gu))$.
Then, it is well-known \cite[Proposition 26]{Bauschke2011} that  $\Fc_{\lambda}u^{\star} = 0$ iff $0 \in Gu^{\star} + Tu^{\star}$.
Therefore, if $\Expn{\norms{\Fc_{\lambda}u^k}^2} \leq \epsilon^2$, then we say that $u^k$ is an $\epsilon$-solution of \eqref{eq:MI}.
Now, we can prove the following convergence result for \eqref{eq:VRKM4NI}.  

\begin{theorem}\label{th:AVFR4NI_convergence}
Suppose that Assumption~\ref{as:A4} holds for \eqref{eq:MI}.
Let $\sets{(x^k, u^k)}$ be generated by  \eqref{eq:VRKM4NI} and $u^k := J_{\lambda T}x^k$ to solve \eqref{eq:MI} using the same parameters as in Theorem~\ref{th:AVFR4NE_convergence} but with $L := \frac{4(1- L_g\nu)}{\lambda(4 - L_g\lambda) - 4\nu}$ for any $\lambda$ such that $2\nu \leq \lambda < \frac{2(1 + \sqrt{1 - L_g\nu})}{L_g}$.

Then the corresponding conclusions of Theorem~\ref{th:AVFR4NE_convergence} and Theorem~\ref{th:AVFR4NE_as_convergence}  still hold, with $G$ replaced by $G_{\lambda}$, for $\norms{x^{k+1} - x^k}^2$, $\norms{u^{k+1} - u^k}^2$, $\norms{\Gc_{\lambda}{x}^k}^2$, and $\norms{\Fc_{\lambda}{u}^k}^2$, where $\Fc_{\lambda}u := \frac{1}{\lambda}(u - J_{\lambda T}(u - \lambda Gu))$ is the FBS residual of \eqref{eq:MI}.
\end{theorem}

\proof{\textbf{Proof.}}
First, by Lemma~\ref{le:BFS_oper_properties}(a), we have $\zer{G+T} = J_{\lambda T}(\zer{\Gc_{\lambda}})$.
Hence, solving \eqref{eq:MI} is equivalent to solving $\Gc_{\lambda}x^{\star} = 0$ in \eqref{eq:BFS_equation}.

\noindent 
Second, by Lemma~\ref{le:BFS_oper_properties}(c), $\Gc_{\lambda}$ satisfies Condition~\eqref{eq:co-coerciveness_of_G} with $L := \frac{4(1- L_g\nu)}{\lambda(4 - L_g\lambda) - 4\nu} > 0$.

\noindent 
Third, if we apply \eqref{eq:VRKM4ME} to solve \eqref{eq:BFS_equation}, then it can be expressed as \eqref{eq:VRKM4NI}.

\noindent 
Fourth, the conclusions of Theorem~\ref{th:AVFR4NE_convergence} and Theorem~\ref{th:AVFR4NE_as_convergence} for $\norms{x^{k+1} - x^k}^2$ and $\norms{\Gc_{\lambda}{x}^k}^2$ continue to hold both in expectation and almost surely.

\noindent 
Fifth, since $\norms{u^{k+1} - u^k}^2 = \norms{J_{\lambda T}x^{k+1} - J_{\lambda T}x^k }^2 \leq  \norms{x^{k+1} - x^k}^2$ by Lemma~\ref{le:BFS_oper_properties}(b), the conclusions for $\norms{x^{k+1} - x^k}^2$ imply the same conclusions for $\norms{u^{k+1} - u^k}^2$.

Finally, from \eqref{eq:BFS_operator}, we have $\lambda \Gc_{\lambda}x^k = \lambda Gu^k + x^k - u^k$, or equivalently, $x^k - \lambda \Gc_{\lambda}x^k = u^k - \lambda Gu^k$.
This leads to $J_{\lambda T}x^k - J_{\lambda T}(x^k - \lambda \Gc_{\lambda}x^k) = u^k - J_{\lambda T}(u^k - \lambda Gu^k)$.
Hence, $\norms{\Fc_{\lambda}{u}^k} = \norms{\lambda^{-1}(u^k - J_{\lambda T}(u^k - \lambda Gu^k))} = \lambda^{-1} \norms{J_{\lambda T}x^k - J_{\lambda T}(x^k - \lambda \Gc_{\lambda}x^k)} \leq  \norms{\Gc_{\lambda}x^k}$.
We conclude that the conclusions of Theorem~\ref{th:AVFR4NE_convergence} and Theorem~\ref{th:AVFR4NE_as_convergence}  still hold for $\norms{\Fc_{\lambda}{u}^k}^2$.
\Eproof
\endproof

\vspace{1ex}
\beforesec
\section{Numerical Experiments.}\label{sec:num_experiments}
\aftersec
We evaluate our algorithms on two numerical experiments and compare them with existing methods.
All algorithms were implemented in Python and run on a MacBook Pro with a 2.8 GHz quad-core Intel Core i7 processor and 16 GB of memory.

\beforesubsec
\subsection{Minimax optimization model.}\label{subsec:math_model}
\aftersubsec
We consider the following general convex-concave minimax optimization problem involving quadratic objective functions:
\myeq{eq:minimax_exam2}{
\begin{array}{lcl}
{\displaystyle\min_{z \in \R^{p_1}}\max_{\xi \in\R^{p_2}}} \Big\{ \Lc(z, \xi) := f(z) + \frac{1}{n}\sum_{i=1}^n\Hc_i(z, \xi) - g(\xi) \Big\},
\end{array}
}
where $\Hc_i(z, \xi) := \frac{1}{2} z^TA_iz +  z^TL_i\xi - \frac{1}{2}\xi^TB_i\xi  + b_i^{\top}z - c_i^{\top}\xi$,  $A_i \in \R^{p_1\times p_1}$ and $B_i \in \R^{p_2\times p_2}$ are symmetric positive semidefinite matrices, $L_i \in \R^{p_1\times p_2}$, $b_i \in \R^{p_1}$, $c_i \in \R^{p_2}$, and $f = \delta_{\Delta_{p_1}}$ and $g=\delta_{\Delta_{p_2}}$ are the indicator functions of the standard simplices in $\R^{p_1}$ and $\R^{p_2}$, respectively.
Let us first define $x := [z, \xi] \in \R^p$, where $p := p_1+p_2$.
Next, we define
\myeqn{
G_ix = \mathbf{G}_ix + \mbf{g}_i := \begin{bmatrix}A_i & L_i \\ -L_i^T & B_i\end{bmatrix}\begin{bmatrix}z \\ \xi \end{bmatrix} + \begin{bmatrix}b_i \\ c_i\end{bmatrix} = \begin{bmatrix}A_iz + L_i\xi + b_i \\ -L_i^Tz + B_i\xi + c_i\end{bmatrix} \quad\text{and} \quad 
T := \begin{bmatrix}\partial{f} \\ \partial{g}\end{bmatrix}.
}
Then, $G_i$ is an affine mapping from $\R^p$ to $\R^p$, but $\mathbf{G}_i$ is nonsymmetric.
Let $Gx := \frac{1}{n}\sum_{i=1}^nG_ix = \mathbf{G}x + \mathbf{g}$, where $\mathbf{G} := \frac{1}{n}\sum_{i=1}^n\mathbf{G}_i$ and $\mathbf{g} := \frac{1}{n}\sum_{i=1}^n\mathbf{g}_i$.
Then, the optimality condition of \eqref{eq:minimax_exam2} becomes $0 \in Gx + Tx$ as in \eqref{eq:MI}.

\beforesubsec
\subsection{Numerical experiments.}\label{subsec:implementation_detail}
\aftersubsec
We now describe our experiment setup and results in detail.

\vspace{0.5ex}
\textit{Data generation and experiment setup.}
We generate $A_i = Q_iD_iQ_i^T$ for a given orthonormal matrix $Q_i$ and a diagonal matrix $D_i$, whose diagonal entries $D_i^j$ are generated from a standard normal distribution and clipped as $\max\sets{D_i^j, 0}$.
The matrix $B_i$ is generated in the same way, while $L_i$, $b_i$, and $c_i$ are generated from a standard normal distribution.
We generate the data so that the finite-sum $G$ satisfies Condition~\eqref{eq:co-coerciveness_of_G}.

We perform two sets of experiments: \textit{Experiment 1}  consists of $10$ problem instances with $p=100$ ($p_1=67$ and $p_2 = 33$) and $n = 5000$, while \textit{Experiment 2} consists of $10$ problem instances with $p=200$ ($p_1=133$ and $p_2 = 67$) and $n=10000$.
We then report the \textit{mean} over  the $10$ problem instances in terms of $\frac{\norms{Gx^k}}{\norms{Gx^0}}$ for \eqref{eq:ME} and in terms of $\frac{\norms{\Gc_{\lambda}x^k}}{\norms{\Gc_{\lambda}x^0}}$ for \eqref{eq:MI} due to Theorem~\ref{th:AVFR4NI_convergence}, where $x^0 := \mathbf{1} \in \R^p$, the vector of all ones, is the initial point and $\lambda := \frac{1}{L}$ in \eqref{eq:BFS_operator}.

\vspace{0.5ex}
\textit{Competitors.}
We implement two variants of \eqref{eq:VRKM4ME} to solve \eqref{eq:minimax_exam2}:  VFKM-Svrg -- a loopless SVRG variant using \eqref{eq:SVRG_estimator} and VFKM-Saga -- a SAGA variant using~\eqref{eq:SAGA_estimator}.
We also compare our methods with two deterministic optimistic gradient methods: OG -- the non-accelerated OG in \cite{daskalakis2018training} and  AOG -- the accelerated OG in \cite{tran2022connection}, RF-Saga -- the variance-reduced root-finding method in \cite{davis2022variance},  VrFRBS -- the variance-reduced FRBS  in \cite{alacaoglu2022beyond}, and VrEG -- the  variance-reduced extragradient method in \cite{alacaoglu2021stochastic}.
These are all recently developed methods.

\vspace{0.5ex}
\textit{Parameter selection.}
For OG and AOG, we set their stepsize $\eta := \frac{1}{2L}$, where $L$ is the cocoercivity constant of $G$.  
For our VFKM-Svrg, we choose $\beta := \frac{0.15}{L} \approx \frac{1}{6L}$, and for our VFKM-Saga, we set  $\beta := \frac{1}{4L}$ and choose $\Bc_k =\hat{\Bc}_k$.
We also choose a mini-batch size of $b := \lfloor 0.5 n^{2/3} \rfloor$, and a probability $\mathbf{p} := \frac{1}{n^{1/3}}$ for VFKM-Svrg as suggested by our theory.
We select $r = 20$ so that $r > 2$.
For VrFRBS, we choose its stepsize $\tau := \frac{5\times 0.99(1 - \sqrt{1 - \mathbf{p}})}{2L}$ (five times its theoretical stepsize), and for RF-Saga, we set its stepsize $\lambda := \frac{1}{4L}$, which is larger than the value $\frac{1}{8L}$ used in \cite[Example 2.3]{davis2022variance}, and for VrEG, we let $\tau := \frac{0.99\sqrt{1 - \alpha}}{L}$ for $\alpha := 1 - \mathbf{p}$ as suggested by the corresponding theory.
We also use $\mathbf{p} :=  \frac{1}{n^{1/3}}$ in both algorithms, which is the same as ours, though their theory suggests smaller values of $\mathbf{p}$.
We also select the same mini-batch size $b := \lfloor 0.5 n^{2/3} \rfloor$ in these algorithms. 
Note that if $n = 5000$, then $b = 150$ and $\mathbf{p} = 0.062$, but if $n=10000$, then $b = 239$ and $\mathbf{p} = 0.0479$. 

\vspace{0.5ex}
\textit{Results for the unconstrained case.}
In this test, we set $f = g = 0$ (i.e., without constraints) so that our model $0 \in Gx + Tx$ reduces to $Gx = 0$ as \eqref{eq:ME}.
We test $7$ algorithms in two sets of experiments: \textit{Experiments 1} and \textit{2} described above.
Figure \ref{fig:minimax_eq_ex1and2} reports the results of these experiments after $100$ epochs. 
Here, the number of epochs denotes the number of passes over all components $G_i$. 

\begin{figure}[ht!]
\vspace{-2ex}
\centering
\includegraphics[width=1\textwidth]{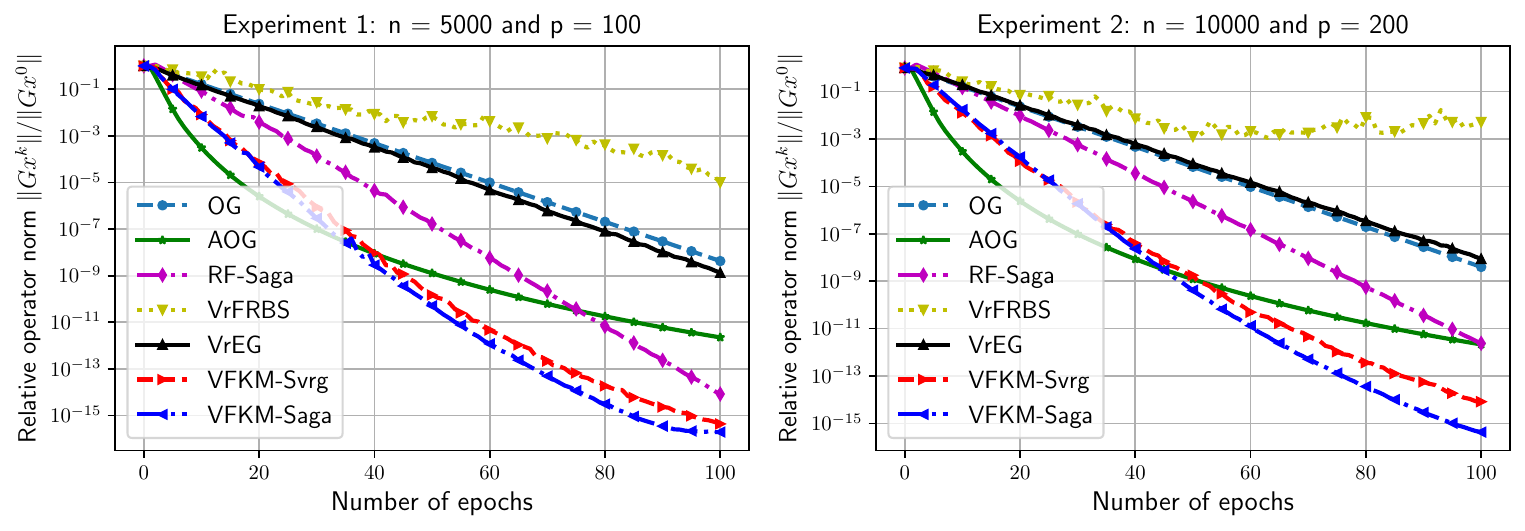}
\vspace{-3ex}
\caption{
Performance of the seven algorithms for \eqref{eq:minimax_exam2} in two experiments  $($averaged over $10$ problem instances$)$.
}
\label{fig:minimax_eq_ex1and2}
\vspace{-2ex}
\end{figure}

Figure~\ref{fig:minimax_eq_ex1and2} shows that our VFKM-Svrg and VFKM-Saga outperform their competitors down to a relative residual of approximately $10^{-15}$.
Since SAGA variants do not require full batch evaluations of $G$ at reference points as in SVRG schemes, they often need fewer epochs than SVRG.
However, as a trade-off, SAGA must store all $G_ix^k$ for $i \in [n]$.
This is also the case in our methods, where VFKM-Saga performs better than VFKM-Svrg in terms of epochs, but requires more memory to store its estimator.

RF-Saga works quite well in this case and exhibits approximately linear convergence in these experiments.
AOG performs well during the early epochs but is still worse than our methods and RF-Saga at the end.
VrEG performs similarly to OG, while VrFRBS seems to have the worst performance in this test.

Note that, although we did not tune the parameters in these experiments, in \textit{Experiment 2}, we can see that $\eta \approx 0.4032$ for VrEG, which is larger than $\beta \approx 0.2783$ in our methods.
For VrFRBS, we have $\eta \approx 0.1113$, which is five times the recommended stepsize.
If we increase it to $\eta \approx 0.4$, then VrFRBS occasionally diverges. 

\vspace{0.5ex}
\textit{Results for the constrained case.}
Finally, we test two variants of our method \eqref{eq:VRKM4NI} to solve \eqref{eq:MI} and compare them with the other methods as in the first test.
We test them with two sets of experiments as before and choose $f = \delta_{\Delta_{p_1}}$ and $g=\delta_{\Delta_{p_2}}$ defined above.
The results are reported in Figure~\ref{fig:minimax_ineq_ex1and2} after $100$ epochs.

\begin{figure}[ht!]
\vspace{-1.25ex}
\centering
\includegraphics[width=1\textwidth]{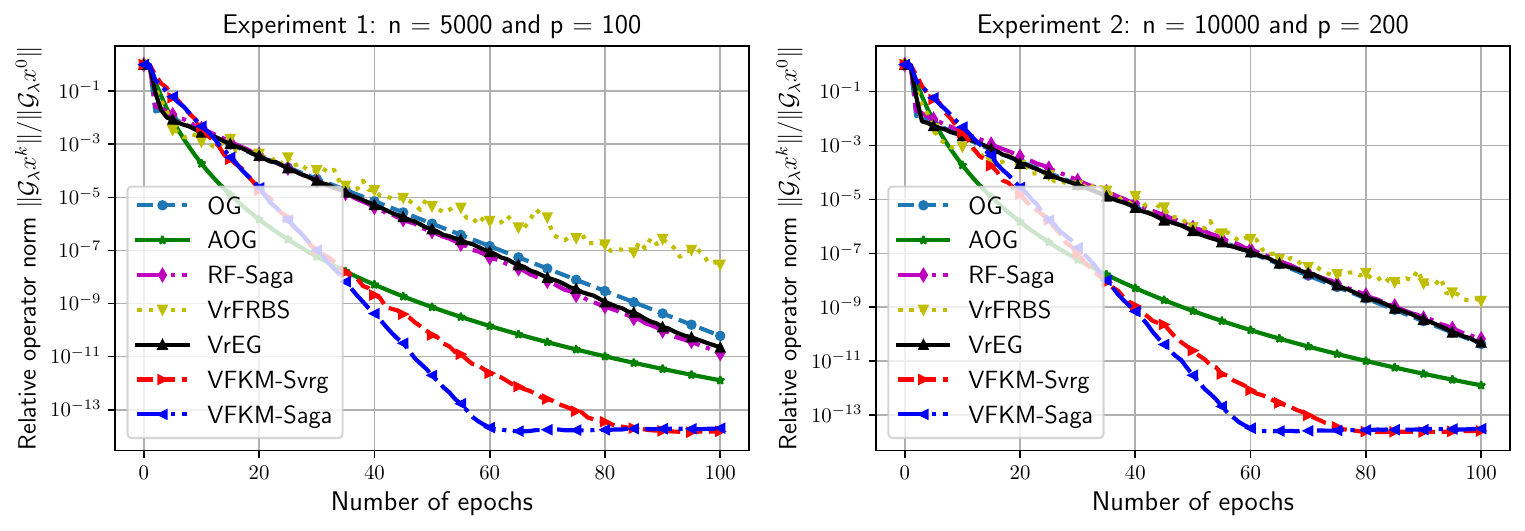}
\vspace{-3ex}
\caption{
Performance of the seven algorithms for \eqref{eq:minimax_exam2} in two experiments  $($averaged over $10$ problem instances$)$.
}
\label{fig:minimax_ineq_ex1and2}
\vspace{-0.5ex}
\end{figure}

From Figure~\ref{fig:minimax_ineq_ex1and2}, we still see that both VFKM-Svrg and VFKM-Saga perform well and outperform the competing methods.
Our VFKM-Saga still works best in this test.
However, RF-Saga  now performs similarly to both VrEG and OG.
AOG still works well and behaves similarly to the first test.
VrFRBS performs better than in the first test.
Note that the stepsize we choose for RF-Saga is $\eta = \frac{1}{4L}$, the same as VFKM-Saga and larger than $\frac{0.15}{L}$ in VFKM-Svrg.

\beforesec
\section*{Acknowledgments}
This work was partially supported by NSF RTG grant DMS-2134107 and ONR grant N00014-23-1-2588.
We are also grateful to the two anonymous reviewers for their invaluable comments and suggestions, which helped improve our paper.

\begin{APPENDICES}
\vspace{1ex}
\beforesec
\section{Preliminary Results.}\label{apdx:sec:tech_lemmas}
\aftersec
We first recall several technical results used in the paper and then provide the proof of Lemma \ref{le:BFS_oper_properties}.

\beforesubsec
\subsection{Technical lemmas.}\label{apdx:subsec:tech_lemmas}
\aftersubsec
We recall some known results used in this paper.

\begin{lemma}[\cite{Bauschke2011}, Lemma 5.31]\label{le:A1_descent}
Let $\sets{\xi_k}$, $\sets{\zeta_k}$, and $\sets{\varepsilon_k}$ be nonnegative sequences such that $\sum_{k=0}^{\infty} \varepsilon_k < +\infty$.
In addition, for all $k\geq 0$, we assume that
\myeqn{
\xi_{k+1} \leq \xi_k - \zeta_k + \varepsilon_k.
}
Then, we conclude that $\lim_{k\to\infty}\xi_k$ exists and $\sum_{k=0}^{\infty}\zeta_k < +\infty$.
\end{lemma}

\begin{lemma}[\cite{chambolle2015convergence}]\label{le:A2_sum}
Let $\sets{\xi_k}$ be nonnegative and $s \geq 0$ such that $\lim_{k\to\infty} k^{s + 1}\xi_k$ exists and $\sum_{k=0}^{\infty}k^{s}\xi_k < +\infty$.
Then, we conclude that $\lim_{k\to\infty}k^{s + 1}\xi_k = 0$.
\end{lemma}

\begin{lemma}[Supermartingale theorem, \cite{robbins1971convergence}]\label{le:RS_lemma}
Let $\sets{X_k}$, $\sets{\alpha_k}$, $\sets{V_k}$ and $\sets{R_k}$ be  sequences of nonnegative integrable random variables on some arbitrary probability space and adapted to the filtration $\set{\Fc_k}$ with $\sum_{k=0}^{\infty}\alpha_k < +\infty$ and $\sum_{k=0}^{\infty}R_k < +\infty$ almost surely, and 
\myeqn{ 
\Expn{ X_{k+1} \mid \Fc_k} \leq (1 + \alpha_k)X_k - V_k + R_k,  \quad \forall k \geq 0 \quad\textrm{almost surely}.
}
Then, $\sets{X_k}$ converges almost surely and $\sum_{k=0}^{\infty}V_k <+\infty$ almost surely.
\end{lemma}

The following lemma is Proposition 4.1 of \cite{davis2022variance}.
It was proven for a demiclosed mapping $G$, but we recall it here for the case where $G$ is continuous in a finite-dimensional space.

\begin{lemma}[Proposition 4.1 of \cite{davis2022variance}]\label{le:A3_lemma}
Suppose that $G$ in \eqref{eq:ME} is continuous.
Let $\sets{x^k}$ be a sequence of random vectors such that for all $x^{\star} \in \zer{G}$, the sequence $\sets{\norms{x^k - x^{\star}}^2}$ converges almost surely to a nonnegative random variable. 
In addition, assume that $\sets{\norms{Gx^k}}$ also almost surely converges to zero.
Then, $\sets{x^k}$ converges almost surely to a $\zer{G}$-valued random variable.
\end{lemma}

\beforesubsec
\subsection{\textbf{Proof of Lemma~\ref{le:BFS_oper_properties}.}}\label{apdx:subsec:proof_of_Lemma11}
(a)~The statement (a) was proven in \cite{attouch2018backward,davis2022variance}.

(b)~Let $A_{\lambda T}x := \frac{1}{\lambda}(x - J_{\lambda T}x)$ be the Moreau-Yosida approximation of $\lambda T$.
First, as shown in \cite{bauschke2020generalized}, $A_{\lambda T}$ is $(\lambda-\nu)$-cocoercive, provided that $\lambda > \nu$.
Next, since $J_{\lambda T}x = x - \lambda A_{\lambda T}x$ and $J_{\lambda T}y = y - \lambda A_{\lambda T}y$, by the $(\lambda-\nu)$-cocoercivity of $A_{\lambda T}$, we have
\myeqn{
\arraycolsep=0.2em
\begin{array}{lcl}
\norms{J_{\lambda T}x - J_{\lambda T}y}^2
&= & \norms{x - y}^2 - 2\lambda\iprods{A_{\lambda T}x - A_{\lambda T}y, x - y} + \lambda^2\norms{A_{\lambda T}x - A_{\lambda T}y}^2 \vspace{1ex}\\
& \leq & \norms{x - y}^2 - \lambda\big(\lambda - 2\nu)\norms{A_{\lambda T}x - A_{\lambda T}y}^2.
\end{array} 
}
Thus if $\lambda \geq 2\nu$, then $\norms{J_{\lambda T}x - J_{\lambda T}y} \leq \norms{x - y}$, implying that $J_{\lambda T}$ is nonexpansive.

(c)~We denote $u := J_{\lambda T}x$ and $v := J_{\lambda T}y$ for given $x, y \in \dom{J_{\lambda T}}$, where $\lambda > \nu$.
Note that $u$ and $v$ are well-defined due to Assumption~\ref{as:A4}, see \cite{bauschke2020generalized}.
Then, we have $A_{\lambda T}x = \frac{1}{\lambda}(x - J_{\lambda T}x) = \frac{1}{\lambda}(x - u)$.

Now, fix any $i \in [n]$.
From \eqref{eq:BFS_operator}, we have $A_{\lambda T}x = \Gc_{i,\lambda}x - G_iu$.
Since $A_{\lambda T}$ is $(\lambda-\nu)$-cocoercive and $A_{\lambda T}x = \Gc_{i,\lambda}x - G_iu$, we can show that
\myeqn{
\arraycolsep=0.2em
\begin{array}{lcl}
\iprods{\Gc_{i,\lambda}x - \Gc_{i,\lambda}y - (G_iu - G_iv), x - y} & \geq  & (\lambda - \nu)\norms{\Gc_{i,\lambda}x - \Gc_{i,\lambda}y - (G_iu - G_iv)}^2.
\end{array} 
}
Using again $u - \lambda G_iu = x - \lambda \Gc_{i,\lambda}x$ from \eqref{eq:BFS_operator}, the last inequality leads to 
\myeqn{
\arraycolsep=0.2em
\begin{array}{lcl}
\iprods{\Gc_{i,\lambda}x - \Gc_{i,\lambda}y, x - y} & \geq  &  (\lambda-\nu) \norms{\Gc_{i,\lambda}x - \Gc_{i,\lambda}y}^2  -  (\lambda - 2\nu)\iprods{\Gc_{i,\lambda}x - \Gc_{i,\lambda}y, G_iu - G_iv} \vspace{1ex}\\
&& + {~}  (\lambda-\nu)\norms{G_iu - G_iv}^2 + \iprods{G_iu - G_iv, x - y - \lambda (\Gc_{i,\lambda}x  -  \Gc_{i,\lambda}y)}\vspace{1ex}\\
& = & (\lambda - \nu) \norms{\Gc_{i,\lambda}x - \Gc_{i,\lambda}y}^2    -  (\lambda - 2\nu)\iprods{\Gc_{i,\lambda}x - \Gc_{i,\lambda}y, G_iu - G_iv} \vspace{1ex}\\
&& + {~}    \iprods{G_iu - G_iv, u - v} - \nu \norms{G_iu - G_iv}^2.
\end{array} 
}
Utilizing $L_g\iprods{G_iu - G_iv, u - v} \geq \norms{G_iu - G_iv}^2$ from the $\frac{1}{L_g}$-cocoercivity of $G_i$ in Assumption~\ref{as:A4}, we can derive that
\myeqn{
\arraycolsep=0.2em
\begin{array}{lcl}
\widetilde{\Tc}_{[1]} &:= & \iprods{\Gc_{i,\lambda}x - \Gc_{i,\lambda}y, x - y} \vspace{1ex}\\
& \geq  &   (\lambda - \nu)\norms{\Gc_{i,\lambda}x - \Gc_{i,\lambda}y}^2 + \big(\frac{1}{L_g} - \nu\big)\norms{G_iu - G_iv}^2 -  (\lambda - 2\nu) \iprods{\Gc_{i,\lambda}x - \Gc_{i,\lambda}y, G_iu - G_iv}  \vspace{1ex}\\
& = & \frac{4(1 - L_g\nu )(\lambda - \nu) - (\lambda - 2\nu)^2L_g}{4(1 -  L_g\nu )} \norms{\Gc_{i,\lambda}x - \Gc_{i,\lambda}y}^2  + \frac{(1- L_g\nu)}{L_g}\norms{G_iu - G_iv - \frac{(\lambda - 2\nu)L_g}{2(1- L_g\nu)}\big(\Gc_{i, \lambda}x - \Gc_{i, \lambda}y)}^2 \vspace{1ex}\\
& \geq &  \frac{4\lambda - 4\nu - L_g\lambda^2}{4(1 -  L_g\nu)}  \norms{\Gc_{i, \lambda}x - \Gc_{i, \lambda}y}^2.
\end{array} 
}
This proves that $\Gc_{i,\lambda}$ is $\frac{1}{L}$-cocoercive with $L :=  \frac{4(1 -  L_g\nu)}{\lambda(4 - L_g\lambda) - 4\nu}$. 
Note that the condition $\frac{2(1 - \sqrt{1 - L_g\nu})}{L_g} < \lambda < \frac{2(1 + \sqrt{1 - L_g\nu})}{L_g}$ guarantees that $L > 0$.
Finally, since $\Gc_{\lambda} = \frac{1}{n}\sum_{i=1}^n\Gc_{i,\lambda}$, it follows immediately that $\Gc_{\lambda}$ satisfies  Condition~\eqref{eq:co-coerciveness_of_G}.
\Eproof

\beforesec
\section{Proof of Technical Results in Section~\ref{sec:VRS_Estimator}.}\label{apdx:sec:VRS_Estimator}
\aftersec
This appendix provides the proof of Lemma~\ref{le:SVRG_estimator} and Lemma~\ref{le:SAGA_estimator} in Section~\ref{sec:VRS_Estimator}.
For convenience, let $\Fc_k$ be the $\sigma$-algebra generated by all randomness of \eqref{eq:VRKM4ME} before sampling $i_k$ and $\Bc_k$.
We also overload the notation $\Expsn{k}{\cdot} := \Expn{\cdot \mid \Fc_k}$.

\beforesubsec
\subsection{Proof of Lemma~\ref{le:SVRG_estimator}.}\label{apdx:le:SVRG_estimator}
\aftersubsec
For simplicity of presentation, let $i_k$ be a Bernoulli random variable reflecting the switching rule  \eqref{eq:SG_estimator_w}.
Let $\Expsn{i}{\cdot}$ and $\Expsn{\Bc_k}{\cdot}$ be the conditional expectations over $i$ and $\Bc_k$, respectively, conditioned on $\sigma(\Fc_k, i_k)$.

By the construction of $G_{\Bc_k}$ in \eqref{eq:SVRG_estimator}, we have 
$\Expsn{\Bc_k}{G_{\Bc_k}w^k} = Gw^k$, $\Expsn{\Bc_k}{G_{\Bc_k}x^k} = Gx^k$, and $\Expsn{\Bc_k}{G_{\Bc_k}x^{k-1} } = Gx^{k-1}$.
Utilizing these relations, from  \eqref{eq:SVRG_estimator}, we can show that $\Expsn{\Bc_k}{ \widetilde{S}^k} = S^k$.
Taking the conditional expectation w.r.t. $\Fc_k$ and using the tower property yields the first line of \eqref{eq:SVRG_vr_properties}.

Next, define $X_i^k := G_ix^k - \gamma_kG_ix^{k-1} - (1-\gamma_k)G_iw^k$ for all $i \in [n]$.
Then, we have $\Expsk{i}{X_i^k} = Gx^k - \gamma_k Gx^{k-1} -  (1-\gamma_k)Gw^k$ and $\Expsk{i}{\norms{X_i^k}^2} = \frac{1}{n}\sum_{j = 1}^n\norms{ G_jx^k - \gamma_k G_jx^{k-1} -  (1-\gamma_k)G_jw^k }^2$ for all $i \in [n]$.
Therefore, we can show that
\myeqn{
\arraycolsep=0.2em
\begin{array}{lcl}
\Expsn{\Bc_k}{ \norms{  \widetilde{S}^k - S^k }^2  } & \overset{\tiny\eqref{eq:SVRG_estimator}}{=} & 
\Expsk{\Bc_k}{\norms{ \frac{1}{b}\sum_{i\in\Bc_k}X_i^k  -  [ Gx^k - \gamma_k Gx^{k-1} -  (1-\gamma_k)Gw^k ] }^2 }   \vspace{1ex}\\ 
& = & 
\Expsk{\Bc_k}{\norms{ \frac{1}{b}\sum_{i\in\Bc_k}\big( X_i^k  -  \Expsn{i}{ X_i^k }  \big) }^2 }   \vspace{1ex}\\ 
&\overset{ \myeqc{1} }{ = } & \frac{1}{b^2} \sum_{i\in\Bc_k} \Expsn{i}{ \norms{X_i^k -  \mathbb{E}_i[ X_i^k] }^2 }    \vspace{1ex}\\ 
& \overset{\myeqc{2}}{=} & \frac{1}{b^2} \sum_{i\in\Bc_k}   \Expsn{i}{ \norms{X_i^k}^2 } - \frac{1}{b^2}\sum_{i\in\Bc_k} \norms{ \Expsn{i}{X_i^k }  }^2    \vspace{1ex}\\ 
& = & \frac{1}{n b}  \sum_{j = 1}^n \norms{  G_jx^k - \gamma_kG_jx^{k-1} - (1-\gamma_k)G_jw^k }^2   \vspace{1ex}\\ 
&& - {~} \frac{1}{b} \norms{ Gx^k - \gamma_k Gx^{k-1} -  (1-\gamma_k)Gw^k}^2 \vspace{1ex}\\
& \leq & \frac{1}{n b}  \sum_{i = 1}^n \norms{  G_ix^k - \gamma_kG_ix^{k-1} - (1-\gamma_k)G_iw^k }^2.
\end{array}
}
Here, $\myeqc{1}$ holds since the $b$ component indices in $B_k$ are sampled independently with replacement, and hence the corresponding sampled vectors $X_{i_{k,j}}^k$, $j \in[b]$, are conditionally i.i.d., and $\myeqc{2}$ holds due to $\Expsn{i}{\norms{X_i^k - \Expn{X_i^k}}^2} = \Expsn{i}{\norms{X_i^k}^2 } - \norms{ \Expsn{i}{ X_i^k } }^2$.
Taking the conditional expectation $\Expsn{k}{\cdot}$ on both sides of the last relation and using the definition \eqref{eq:SVRG_Deltak} of $\Delta_k$, we obtain the second line of \eqref{eq:SVRG_vr_properties}.

Now, from  \eqref{eq:SVRG_Deltak} and \eqref{eq:SG_estimator_w}, for any $c > 0$, by Young's inequality, we can show that
\myeqn{
\arraycolsep=0.2em
\hspace{-2ex}
\begin{array}{lcl}
\Expsn{i_k}{\Delta_k} & \overset{\eqref{eq:SVRG_Deltak}}{ =} &  \frac{1}{n b}  \sum_{i = 1}^n \Expsn{i_k}{  \norms{  G_ix^k - \gamma_kG_ix^{k-1} - (1-\gamma_k)G_iw^k }^2 }   \vspace{1ex}\\ 
& \overset{\eqref{eq:SG_estimator_w}}{ = } &  \frac{1 - \mbf{p}}{n b }  \sum_{i = 1}^n   \norms{  G_ix^k - \gamma_kG_ix^{k-1} - (1-\gamma_k)G_iw^{k-1} }^2 \vspace{1ex}\\ 
&& +  {~}  \frac{ \mbf{p}}{n b}  \sum_{i = 1}^n  \norms{  G_ix^k - \gamma_kG_ix^{k-1} - (1-\gamma_k)G_ix^{k-1} }^2    \vspace{1ex}\\ 
& \overset{\tiny\myeqc{1} }{\leq} &  \frac{(1+c)(1- \mbf{p} )}{nb} \frac{(1-\gamma_k)^2}{(1-\gamma_{k-1})^2}  \sum_{i = 1}^n   \norms{G_ix^{k-1} - \gamma_{k-1} G_ix^{k-2}   - (1-\gamma_{k-1})G_{i}w^{k-1} }^2   \vspace{1ex}\\ 
&& + {~} \frac{(1 + c)(1 - \mbf{p} )}{c n b}  \sum_{i = 1}^n \norms{   G_ix^k - \gamma_k G_ix^{k-1}  - \frac{1-\gamma_k}{1-\gamma_{k-1}} ( G_ix^{k-1} - \gamma_{k-1} G_ix^{k-2} ) }^2    \vspace{1ex}\\ 
&& + {~} \frac{ \mbf{p} }{ n b }  \sum_{i = 1}^n \norms{ G_ix^k - G_ix^{k-1} }^2    \vspace{1ex}\\ 
&\overset{\tiny\myeqc{2} }{ \leq } &  
  (1+c)(1- \mbf{p} )  \frac{(1-\gamma_k)^2}{(1-\gamma_{k-1})^2} \Delta_{k-1} +  \frac{1}{n b}\big[ \mbf{p} + \frac{2(1 + c)(1 - \mbf{p} )}{c} \big]   \sum_{i = 1}^n   \norms{ G_ix^k - G_ix^{k-1} }^2 \vspace{1ex}\\
&& + {~} \frac{2(1 + c)(1 - \mbf{p} ) \gamma_{k-1}^2(1-\gamma_k)^2 }{n b c (1-\gamma_{k-1})^2 }  \sum_{i = 1}^n  \norms{   G_ix^{k-1} - G_ix^{k-2} }^2.
\end{array}
\hspace{-2ex}
}
Here, we have used Young's inequality for both $\myeqc{1}$ and $\myeqc{2}$.
If we choose $c := \frac{ \mbf{p} }{2(1 - \mbf{p} )}$, then 
$(1+c)(1 - \mbf{p} ) = 1 - \frac{\mbf{p}}{2}$, 
$\frac{(1+c)(1 - \mbf{p})}{c}  = \frac{2 - 3 \mbf{p} + \mbf{p}^2}{ \mbf{p} }$, and $\frac{2(1+c)(1 - \mbf{p} )}{c} + \mbf{p}  = \frac{4 - 6 \mbf{p} + 3 \mbf{p}^2}{ \mbf{p} }$.
Hence, we obtain from the last inequality that
\myeqn{
\arraycolsep=0.2em
\begin{array}{lcl}
\Expsn{i_k}{\Delta_k}  & \leq &  \big(1 -  \frac{\mbf{p}}{2} \big) \frac{(1-\gamma_k)^2}{(1-\gamma_{k-1})^2} \Delta_{k-1} 
+ \frac{4 - 6\mbf{p} + 3\mbf{p}^2 }{ nb  \mbf{p} } \sum_{i=1}^n   \norms{G_ix^k - G_ix^{k-1}}^2  \vspace{1ex}\\
&& + {~}  \frac{2(2 - 3\mbf{p} + \mbf{p}^2) \gamma_{k-1}^2(1 - \gamma_k)^2}{ nb \mbf{p} (1-\gamma_{k-1})^2 } \sum_{i=1}^n \norms{G_ix^{k-1} - G_ix^{k-2} }^2.
\end{array}
}
Since $\Delta_k$ depends on the fresh randomness only through $i_k$, we have $\mathbb{E}_{i_k}[\Delta_k \mid \Fc_k ] = \mathbb{E}_k[ \Delta_k ]$. 
Dividing the preceding inequality by $(1-\gamma_k)^2$ yields the last inequality of \eqref{eq:SVRG_vr_properties}.

Finally, if we denote $\Theta :=  \frac{4 - 6\mbf{p} + 3\mbf{p}^2 }{b  \mbf{p} }$ and $\hat{\Theta} :=  \frac{2(2 - 3\mbf{p} + \mbf{p}^2)}{b \mbf{p}}$, then since $\gamma_{k-1} \in [0, 1)$, we can easily show that $\widetilde{S}^k$ generated by \eqref{eq:SVRG_estimator} satisfies Definition~\ref{de:ub_SG_estimator} with $\rho := \frac{\mbf{p}}{2}$.
\Eproof

\beforesubsec
\subsection{Proof of Lemma~\ref{le:SAGA_estimator}.}\label{apdx:le:SAGA_estimator}
\aftersubsec
For simplicity of presentation, let $\Expsn{i}{\cdot}$ denote the expectation over a uniformly sampled index $i\in[n]$, conditioned on $\Fc_k$ and $\widehat{\Bc}_k$, let $\Expsn{\Bc_k}{\cdot}$ denote the expectation over $\Bc_k$, conditioned on $\Fc_k$ and $\widehat{\Bc}_k$, and let $\Expsn{\widehat{\Bc}_k}{\cdot}$ denote the expectation over $\widehat{\Bc}_k$, conditioned on $\Fc_k$.

First, conditioned on $\Fc_k$ and $\widehat{\Bc}_k$, the table $\{\widehat{G}_i^k\}_{i=1}^n$ is fixed. Since $\Bc_k$ is sampled independently of $\widehat{\Bc}_k$, we have $\Expsn{\Bc_k}{G_{\Bc_k}x^k} = Gx^k$, $\Expsn{\Bc_k}{G_{\Bc_k}x^{k-1}} = Gx^{k-1}$, and $\Expsn{\Bc_k}{\widehat{G}_{\Bc_k}^k} = \frac{1}{n}\sum_{i=1}^n\widehat{G}_i^k$.
Using these relations and \eqref{eq:SAGA_estimator}, we can easily show that $\Expsn{\Bc_k}{\widetilde{S}^k} = S^k$.
Taking $\Expsn{k}{\cdot}$ on both sides of this relation and using the tower property, we prove the first line of \eqref{eq:SAGA_vr_properties}.

Next, we define $X_i^k := G_ix^k - \gamma_kG_ix^{k-1} - (1-\gamma_k)\widehat{G}_i^k$ for all $i \in [n]$.
Then, conditioned on $\Fc_k$ and $\widehat{\Bc}_k$, we have $\Expsn{i}{X_i^k} = Gx^k - \gamma_k Gx^{k-1} -  \frac{1-\gamma_k}{n}\sum_{j=1}^n\widehat{G}^k_j$ and $\Expsn{i}{\norms{X_i^k}^2 } = \frac{1}{n}\sum_{j =1}^n \norms{ G_jx^k - \gamma_k G_jx^{k-1} -  (1-\gamma_k) \widehat{G}^k_j}^2$.
Therefore, we can derive
\myeqn{
\arraycolsep=0.2em
\begin{array}{lcl}
\Expsn{\Bc_k}{ \norms{  \widetilde{S}^k - S^k }^2 }  & = & \Expsn{\Bc_k}{\norms{ \frac{1}{b}\sum_{i\in\Bc_k}X_i^k  -  ( Gx^k - \gamma_k Gx^{k-1} )  +  \frac{1-\gamma_k}{n}\sum_{j=1}^n \widehat{G}^k_j  }^2 }   \vspace{1ex}\\ 
& = & \Expsn{\Bc_k}{\norms{ \frac{1}{b}\sum_{i\in\Bc_k} ( X_i^k  -  \Expsn{i}{ X_i^k } ) }^2 }   \vspace{1ex}\\ 
& \overset{\myeqc{1}}{ = } & \frac{1}{b} \Expsn{i}{ \norms{X_i^k -  \Expsn{i}{ X_i^k } }^2 }   \vspace{1ex}\\ 
& \overset{\myeqc{2}}{ = } & \frac{1}{b} \Expsn{i}{  \norms{  X_i^k }^2 } - \frac{1}{b} \norms{ \Expsn{i}{X_i^k} }^2   \vspace{1ex}\\ 
& = & \frac{1}{n b}  \sum_{j  = 1}^n \norms{  G_jx^k - \gamma_kG_jx^{k-1} - (1-\gamma_k) \widehat{G}_j^k  }^2   \vspace{1ex}\\ 
&& - {~} \frac{1}{b} \norms{ Gx^k - \gamma_k Gx^{k-1} -  \frac{1-\gamma_k}{n}\sum_{j=1}^n\widehat{G}^k_j }^2 \vspace{1ex}\\
& \leq &  \frac{1}{n b}  \sum_{i  = 1}^n \norms{  G_ix^k - \gamma_kG_ix^{k-1} - (1-\gamma_k) \widehat{G}_i^k  }^2.
\end{array}
}
Here, $\myeqc{1}$ holds since the $b$ component indices in $\Bc_k$ are sampled independently with replacement, and hence the corresponding sampled vectors are conditionally i.i.d., while $\myeqc{2}$ follows from the variance identity
\begin{equation*}
\Expsn{i}{\norms{X_i^k-\Expsn{i}{X_i^k}}^2} = \Expsn{i}{\norms{X_i^k}^2} - \norms{\Expsn{i}{X_i^k}}^2.
\end{equation*}
Taking the conditional expectation $\Expsn{k}{\cdot}$ and using \eqref{eq:SAGA_Deltak} yields the second line of \eqref{eq:SAGA_vr_properties}.

Now, from \eqref{eq:SAGA_Deltak} and \eqref{eq:SAGA_ref_points}, for any $c > 0$, we can show that
\myeqn{
\arraycolsep=0.2em
\begin{array}{lcl}
\Expsn{\widehat{\Bc}_k}{\Delta_k} & \overset{\eqref{eq:SAGA_Deltak}}{ = } &  \frac{1}{n b}  \sum_{i = 1}^n \Expsn{\widehat{\Bc}_k}{  \norms{  G_ix^k - \gamma_kG_ix^{k-1} - (1-\gamma_k) \widehat{G}_i^k }^2 }   \vspace{1ex}\\ 
& \overset{ \eqref{eq:SAGA_ref_points} }{ = } & \big(1 - \frac{b}{n}\big)  \frac{1}{n b}  \sum_{i = 1}^n  \norms{  G_ix^k - \gamma_kG_ix^{k-1} - (1-\gamma_k) \widehat{G}_i^{k-1} }^2    \vspace{1ex}\\ 
&& + {~}  \frac{b}{n}   \frac{1}{n b}  \sum_{i = 1}^n  \norms{  G_ix^k - \gamma_kG_ix^{k-1} - (1-\gamma_k)G_ix^{k-1} }^2    \vspace{1ex}\\ 
&\overset{\tiny\myeqc{1}}{ \leq } &   \frac{(1+c)(1-\gamma_k)^2}{ n b(1-\gamma_{k-1})^2}\big(1 - \frac{b}{n}\big)  \sum_{i = 1}^n  \norms{  G_ix^{k-1} - \gamma_{k-1}G_ix^{k-2} - (1-\gamma_{k-1}) \widehat{G}_i^{k-1} }^2   \vspace{1ex}\\ 
&& + {~} \frac{1+c}{c n b}\big(1 - \frac{b}{n}\big)  \sum_{i = 1}^n \norms{  G_ix^k - \gamma_kG_ix^{k-1} - \frac{1-\gamma_k}{1-\gamma_{k-1}}( G_ix^{k-1} - \gamma_{k-1} G_ix^{k-2}) }^2    \vspace{1ex}\\ 
&& + {~}  \frac{1}{n^2}   \sum_{i = 1}^n   \norms{  G_ix^k -  G_ix^{k-1} }^2    \vspace{1ex}\\ 
&\overset{\tiny\myeqc{2} }{ \leq } & \frac{(1+c)(1-\gamma_k)^2}{(1-\gamma_{k-1})^2 }\big(1 - \frac{b}{n}\big) \Delta_{k-1} +   \big[ \frac{1}{n^2} + \big(1 - \frac{b}{n}\big)\frac{2(1+c)}{c n b} \big]  \sum_{i = 1}^n \norms{  G_ix^k - G_ix^{k-1} }^2    \vspace{1ex}\\ 
&& +  {~} \frac{2(1+c)\gamma_{k-1}^2(1-\gamma_k)^2}{c n b(1-\gamma_{k-1})^2}\big(1 - \frac{b}{n}\big)  \sum_{i = 1}^n   \norms{  G_ix^{k-1} - G_ix^{k-2}  }^2, 
\end{array}
}
where $\myeqc{1}$ and $\myeqc{2}$ follow from Young's inequality, and $\Expsn{\widehat{\Bc}_k}{\cdot}$ is the expectation associated with the update rule \eqref{eq:SAGA_ref_points}. In particular, since $\widehat{\Bc}_k$ is sampled uniformly without replacement with $|\widehat{\Bc}_k|=b$, we have $\mathbb{P}(i\in\widehat{\Bc}_k\mid\Fc_k)=\frac{b}{n}$ for $i\in[n]$.
If we choose $c := \frac{b}{2n}$, then $(1 - \frac{b}{n})(1+c) = 1 - \frac{b}{2n} - \frac{b^2}{2n^2} \leq 1 - \frac{b}{2n}$.
Hence, we can upper bound the last inequality as follows:
\myeqn{
\arraycolsep=0.2em
\begin{array}{lcl}
\Expsn{\widehat{\Bc}_k}{\Delta_k} & \leq &  \big(1 -  \frac{b}{2n} \big) \frac{(1-\gamma_k)^2}{(1-\gamma_{k-1})^2} \Delta_{k-1}  +   \frac{2(n-b)(b +2n) + b^2}{ n^2b^2 } \sum_{i=1}^n  \norms{G_ix^k - G_ix^{k-1}}^2    \vspace{1ex}\\ 
&& + {~}  \frac{2(n-b)(b+2n) \gamma_{k-1}^2(1 - \gamma_k)^2}{ n^2b^2 (1-\gamma_{k-1})^2 } \sum_{i=1}^n \norms{G_ix^{k-1} - G_ix^{k-2} }^2.
\end{array}
}
Since $\Delta_k$ depends on the fresh randomness at iteration $k$ only through $\widehat{\Bc}_k$, we have $\Expsn{\widehat{\Bc}_k}{\Delta_k} = \Expsn{k}{\Delta_k}$.
Therefore, dividing the preceding inequality by $(1-\gamma_k)^2$ and noting that $\gamma_{k-1}^2 \leq 1$, we get the last inequality of \eqref{eq:SAGA_vr_properties}, where 
$\Theta :=  \frac{2(n-b)(b +2n) + b^2}{ n b^2 }$ and $\hat{\Theta} := \frac{2(n-b)(b+2n)}{ n b^2}$.

Finally, using \eqref{eq:SAGA_vr_properties}, one can easily check that $\widetilde{S}^k$ generated by \eqref{eq:SAGA_estimator} satisfies Definition~\ref{de:ub_SG_estimator} with $\rho := \frac{b}{2n} \in (0, 1)$.
\Eproof

\beforesec
\section{Proof of Theoretical Results for \eqref{eq:VRKM4ME} in Section~\ref{sec:AVFR4NE}.}\label{apdx:sec:AVFR4NE}
\aftersec
We prove Lemma~\ref{le:AVFR4NE_key_estimate} and provide technical lemmas for proving Theorems~\ref{th:AVFR4NE_convergence} and \ref{th:AVFR4NE_as_convergence}.

\beforesubsec
\subsection{Proof of Lemma~\ref{le:AVFR4NE_key_estimate}.}\label{apdx:le:AVFR4NE_key_estimate}
\aftersubsec
Fix $x^{\star} \in \zer{G}$.
We first introduce the following functions:
\myeq{eq:AVFR4NE_lm2_proof1}{
\arraycolsep=0.2em
\begin{array}{lcl}
\Qc_k &:= &  \norms{ r( x^k - x^{\star} ) + t_{k+1}(x^{k+1} - x^k) }^2 + \mu r \norms{x^k - x^{\star}}^2,   \vspace{1ex}\\ 
\Lc_k & := & 4r\beta (t_k - \mu)  \big[ \iprods{Gx^k, x^k - x^{\star}} - \beta\norms{Gx^k}^2 \big] + \Qc_k.
\end{array}
}
Then, we can split the proof of Lemma~\ref{le:AVFR4NE_key_estimate} into two steps in Lemma~\ref{le:bound_Qk} and Lemma~\ref{le:bound_Lk}.

\begin{lemma}\label{le:bound_Qk}
Let us choose parameters $t_k$, $\theta_k$, $\gamma_k$, $\eta_k$, and $\beta$ such that
\myeq{eq:AVFR4NE_lm2_para_conds}{
\arraycolsep=0.2em
\begin{array}{ll}
t_k - r -  t_{k+1}  \theta_k - \mu = 0,  \ \  (1-\gamma_k) t_{k+1}\theta_k - r \gamma_k  = 0, \ \ \text{and} \ \ 2\beta(t_k - \mu) = \eta_kt_{k+1}.
\end{array}
\hspace{-3ex}
}
Then,  for $\Qc_k$ defined by \eqref{eq:AVFR4NE_lm2_proof1} and $e^k := \widetilde{S}^k - S^k$, we have
\myeq{eq:AVROG_descent2}{
\arraycolsep=0.2em
\begin{array}{lcl}
\Qc_{k-1} - \Expsk{k}{\Qc_k } 
& = &  \big[ (t_k - r)^2 - t_{k+1}^2\theta_k^2 + \mu r \big]\norms{x^k - x^{k-1}}^2 - \eta_k^2t_{k+1}^2 \Expsk{k}{ \norms{e^k}^2 }    \vspace{1ex}\\ 
&& + {~} 2 r \eta_k t_{k+1}  \big[  \iprods{Gx^k, x^k - x^{\star}} - \beta\norms{Gx^k}^2 \big]   \vspace{1ex}\\ 
&& - {~} 2 r \gamma_k \eta_k t_{k+1}  \big[ \iprods{Gx^{k-1}, x^{k-1} - x^{\star}}  - \beta\norms{Gx^{k-1}}^2 \big]   \vspace{1ex}\\ 
&& +  {~} 2\eta_k t_{k+1}^2\theta_k \iprods{Gx^k - Gx^{k-1}, x^k - x^{k-1}}   -  \gamma_k \eta_k^2t_{k+1}^2 \norms{Gx^k - Gx^{k-1}}^2.
\end{array}
}
\end{lemma}

\proof{\textbf{Proof.}}
First, a direct expansion gives
\myeqn{
\arraycolsep=0.2em
\begin{array}{lcl}
\Tc_{[1]} &:= &  \norms{ r ( x^{k-1} - x^{\star}) + t_k(x^k - x^{k-1}) }^2 =  \norms{r (x^k - x^{\star}) + (t_k - r)(x^k - x^{k-1})}^2   \vspace{1ex}\\ 
& = & r^2 \norms{x^k - x^{\star}}^2 + (t_k - r)^2\norms{x^k - x^{k-1}}^2 + 2r (t_k -r)\iprods{x^k - x^{k-1}, x^k - x^{\star}}.
\end{array}
}
Alternatively, using \eqref{eq:VRKM4ME}, we can expand 
\myeqn{
\hspace{-0.5ex}
\arraycolsep=-0.05em
\begin{array}{lcl}
\Tc_{[2]} &:= &  \norms{ r ( x^{k} - x^{\star} ) + t_{k+1}(x^{k+1} - x^{k}) }^2   \vspace{1ex}\\ 
& \overset{\tiny\eqref{eq:VRKM4ME}}{ = } & \norms{ r( x^k - x^{\star} ) + t_{k+1}\theta_k(x^k - x^{k-1}) - \eta_k t_{k+1}\widetilde{S}^k}^2   \vspace{1ex}\\ 
& = & r^2 \norms{x^k - x^{\star}}^2 + t_{k+1}^2\theta_k^2\norms{x^k - x^{k-1}}^2 + \eta_k^2t_{k+1}^2\norms{\widetilde{S}^k }^2 - 2 r \eta_k t_{k+1}\iprods{\widetilde{S}^k, x^k - x^{\star}}   \vspace{1ex}\\ 
&& + {~} 2 r t_{k+1}\theta_k\iprods{x^k - x^{k-1}, x^k - x^{\star}} - 2\eta_k t_{k+1}^2\theta_k\iprods{\widetilde{S}^k, x^k - x^{k-1}}.
\end{array}
\hspace{-2ex}
}
Moreover, for any $\mu > 0$, we also have 
\myeqn{
\arraycolsep=0.2em
\begin{array}{lcl}
\mu r \norms{x^{k-1} - x^{\star}}^2 - \mu r \norms{x^k - x^{\star}}^2 = \mu r \norms{x^k - x^{k-1}}^2 -  2\mu r  \iprods{x^k - x^{\star}, x^k - x^{k-1}}.
\end{array}
}
Combining the three last expressions, and using $\Qc_k$ from \eqref{eq:AVFR4NE_lm2_proof1}, we get
\myeqn{
\arraycolsep=0.2em
\begin{array}{lcl}
\Qc_{k-1} - \Qc_k & = &  \big[ (t_k - r)^2 - t_{k+1}^2\theta_k^2 + \mu r \big]\norms{x^k - x^{k-1}}^2 -  \eta_k^2t_{k+1}^2\norms{  \widetilde{S}^k }^2   \vspace{1ex}\\ 
&& + {~} 2r (t_k - r -   t_{k+1}\theta_k - \mu )\iprods{x^k - x^{k-1}, x^k - x^{\star}}     \vspace{1ex}\\ 
&& + {~} 2 r \eta_k t_{k+1}\iprods{\widetilde{S}^k, x^k - x^{\star}} +  2\eta_k t_{k+1}^2\theta_k \iprods{\widetilde{S}^k, x^k - x^{k-1}}.
\end{array}
}
Since $\Qc_{k-1}$, $x^{k-1}$, and $x^k$ are $\Fc_k$-measurable, taking the conditional expectation $\Expsn{k}{\cdot}$ on both sides of this expression, we obtain
\myeq{eq:AVFR4NE_lm2_proof2}{
\hspace{-3ex}
\arraycolsep=0.2em
\begin{array}{lcl}
\Qc_{k-1} - \Expsn{k}{\Qc_k } & = &  \big[ (t_k - r)^2 - t_{k+1}^2 \theta_k^2 + \mu r \big]\norms{x^k - x^{k-1}}^2 - \eta_k^2t_{k+1}^2 \Expsn{k}{ \norms{  \widetilde{S}^k }^2 }   \vspace{1ex}\\ 
&& + {~} 2r (t_k - r - t_{k+1}\theta_k - \mu )\iprods{x^k - x^{k-1}, x^k - x^{\star}}   \vspace{1ex}\\ 
&& + {~}  2 r \eta_k t_{k+1} \Expsk{k}{ \iprods{\widetilde{S}^k, x^k - x^{\star}} } + 2\eta_k t_{k+1}^2\theta_k \Expsk{k}{ \iprods{\widetilde{S}^k, x^k - x^{k-1}} }.
\end{array}
\hspace{-4ex}
}
Since  $S^k = Gx^k - \gamma_k Gx^{k-1}$ and $e^k = \widetilde{S}^k - S^k$, we have $\widetilde{S}^k = S^k + e^k$ and  $\Expsn{k}{e^k} = 0$ as $\widetilde{S}^k$ is an unbiased estimator of $S^k$ stated in  \eqref{eq:ub_SG_estimator}.
Therefore, we can show that
\myeqn{
\hspace{-1ex}
\arraycolsep=0.2em
\begin{array}{lcl}
\Expsk{k}{\iprods{\widetilde{S}^k, x^k - x^{\star}}}  & = &  \iprods{S^k, x^k - x^{\star}} + \iprods{\Expsn{k}{e^k}, x^k - x^{\star}}  =  \iprods{S^k, x^k - x^{\star}}   \vspace{1ex}\\ 
&= &  \iprods{Gx^k, x^k - x^{\star}} - \gamma_k \iprods{Gx^{k-1}, x^k - x^{\star}}   \vspace{1ex}\\ 
& = &  \iprods{Gx^k, x^k - x^{\star}} - \gamma_k \iprods{Gx^{k-1}, x^{k-1} - x^{\star}}  - \gamma_k \iprods{Gx^{k-1}, x^k - x^{k-1}}.
\end{array}
\hspace{-2ex}
}
Similarly, we can also derive that
\myeqn{
\arraycolsep=0.2em
\begin{array}{lcl}
\Expsk{k}{ \iprods{ \widetilde{S}^k, x^k - x^{k-1}} } & = &  \iprods{S^k, x^k - x^{k-1}} + \iprods{\Expsn{k}{e^k}, x^k - x^{k-1}}  =   \iprods{S^k, x^k - x^{k-1} }   \vspace{1ex}\\ 
&= & \iprods{Gx^k - Gx^{k-1}, x^k - x^{k-1}} + (1- \gamma_k)\iprods{Gx^{k-1}, x^k - x^{k-1}}.
\end{array}
}
Again, since $\widetilde{S}^k$ is an unbiased estimator of $S^k$ as stated in \eqref{eq:ub_SG_estimator}, and $e^k = \widetilde{S}^k - S^k$, we have $\Expsk{k}{ \norms{\widetilde{S}^k}^2 } = \Expsk{k}{ \norms{S^k + e^k}^2 } = \norms{S^k}^2 + 2 \Expsk{k}{\iprods{e^k, S^k}} + \Expsk{k}{ \norms{e^k}^2 } = \Expsk{k}{ \norms{e^k}^2 } + \norms{S^k}^2$.
Therefore, we can easily show that
\myeqn{
\arraycolsep=0.2em
\begin{array}{lcl}
\Expsk{k}{ \norms{\widetilde{S}^k}^2 } &= & \norms{S^k}^2 + \Expsk{k}{ \norms{e^k}^2 }  = \norms{Gx^k - \gamma_k Gx^{k-1}}^2 + \Expsk{k}{ \norms{e^k}^2 }   \vspace{1ex}\\ 
&= & \norms{Gx^k}^2 - 2\gamma_k\iprods{Gx^k, Gx^{k-1}} + \gamma_k^2\norms{Gx^{k-1}}^2 + \Expsk{k}{ \norms{e^k}^2 }   \vspace{1ex}\\ 
&= & (1 - \gamma_k)\norms{Gx^k}^2 - \gamma_k(1-\gamma_k) \norms{Gx^{k-1}}^2   \vspace{1ex}\\ 
&& + {~} \gamma_k\norms{Gx^k - Gx^{k-1}}^2  + \Expsk{k}{ \norms{e^k}^2 }.
\end{array}
}
Substituting the last three expressions into \eqref{eq:AVFR4NE_lm2_proof2}, we can derive that
\myeqn{
\hspace{-1ex}
\arraycolsep=0.2em
\begin{array}{lcl}
\Qc_{k-1} - \Expsk{k}{\Qc_k } 
& = &  \big[ (t_k - r)^2 - t_{k+1}^2\theta_k^2 + \mu r \big]\norms{x^k - x^{k-1}}^2 - \eta_k^2t_{k+1}^2 \Expsk{k}{ \norms{e^k}^2 }    \vspace{1ex}\\ 
&& + {~} 2 r \eta_k t_{k+1}  \big[  \iprods{Gx^k, x^k - x^{\star}} - \beta\norms{Gx^k}^2 \big]   \vspace{1ex}\\ 
&& - {~} 2 r \gamma_k \eta_k t_{k+1}  \big[ \iprods{Gx^{k-1}, x^{k-1} - x^{\star}}  - \beta\norms{Gx^{k-1}}^2 \big]   \vspace{1ex}\\ 
&& +  {~} 2\eta_k t_{k+1}^2\theta_k \iprods{Gx^k - Gx^{k-1}, x^k - x^{k-1}}   -  \gamma_k \eta_k^2t_{k+1}^2 \norms{Gx^k - Gx^{k-1}}^2 \vspace{1ex}\\ 
&& +  {~} 2r\eta_kt_{k+1}\big[ \beta -   \frac{(1- \gamma_k)\eta_kt_{k+1}}{2r} \big] \norms{Gx^k}^2 \vspace{1ex}\\
&& - {~}  2r\gamma_k \eta_kt_{k+1}\big[   \beta - \frac{(1- \gamma_k)\eta_kt_{k+1}}{2r} ] \norms{Gx^{k-1}}^2   \vspace{1ex}\\ 
&& + {~} 2r (t_k - r -   t_{k+1} \theta_k  - \mu )\iprods{x^k - x^{k-1}, x^k - x^{\star}}   \vspace{1ex}\\ 
&& + {~} 2\eta_kt_{k+1}\big[ (1-\gamma_k) t_{k+1}\theta_k - r \gamma_k \big] \iprods{Gx^{k-1}, x^k - x^{k-1}}.
\end{array}
\hspace{-2ex}
}
Under the conditions in \eqref{eq:AVFR4NE_lm2_para_conds}, we have  $t_k - r -   t_{k+1} \theta_k  - \mu = 0$, $(1-\gamma_k) t_{k+1}\theta_k - r \gamma_k = 0$, and $\frac{(1-\gamma_k)\eta_kt_{k+1}}{2r} = \frac{\eta_kt_{k+1}}{2(t_{k+1}\theta_k + r)} = \frac{\eta_kt_{k+1}}{2(t_k - \mu)} = \beta$.
Substituting these identities into the above expression, the last four terms vanish.
Thus, we obtain \eqref{eq:AVROG_descent2}.
\Eproof
\endproof

\begin{lemma}\label{le:bound_Lk}
If $t_k$, $\theta_k$, $\gamma_k$, and $\eta_k$ are updated as in \eqref{eq:AVFR4NE_para_update}, then we have
\myeq{eq:AVFR4NE_lm2_proof4}{
\arraycolsep=0.2em
\begin{array}{lcl}
\Lc_{k-1} -   \Expsk{k}{ \Lc_k }  & \geq &  \mu (2t_k - r - \mu)  \norms{x^k - x^{k-1}}^2 - 4\beta^2 (t_k - \mu)^2 \Expsk{k}{ \norms{e^k}^2 } \vspace{1ex}\\ 
&& + {~} 4r \beta (t_{k-1} - t_k + r) \big[  \iprods{Gx^{k-1}, x^{k-1} - x^{\star}}  - \beta \norms{Gx^{k-1}}^2 \big]   \vspace{1ex}\\ 
&& + {~} 4\beta (t_k -  \mu)(t_k - r - \mu) \big(\frac{1}{L} - \beta\big) U_k,
\end{array}
}
where $\Lc_k$ is defined by \eqref{eq:AVFR4NE_lm2_proof1} and $U_k := \frac{1}{n}\sum_{i=1}^n  \norms{G_ix^k - G_ix^{k-1}}^2$.
\end{lemma}

\proof{\textbf{Proof.}}
A direct calculation shows that the update rules in \eqref{eq:AVFR4NE_para_update} guarantee the three conditions in \eqref{eq:AVFR4NE_lm2_para_conds}.
Utilizing $\theta_k$, $\gamma_k$, and $\eta_k$ from \eqref{eq:AVFR4NE_para_update} we can rewrite \eqref{eq:AVROG_descent2} as follows:
\myeqn{
\arraycolsep=0.2em
\begin{array}{lcl}
\Qc_{k-1} - \Expsk{k}{\Qc_k } & = &   \mu(2t_k - r - \mu) \norms{x^k - x^{k-1}}^2  - 4\beta^2 (t_k - \mu)^2 \Expsk{k}{ \norms{e^k}^2 }    \vspace{1ex}\\ 
&& +  {~} 4 r \beta (t_k-\mu) \big[  \iprods{Gx^k, x^k - x^{\star}} -   \beta \norms{Gx^k}^2 \big]   \vspace{1ex}\\ 
&& - {~} 4 r \beta (t_k - r - \mu) \big[  \iprods{Gx^{k-1}, x^{k-1} - x^{\star}}  - \beta \norms{Gx^{k-1}}^2 \big]   \vspace{1ex}\\ 
&& + {~} 4\beta (t_k-\mu)(t_k - r -  \mu )   \iprods{Gx^k  - Gx^{k-1}, x^k - x^{k-1}}   \vspace{1ex}\\ 
&& - {~} 4\beta^2 (t_k-\mu)(t_k - r -  \mu )    \norms{Gx^k - Gx^{k-1}}^2.
\end{array}
}
Rearranging the last expression, and then using $\Lc_k$ from \eqref{eq:AVFR4NE_lm2_proof1}, we can show that
\myeqn{
\arraycolsep=0.2em
\begin{array}{lcl}
\Lc_{k-1}  -  \Expsk{k}{ \Lc_k } & = &  \mu(2t_k - r - \mu) \norms{x^k - x^{k-1}}^2   - 4\beta^2 (t_k - \mu)^2  \Expsk{k}{ \norms{e^k}^2 }   \vspace{1ex}\\ 
&& + {~} 4r \beta (t_{k-1} - t_k + r) \big[  \iprods{Gx^{k-1}, x^{k-1} - x^{\star}}  - \beta \norms{Gx^{k-1}}^2 \big]   \vspace{1ex}\\ 
&& + {~} 4\beta (t_k - \mu)(t_k - r - \mu)   \iprods{Gx^k  - Gx^{k-1}, x^k - x^{k-1}}     \vspace{1ex}\\ 
&& - {~}  4\beta^2 (t_k - \mu)(t_k - r - \mu) \norms{Gx^k - Gx^{k-1}}^2.
\end{array}
}
By Condition~\eqref{eq:co-coerciveness_of_G} and Jensen's inequality in $\myeqc{1}$, we have
\myeq{eq:AVFR4NE_lm2_A2_cond}{
\hspace{-2ex}
\arraycolsep=0.2em
\begin{array}{llcl}
& \iprods{Gx^k  - Gx^{k-1}, x^k - x^{k-1}}   & \overset{\eqref{eq:co-coerciveness_of_G}}{ \geq } &  \frac{1}{n L}\sum_{i=1}^n\norms{G_ix^k - G_ix^{k-1}}^2 = \frac{1}{L} U_k,   \vspace{1ex}\\ 
& - \norms{Gx^k - Gx^{k-1}}^2 & \overset{\myeqc{1}}{ \geq } & - \frac{1}{n}\sum_{i=1}^n\norms{G_ix^k - G_ix^{k-1}}^2  = - U_k.
\end{array}
\hspace{-2ex}
}
Applying  \eqref{eq:AVFR4NE_lm2_A2_cond}, we can lower bound the last estimate to obtain \eqref{eq:AVFR4NE_lm2_proof4}.
\Eproof
\endproof

\vspace{1ex}
\beforesubsec
\subsection*{\textbf{Proof of Lemma~\ref{le:AVFR4NE_key_estimate}.}}
\aftersubsec
First, since $\gamma_k = \frac{ t_k - r -  \mu}{t_k - \mu}$, we have $(1-\gamma_k)^2 = \frac{r^2}{(t_k - \mu)^2}$.
From \eqref{eq:ub_SG_estimator}, $\rho \in (0, 1)$, and the definitions of $e^k := \widetilde{S}^k - S^k$ and $U_k$, we can show that
\myeq{eq:AVFR4NE_lm2_upper_bound}{
\hspace{-1ex}
\arraycolsep=0.2em
\begin{array}{lcl}
\Expsn{k}{ \norms{ e^k}^2 } & \leq & \Expsn{k}{ \Delta_k },   \vspace{1ex}\\ 
(t_k - \mu)^2 \Expsn{k}{ \Delta_k } &\leq & \frac{1 - \rho}{\rho }\big[ (t_{k-1}  - \mu)^2 \Delta_{k-1} - (t_k - \mu)^2 \Expsn{k}{ \Delta_k } \big]    \vspace{1ex}\\ 
&& + {~} \frac{1 }{\rho} \big[ \hat{\Theta}(t_{k-1} - \mu)^2 U_{k-1} - \hat{\Theta} (t_k - \mu)^2 U_k \big] + \frac{(\Theta + \hat{\Theta}) (t_k - \mu)^2 }{\rho} U_k.
\end{array} 
\hspace{-3ex}
}
Substituting the two inequalities in \eqref{eq:AVFR4NE_lm2_upper_bound} into \eqref{eq:AVFR4NE_lm2_proof4}, we obtain 
\myeqn{
\arraycolsep=0.2em
\begin{array}{lcl}
\Lc_{k-1}  - \Expsn{k}{ \Lc_k } & \geq & \mu (2t_k - r - \mu)  \norms{x^k - x^{k-1}}^2   \vspace{1ex}\\ 
&& + {~} 4r \beta (t_{k-1} - t_k + r) \big[  \iprods{Gx^{k-1}, x^{k-1} - x^{\star}}  - \beta \norms{Gx^{k-1}}^2 \big]   \vspace{1ex}\\ 
&& + {~} 4\beta (t_k - \mu)^2  \Big[ \frac{(t_k - r - \mu)}{(t_k - \mu)}  \big(\frac{1}{L} - \beta\big) - \frac{ \beta (\Theta + \hat{\Theta}) }{\rho} \Big]  U_k   \vspace{1ex}\\ 
&& - {~} \frac{4 \beta^2(1-\rho)}{ \rho } \big[ (t_{k-1}  - \mu)^2 \Delta_{k-1} - (t_k - \mu)^2 \Expsn{k}{ \Delta_k } \big]    \vspace{1ex}\\ 
&& - {~} \frac{4 \beta^2 }{\rho} \big[ \hat{\Theta}(t_{k-1} - \mu)^2 U_{k-1} - \hat{\Theta} (t_k - \mu)^2 U_k \big].
\end{array}
}
Rearranging this inequality, and using $\Ec_k$ from \eqref{eq:AVFR4NE_Lfunc}, we obtain \eqref{eq:AVFR4NE_key_estimate1}.
\Eproof

\beforesubsec
\subsection{Technical lemmas for proving Theorems~\ref{th:AVFR4NE_convergence} and \ref{th:AVFR4NE_as_convergence}.}\label{apdx:subsec:tech_lemmas}
\aftersubsec
We need the following two technical lemmas to prove Theorem~\ref{th:AVFR4NE_convergence} and Theorem~\ref{th:AVFR4NE_as_convergence}.

\begin{lemma}\label{le:lower_bounding_Ec_0}
Under the same conditions and parameters as in Theorem~\ref{th:AVFR4NE_convergence}, $\Ec_k$ defined by \eqref{eq:AVFR4NE_Lfunc} satisfies
\myeq{eq:Ec0_lower_bound}{
\arraycolsep=0.2em
\begin{array}{lcl}
\Ec_0 & \leq & r\big(1 + 3r + 8rL^2\beta^2 \big) \norms{x^0 - x^{\star}}^2    =  C_0  \norms{x^0 - x^{\star}}^2, 
\end{array}
}
where $C_0 := r (1 + 3r + 8rL^2\beta^2)$ as defined in \eqref{eq:AVFR4NE_constants}. 
\end{lemma}

\proof{\textbf{Proof.}}
Since $x^{-1} = x^0$, $U_0 = 0$, $\eta_0 = \frac{2r\beta }{r+2}$, $\gamma_0 = 0$, $\Delta_0 = 0$, $t_1 = r+2$, and $\widetilde{S}^0 = S^0 = Gx^0$,  \eqref{eq:AVFR4NE_Lfunc} leads to
\myeqn{
\hspace{-1ex}
\arraycolsep=0.2em
\begin{array}{lcl}
\Ec_0 & = & 4r^2\beta  \big[ \iprods{Gx^0, x^0 \! - \! x^{\star}} \! - \! \beta\norms{ Gx^0 }^2 \big] + \norms{r(x^0 \! -  x^{\star}\!) + (r \! + \! 2)(x^1 \! - \! x^0)}^2 
+   r \norms{x^0 - x^{\star}}^2. 
\end{array}
\hspace{-2ex}
}
From \eqref{eq:VRKM4ME}, $x^{-1} = x^0$, and $\widetilde{S}^0 = (1-\gamma_0)Gx^0$, we also have $x^1 - x^0 = \theta_0(x^0 - x^{-1}) - \eta_0\widetilde{S}^0 = -\frac{2r \beta }{r+2}Gx^0$.
By Young's inequality and $\norms{Gx^0}^2 = \norms{Gx^0 - Gx^{\star}}^2 \leq L^2\norms{x^0 - x^{\star}}^2$ (by the $L$-Lipschitz continuity of $G$ derived from \eqref{eq:co-coerciveness_of_G} and $Gx^{\star} = 0$), we have
\myeqn{
\arraycolsep=0.2em
\begin{array}{lcl}
\norms{r(x^0 -  x^{\star}) + (r+2)(x^1 - x^0)}^2 &\leq & 2r^2\norms{x^0 - x^{\star}}^2 + 2(r+2)^2\norms{x^1 - x^0}^2    \vspace{1ex}\\ 
& = & 2r^2\norms{x^0 - x^{\star}}^2 + 8r^2\beta^2\norms{Gx^0}^2   \vspace{1ex}\\ 
& \leq & 2r^2\big( 1 + 4L^2\beta^2\big) \norms{x^0 - x^{\star}}^2.
\end{array}
}
By Young's inequality, we also have 
\myeqn{
\arraycolsep=0.2em
\begin{array}{lcl}
\iprods{Gx^0, x^0 - x^{\star}} - \beta\norms{ Gx^0 }^2 &\leq & \beta\norms{Gx^0}^2 + \frac{1}{4\beta}\norms{x^0 - x^{\star}}^2 - \beta\norms{ Gx^0 }^2  =  \frac{1}{4\beta}\norms{x^0 - x^{\star}}^2.
\end{array}
}
Combining the last three expressions, we can show that
\myeqn{
\arraycolsep=0.2em
\begin{array}{lcl}
\Ec_0 & \leq & r\big(1 + 3r + 8rL^2\beta^2 \big) \norms{x^0 - x^{\star}}^2    =  C_0  \norms{x^0 - x^{\star}}^2, 
\end{array}
}
where $C_0 := r (1 + 3r + 8rL^2\beta^2)$ is defined in \eqref{eq:AVFR4NE_constants}. 
This proves \eqref{eq:Ec0_lower_bound}.
\Eproof
\endproof

\begin{lemma}\label{le:AVFR4NE_key_bound_of_Wk}
Under the same conditions and parameters as in Theorem~\ref{th:AVFR4NE_convergence}, let 
\myeq{eq:AVFR4NE_Wk_def}{
\hspace{-2ex}
\arraycolsep=0.1em
\begin{array}{lcl}
W_k & := & (k \!+\!  r \! - \! 1)(k  \! + \! r \! + \! 2)\norms{ x^{k+1} {\!\!\!} - \! x^k  \! + \! \eta_kGx^k }^2 + \frac{4\beta^2(1 \! - \! \rho)(k+r)^2}{\rho}\Delta_k + \frac{4\beta^2\hat{\Theta}(k+r)^2}{\rho}U_k,
\end{array}
\hspace{-4ex}
}
where $U_k := \frac{1}{n}\sum_{i=1}^n\norms{G_ix^k - G_ix^{k-1} }^2$. 
Then,  for all $k \geq 1$, we have
\myeq{eq:AVFR4NE_Wk_bound}{
\arraycolsep=0.2em
\begin{array}{lcl}
\Expsn{k} {W_k} & \leq & W_{k-1} - (r-2)(k+r+1) \norms{x^k - x^{k-1} + \eta_{k-1}Gx^{k-1} }^2 \vspace{1ex}\\
&& + {~}  \frac{2(2k + r +1)}{r} \norms{x^k - x^{k-1}}^2  + \frac{4\beta^2(\Theta + \hat{\Theta} ) (k + r)^2 }{\rho} U_k.
\end{array}
}
Moreover, we also have
\myeq{eq:AVFR4NE_Deltak_bound}{ 
\arraycolsep=0.2em
\begin{array}{lcl}
\sum_{k=1}^{\infty} (k + r)^2 \Expn{\Delta_k} &\leq &  \frac{(\Theta + \hat{\Theta} )  }{\rho} \sum_{k=1}^{\infty} (k + r)^2 \Expn{ U_k}.
\end{array} 
}
\end{lemma}

\proof{\textbf{Proof.}}
Since $\widetilde{S}^k = S^k + e^k = Gx^k - \gamma_kGx^{k-1} + e^k$ for $e^k := \widetilde{S}^k - S^k$, we obtain from \eqref{eq:VRKM4ME} that $x^{k+1} = x^k + \theta_k(x^k - x^{k-1}) - \eta_k Gx^k + \eta_k\gamma_kGx^{k-1} - \eta_k e^k$.
If we denote $v^k := x^{k+1} - x^k + \eta_kGx^k$ and $s_k := \frac{\eta_k\gamma_k}{\eta_{k-1}}$, then this expression is equivalent to 
\myeqn{
\arraycolsep=0.2em
\begin{array}{lcl}
v^k & = & s_k v^{k-1}   + (1 - s_k)\frac{\theta_k - s_k}{1-s_k} (x^k - x^{k-1}) - \eta_ke^k.
\end{array}
}
Leveraging \eqref{eq:AVFR4NE_para_update2}, we can easily check that $s_k = \frac{(k+r+1)k}{(k+r-1)(k+r+2)} \in (0, 1)$ and $\frac{(\theta_k-s_k)^2}{1-s_k} = \frac{4k^2}{[rk + (r-1)(r+2)](k+r-1)(k+r+2)}$.
Utilizing these expressions, the convexity of $\norms{\cdot}^2$, $\Expsn{k}{e^k} = 0$, $\Expsn{k}{\norms{e^k}^2} \leq  \Expsn{k}{\Delta_k}$, and $\eta_k$ from \eqref{eq:AVFR4NE_para_update}, we can derive that
\myeqn{
\arraycolsep=0.2em
\begin{array}{lcl}
\Expsn{k}{ \norms{ v^k}^2 } & = &  \norms{ s_k v^{k-1}  + (1 - s_k)\frac{\theta_k-s_k}{1-s_k}(x^k - x^{k-1}) }^2  + \eta_k^2\Expsn{k}{\norms{e^k}^2}   \vspace{1ex}\\ 
& \leq & s_k  \norms{v^{k-1}}^2 + \frac{(\theta_k-s_k)^2}{1-s_k}  \norms{x^k - x^{k-1}}^2 +  \eta_k^2\Expsn{k}{ \Delta_k }   \vspace{1ex}\\ 
& = &\frac{(k+r+1)k}{(k+r-1)(k+r+2)}  \norms{v^{k-1}  }^2 + \frac{4\beta^2 ( k+ r)^2}{(k + r +2)^2} \Expsn{k}{ \Delta_k }  \vspace{1ex}\\ 
&& + {~} \frac{4k^2}{[rk + (r-1)(r+2)](k+r-1)(k+r+2)}  \norms{x^k - x^{k-1}}^2   \vspace{1ex}\\ 
&\leq &\frac{(k+r+1)k}{(k+r-1)(k+r+2)}  \norms{v^{k-1}}^2 + \frac{4\beta^2 ( k+ r)^2}{(k+r-1)(k + r +2)} \Expsn{k}{ \Delta_k }   \vspace{1ex}\\ 
&& + {~} \frac{2(2k+r+1)}{r(k+r-1)(k+r+2)}  \norms{x^k - x^{k-1}}^2.
\end{array}
}
Here, we have used $k+r - 1 \leq k + r + 2$ and  $4rk^2 \leq 2(2k+r+1)[rk + (r-1)(r+2)]$ in the last inequality.
Multiplying this inequality by $(k+r - 1)(k+r+2)$, we obtain
\myeqn{ 
\arraycolsep=0.2em
\begin{array}{lll}
(k+r-1)(k+r+2) &  \Expsn{k}{ \norms{ v^k }^2 }  \leq  (k + r - 2)(k+r+1) \norms{v^{k-1} }^2    \vspace{1ex}\\ 
& - {~} (r-2)(k+r+1) \norms{v^{k-1}  }^2 + 4 \beta^2(k+r)^2 \Expsn{k}{ \Delta_k } \vspace{1ex}\\
& + {~} \frac{2(2k + r +1)}{r}  \norms{x^k - x^{k-1}}^2.
\end{array}
}
Next, recalling \eqref{eq:AVFR4NE_lm2_upper_bound} from the proof of Lemma~\ref{le:AVFR4NE_key_estimate} with $t_k - \mu = k + r$ and $U_k = \frac{1}{n}\sum_{i=1}^n\norms{G_ix^k - G_ix^{k-1}}^2$, we get
\myeq{eq:AVFR4NE_lm2_proof100}{
\arraycolsep=0.2em
\begin{array}{lcl}
(k + r)^2 \Expsn{k}{\Delta_k} &\leq & \frac{(1 - \rho )}{\rho }\big[ (k  + r - 1)^2  \Delta_{k-1}  - (k + r)^2 \Expsn{k}{ \Delta_k } \big]    \vspace{1ex}\\ 
&& + {~} \frac{\hat{\Theta} }{\rho} \big[ (k + r - 1)^2 U_{k-1} -  (k + r)^2  U_k  \big] + \frac{(\Theta + \hat{\Theta} ) (k + r)^2 }{\rho}   U_k.
\end{array} 
}
Combining \eqref{eq:AVFR4NE_lm2_proof100} and the last inequality, and using $W_k$ from \eqref{eq:AVFR4NE_Wk_def}, we obtain \eqref{eq:AVFR4NE_Wk_bound}.

Finally, taking the total expectation on both sides of  \eqref{eq:AVFR4NE_lm2_proof100}, and summing the result from $k = 1$ to $K$, and noting that $\Delta_0 = 0$ and $U_0 = 0$, we get
\myeqn{ 
\arraycolsep=0.2em
\begin{array}{lcl}
\sum_{k=1}^K (k + r)^2 \Expn{\Delta_k} &\leq & \frac{(1 - \rho )r^2  \Delta_{0}}{\rho }  + \frac{\hat{\Theta}r^2 }{\rho}  U_0 + \frac{(\Theta + \hat{\Theta} )  }{\rho} \sum_{k=1}^K (k + r)^2 \Expn{ U_k} \vspace{1ex}\\
& = &  \frac{(\Theta + \hat{\Theta} )  }{\rho} \sum_{k=1}^K (k + r)^2 \Expn{ U_k}.
\end{array} 
}
Letting $K\to\infty$, we obtain \eqref{eq:AVFR4NE_Deltak_bound}.
\Eproof
\endproof

\begin{lemma}\label{le:AVFR4NE_auxiliary_lemma2}
Under the same conditions and parameters as in Theorem~\ref{th:AVFR4NE_convergence}, let 
\myeq{eq:AVFR4NE_Zk_def}{
\arraycolsep=0.2em
\begin{array}{lll}
Z_k & := & (k+r+2)^2\norms{ x^{k+1} - x^k}^2 + \frac{4\beta^2(1-\rho)(k+r)^2}{\rho}\Delta_k + \frac{4\beta^2\hat{\Theta}(k+r)^2}{\rho}U_k,
\end{array}
}
where $U_k := \frac{1}{n}\sum_{i=1}^n\norms{G_ix^k - G_ix^{k-1} }^2$.
Then, we have
\myeq{eq:AVFR4NE_Zk_bound}{
\arraycolsep=0.2em
\begin{array}{lcl}
\Expsn{k} {Z_k} & \leq & Z_{k-1} + 4r\beta^2(2k + r + 3) \norms{Gx^k}^2 + \frac{4\beta^2(\Theta + \hat{\Theta} ) (k + r)^2 }{\rho} U_k.
\end{array}
}
\end{lemma}

\proof{\textbf{Proof.}}
Since $x^{k+1} - x^k  = \theta_k(x^k - x^{k-1})  - \eta_k(Gx^k - \gamma_kGx^{k-1}) - \eta_ke^k$ from \eqref{eq:VRKM4ME} with $e^k = \widetilde{S}^k - S^k$ and $\Expsn{k}{e^k} = 0$, we can derive that
\myeqn{
\hspace{-0.5ex}
\arraycolsep=0.2em
\begin{array}{lcl}
\Expsn{k}{ \norms{x^{k+1} - x^k}^2 }  
& = & \theta_k^2   \norms{x^k - x^{k-1}}^2  - 2\theta_k\eta_k \iprods{Gx^k - \gamma_kGx^{k-1}, x^k - x^{k-1}}   \vspace{1ex}\\ 
&& + {~} \eta_k^2 \norms{Gx^k - \gamma_kGx^{k-1}}^2 + \eta_k^2\Expsn{k}{\norms{e^k}^2}   \vspace{1ex}\\ 
& = & \theta_k^2   \norms{x^k - x^{k-1}}^2  - 2\theta_k\eta_k\gamma_k \iprods{Gx^k - Gx^{k-1}, x^k - x^{k-1}}   \vspace{1ex}\\ 
&& - {~} 2\theta_k\eta_k(1-\gamma_k)\iprods{Gx^k, x^k - x^{k-1}} \vspace{1ex}\\
&& + {~}  \eta_k^2 \norms{Gx^k - \gamma_kGx^{k-1}}^2 + \eta_k^2\Expsn{k}{\norms{e^k}^2}.
\end{array}
\hspace{-2ex}
}
By Young's inequality in $\myeqc{1}$, Condition \eqref{eq:co-coerciveness_of_G}, and the definition of $U_k$, we have
\myeqn{
\arraycolsep=0.2em
\begin{array}{lcl}
-2\iprods{Gx^k, x^k - x^{k-1}} & \overset{\tiny\myeqc{1}}{ \leq } & 2\beta \norms{Gx^k}^2 + \frac{1}{2\beta} \norms{x^k - x^{k-1}}^2,   \vspace{0ex}\\ 
\iprods{Gx^k - Gx^{k-1}, x^k - x^{k-1}} & \overset{\tiny\eqref{eq:co-coerciveness_of_G}}{ \geq } & \frac{1}{nL}\sum_{i=1}^n\norms{G_ix^k - G_ix^{k-1}}^2 = \frac{1}{L}U_k.
\end{array}
}
Moreover, by Young's inequality again in $\myeqc{1}$, we also have
\myeqn{
\arraycolsep=0.1em
\begin{array}{lcl}
\norms{Gx^k - \gamma_kGx^{k-1}}^2 & = & (1 - \gamma_k)\norms{Gx^k}^2 - \gamma_k(1-\gamma_k) \norms{Gx^{k-1}}^2 + \gamma_k\norms{Gx^k - Gx^{k-1}}^2   \vspace{0.0ex}\\ 
& \overset{{\tiny\myeqc{1}}}{ \leq } & (1 - \gamma_k)\norms{Gx^k}^2 - \gamma_k(1-\gamma_k) \norms{Gx^{k-1}}^2 + \gamma_k U_k.
\end{array}
}
Combining the last four expressions, and then using $\Expsn{k}{\norms{e^k}^2} \leq \Expsn{k}{ \Delta_k}$, we get
\myeqn{
\arraycolsep=0.2em
\begin{array}{lcl}
\Expsn{k}{ \norms{x^{k+1} - x^k}^2 }  & \leq & \big[ \theta_k^2 + \frac{\theta_k\eta_k(1-\gamma_k)}{2\beta} \big]  \norms{x^k - x^{k-1}}^2   \vspace{1ex}\\ 
&& + {~}  \eta_k(1-\gamma_k)( 2\beta\theta_k + \eta_k ) \norms{Gx^k}^2 - \gamma_k(1-\gamma_k)\eta_k^2 \norms{Gx^{k-1}}^2   \vspace{1ex}\\ 
&& - {~} \gamma_k \eta_k \left( \frac{2\theta_k}{L} - \eta_k \right) U_k  + \eta_k^2\Expsn{k}{\Delta_k}.
\end{array}
}
Next, from the update rule \eqref{eq:AVFR4NE_para_update2}, we can easily check that 
$\theta_k^2 + \frac{\theta_k\eta_k(1-\gamma_k)}{2\beta} = \frac{k(k+r)}{(k+r+2)^2}$, 
$\eta_k(1-\gamma_k) ( 2\beta\theta_k + \eta_k ) = \frac{4\beta^2r (2k+r)}{(k+r+2)^2}$, 
$ \gamma_k(1-\gamma_k)\eta_k^2 =  \frac{4\beta^2r k}{(k+r+2)^2} > 0$, and
$ \frac{2\theta_k}{L} - \eta_k =  \frac{2k}{(k+r+2)L}\big(1 - \frac{(k+r)L\beta}{k} \big) \geq  \frac{2k}{(k+r+2)L}\big(1 - (r+1)L\beta \big) > 0$  due to $0 < \beta <  \bar{\beta}$ in \eqref{eq:AVFR4NE_constants}.
Using these relations and dropping the term with $U_k$, we can derive from the last inequality that
\myeqn{
\arraycolsep=0.2em
\begin{array}{lcl}
\Expsn{k}{ \norms{x^{k+1} - x^k}^2 }  & \leq & \frac{k(k+r)}{(k + r + 2)^2}   \norms{x^k - x^{k-1}}^2   +   \frac{4\beta^2r(2k+r)}{(k+r+2)^2}   \norms{Gx^k}^2 + \frac{4\beta^2(k+r)^2}{(k+r+2)^2} \Expsn{k}{ \Delta_k }.
\end{array}
}
Multiplying this inequality by $(k + r + 2)^2$ and noting that $k(k+r) \leq (k+r+1)^2$ and $2k + r \leq 2k+r+3$, one can further derive
\myeqn{
\arraycolsep=0.2em
\begin{array}{lcl}
(k+r + 2)^2 \Expsn{k}{ \norms{x^{k+1} - x^k}^2 }  & \leq & (k+r+1)^2 \norms{x^k - x^{k-1}}^2  + 4 \beta^2(k+r)^2 \Expsn{k}{ \Delta_k }   \vspace{1ex}\\ 
&& + {~} 4r\beta^2(2k + r + 3) \norms{Gx^k}^2.
\end{array}
}
Combining this inequality and \eqref{eq:AVFR4NE_lm2_proof100}, and using the definition \eqref{eq:AVFR4NE_Zk_def} of $Z_k$, we get \eqref{eq:AVFR4NE_Zk_bound}.
\Eproof
\endproof

\vspace{1ex}
\beforesubsec
\subsection{Proof of Corollaries~\ref{co:SVRG_complexity} and \ref{co:SAGA_complexity}.}\label{apdx:co:SVRG_SAGA_complexity}
\aftersubsec
We now prove  Corollaries~\ref{co:SVRG_complexity} and \ref{co:SAGA_complexity}.

\vspace{1ex}
\proof{\textbf{Proof of Corollary~\ref{co:SVRG_complexity}.}}
Since $4-6\mbf{p} + 3\mbf{p}^2 \leq 4$ and $2 - 3\mbf{p} + \mbf{p}^2 \leq 2$, by Lemma~\ref{le:SVRG_estimator}, $\widetilde{S}^k$ constructed by \eqref{eq:SVRG_estimator} still satisfies Definition~\ref{de:ub_SG_estimator} with  $\rho = \frac{\mbf{p}}{2}$, $\Theta := \frac{4}{b\mbf{p}}$, and $\hat{\Theta} := \frac{4}{b\mbf{p}}$.
Hence,  with $r := 3$, we can show that $\beta := \frac{\bar{\beta}}{2} = \frac{\rho}{2L [\rho + 4(\Theta + \hat{\Theta})]} = \frac{b\mbf{p}^2}{2L(b\mbf{p}^2 + 64)}$, and $\Lambda = \frac{\beta}{2L} =  \frac{b\mbf{p}^2}{4L^2(b\mbf{p}^2 + 64)}$.
Suppose that $1 \leq b\mbf{p}^2 \leq 32$.
Then, we have $\frac{1}{130 L} \leq \beta \leq \frac{1}{6 L}$ and $\psi = \frac{4\beta^2(\Theta + \hat{\Theta})}{\rho\Lambda} = \frac{64  }{64 + b\mbf{p}^2} \in \big(\frac{2}{3}, 1\big)$.

Using the bounds of $\beta$ and $\psi$, and $r := 3$, we can show from \eqref{eq:AVFR4NE_constants} that  
\myeqn{
\arraycolsep=0.2em
\begin{array}{lcllcl}
C_0 & := & 3 (10 + 24L^2\beta^2) \leq 32,   &  C_1 & := &  \frac{72L^2\beta^2}{5} + \big(\frac{2}{3} +  \psi \big) C_0 \leq 53.8, \vspace{1ex}\\  
C_2 & := & 36L^2\beta^2  + \big(\frac{50}{3} + \psi \big) C_0 + \frac{100C_1}{3} \leq 2360,    \ \  & C_3 &:= & \frac{25(C_1 + C_2)}{18 } \leq 3353.
\end{array}
}
Therefore, we obtain \eqref{eq:AVFR4NE_for_SVRG} from \eqref{eq:AVFR4NE_convergence2b} and \eqref{eq:AVFR4NE_convergence2c}.

Finally, since $C_3 = \BigO{1}$ and $\beta = \BigOs{\frac{b\mbf{p}^2}{L}}$, from \eqref{eq:AVFR4NE_for_SVRG}, to guarantee $\Expk{ \norms{Gx^k}^2 } \leq \epsilon^2$, we can impose $\frac{C_3 L^2\Rc_0^2}{b^2\mbf{p}^4(k+2)^2} \leq \epsilon^2$, where $\Rc_0 := \norms{x^0 - x^{\star}}$.
This leads to $k = \BigOs{ \frac{L\Rc_0}{b\mbf{p}^2\epsilon} }$.
Hence, the expected number of evaluations of $G_i$ is $\Expn{ \Tc_{G_i} } = n +  (\mbf{p}n + 3b)k = \BigOs{ n +  \frac{L\Rc_0}{\epsilon}\big( \frac{n}{b\mbf{p}} +   \frac{1}{\mbf{p}^2} \big) }$.
Clearly, if we choose $b = \BigOs{n^{2/3}}$ and $\mbf{p} = \BigOs{\frac{1}{n^{1/3}}}$, then we get $\Expn{ \Tc_{G_i}} = \BigOs{n + \frac{n^{2/3}L\Rc_0}{\epsilon}}$.
\Eproof
\endproof

\vspace{1ex}
\proof{\textbf{Proof of Corollary~\ref{co:SAGA_complexity}.}}
Since $2(n-b)(2n+b) + b^2 \leq 4n^2$ and $2(n-b)(2n+b) \leq 4n^2$, by Lemma~\ref{le:SAGA_estimator}, $\widetilde{S}^k$ constructed by \eqref{eq:SAGA_estimator} still satisfies Definition~\ref{de:ub_SG_estimator} with  $\rho = \frac{b}{2n}$, $\Theta := \frac{4n}{b^2}$, and $\hat{\Theta} := \frac{4n}{b^2}$.
Hence,  for $r := 3$, we can show that $\beta := \frac{\bar{\beta}}{2} = \frac{\rho}{2L [\rho + 4(\Theta + \hat{\Theta})]} = \frac{b^3}{2L(b^3 + 64 n^2)}$.
If $1 \leq b \leq \min\sets{n, 4 n^{2/3}}$, then we have $\frac{1}{2L(1 + 64n^2)} \leq \beta \leq \frac{1}{4L}$ and  $\Lambda := \frac{\beta}{2L} = \BigOs{\frac{b^3}{L^2n^2}} \in \big[ \frac{1}{4L^2(1 + 64n^2)}, \frac{1}{4L^2}\big]$.
Hence, it is easy to check  that $L\beta \leq \frac{1}{4}$ and $\psi := \frac{4\beta^2(\Theta + \hat{\Theta})}{\rho\Lambda} = \frac{64  n^2}{(64n^2 + b^3)} \in \big[\frac{1}{2}, 1\big]$.

Using the bounds of $\beta$ and $\psi$, and $r := 3$ we can show from \eqref{eq:AVFR4NE_constants} that  
\myeqn{
\arraycolsep=0.2em
\begin{array}{lcllcl}
C_0 & := & 3 (10 + 24L^2\beta^2) \leq 34.5,   &  C_1 & := &  \frac{72L^2\beta^2}{5} + \big(\frac{2}{3} +  \psi \big) C_0 \leq 58.4, \vspace{1ex}\\  
C_2 & := & 36L^2\beta^2  + \big(\frac{50}{3} + \psi \big) C_0 + \frac{100C_1}{3} \leq 2559,    \ \  & C_3 &:= & \frac{25(C_1 + C_2)}{18 } \leq 3636.
\end{array}
}
Thus we obtain \eqref{eq:AVFR4NE_for_SAGA} from \eqref{eq:AVFR4NE_convergence2b} and \eqref{eq:AVFR4NE_convergence2c}.

Finally, since $C_3 = \BigO{1}$ and $\beta = \BigO{\frac{b^3}{Ln^2}}$, from \eqref{eq:AVFR4NE_for_SAGA}, to guarantee $\Expk{ \norms{Gx^k}^2 } \leq \epsilon^2$, we can impose $\frac{C_3 L^2\Rc_0^2 n^4}{b^6(k+2)^2} \leq \epsilon^2$.
This leads to $k = \BigOs{ \frac{n^2 L\Rc_0}{b^3\epsilon} }$.
Hence, the expected number of evaluations of $G_i$ is $\Expn{ \Tc_{G_i}} = n + 2bk = \BigOs{ n + \frac{n^2L\Rc_0}{b^2\epsilon}}$.
Clearly, if we choose $b = \BigOs{n^{2/3}}$, then we obtain $\Expn{ \Tc_{G_i}} = \BigOs{n + \frac{n^{2/3}L\Rc_0}{\epsilon}}$.
\Eproof
\endproof

\beforesec
\section{Proof of Theoretical Results for \eqref{eq:VRKM4ME} in Section~\ref{sec:star_monotone_convergence}.}\label{apdx:th:AVFR4SNE_convergence}
\aftersec
First, we prove the following key lemma.

\begin{lemma}\label{le:AVFR4SNE_key_estimate}
Suppose that $G$ in \eqref{eq:ME} satisfies Condition~\eqref{eq:co-coerciveness_of_G} and Assumption~\ref{as:A3}.
Given $s$, $r$, and $\beta$ such that $s-1 > r > (s-1)\rho$ and $4r s^2 > 2(s-1)(2s-r-1)$, and
\myeq{eq:AVFR4SNE_para_condition2}{
\arraycolsep=0.2em
\begin{array}{lll}
0 < \beta\sigma \leq \frac{(2r + 1)(2s-r-1)}{4r s^2 - 2(s-1)(2s-r-1)} \quad \text{and}\quad 0 < \beta\sigma < \frac{(2r+1)\rho}{2[r - \rho(s-1)]}.
\end{array}
}
Let $\sets{x^k}$ be generated by \eqref{eq:VRKM4ME} to solve \eqref{eq:ME} using an estimator $\widetilde{S}^k$ for $S^k$ satisfying Definition~\ref{de:ub_SG_estimator} and  the following fixed parameters:
\myeq{eq:AVFR4SNE_para_update_with_r}{
\arraycolsep=0.2em
\begin{array}{lll}
\theta_k = \theta := \frac{s-r-1}{s}, \quad \gamma_k = \gamma := \frac{s-r-1}{s-1} , \quad \text{and} \quad \eta_k = \eta := \frac{2\beta(s-1)}{s}.
\end{array}
}
For $\omega := \frac{2r \beta\sigma}{2r + 1 + 2\beta\sigma(s-1)}$,  consider the following \textbf{Lyapunov function}:
\myeq{eq:AVFR4SNE_Lfunc}{
\hspace{-1ex}
\arraycolsep=0.2em
\begin{array}{lcl}
\Ec_k &:= & 4r \beta (s-1) \big[ \iprods{Gx^k, x^k - x^{\star}} - \beta \norms{Gx^k}^2 \big]  + \norms{ r(x^k - x^{\star})  + s(x^{k+1} - x^k) }^2  \vspace{1ex} \\
&& + {~}   r \norms{x^k - x^{\star}}^2  +  \frac{4\beta^2(1 - \rho )(s-1)^2 }{\rho - \omega }  \Delta_k + \frac{4\beta^2\hat{\Theta}(s-1)^2}{\rho - \omega} U_k,
\end{array}
\hspace{-2ex}
}
where $U_k := \frac{1}{n}\sum_{i=1}^n\norms{G_ix^k - G_ix^{k-1}}^2$.
Then, for all $k \geq 1$, we have
\myeq{eq:AVFR4SNE_key_estimate}{
\arraycolsep=0.2em
\begin{array}{lcl}
\Expsk{k}{\Ec_k } & \leq & (1-\omega) \Ec_{k-1} - \frac{4\beta r^2(2r+1)}{2r+1 +  2\beta\sigma(s-1)} \big(\frac{1}{2L} - \beta\big) \norms{Gx^{k-1}}^2  -  \alpha  U_k, 
\end{array}
}
where $\alpha := 4\beta(s-1)(s - r-1) \Big[ \frac{1}{L} -  \left( 1   + \frac{(s-1)(\Theta + \hat{\Theta})}{(s-r-1)(\rho - \omega) } \right)\beta \Big]$.
\end{lemma}

\proof{\textbf{Proof.}}
For given $x^{\star} \in \zer{G}$,  $s > 0$, $r > 0$, and $\beta > 0$, we introduce the following function:
\myeq{eq:AVFR4SNE_lm3_proof1}{
\arraycolsep=0.2em
\begin{array}{lcl}
\Lc_k &:= & 4 r\beta(s-1)  \big[ \iprods{Gx^k, x^k - x^{\star}} - \beta\norms{Gx^k}^2 \big] +   r \norms{x^k - x^{\star}}^2 \vspace{1ex}\\
&& + {~}  \norms{ r( x^k - x^{\star} ) + s (x^{k+1} - x^k) }^2.
\end{array}
}
Then, by using \eqref{eq:VRKM4ME} with $\theta_k = \theta := \frac{s-r-1}{s}$, $\gamma_k = \gamma := \frac{s-r-1}{s-1}$, and  $\eta_k = \eta := \frac{2\beta(s-1)}{s}$  from \eqref{eq:AVFR4SNE_para_update_with_r}, with a similar proof as in Lemma~\ref{le:AVFR4NE_key_estimate}, we can show that
\myeq{eq:AVFR4SNE_lm3_proof2}{
\arraycolsep=0.2em
\begin{array}{lcl}
\Lc_{k-1} - \Expsk{k}{\Lc_k } & = &   (2s  - r - 1) \norms{x^k - x^{k-1}}^2  - s^2 \eta^2 \Expsk{k}{ \norms{e^k}^2 }    \vspace{1ex}\\ 
&& + {~} 4\beta r^2  \big[  \iprods{Gx^{k-1}, x^{k-1} - x^{\star}}  - \beta \norms{Gx^{k-1}}^2 \big]   \vspace{1ex}\\ 
&& + {~}4 \beta(s-1) (s - r -  1) \iprods{Gx^k  - Gx^{k-1}, x^k - x^{k-1}}   \vspace{1ex}\\
&& - {~}  4\beta^2 (s-1)(s - r - 1)   \norms{Gx^k - Gx^{k-1}}^2.
\end{array}
}
Combining the $\frac{1}{L}$-cocoercivity of $G$ from Condition \eqref{eq:co-coerciveness_of_G}, $Gx^{\star} = 0$, and Assumption~\ref{as:A3}, we have
\myeqn{
\arraycolsep=0.2em
\begin{array}{lcl}
 \iprods{Gx^{k-1}, x^{k-1} - x^{\star}}  - \beta \norms{Gx^{k-1}}^2 &\geq & \big(\frac{1}{2L} - \beta\big)\norms{Gx^{k-1}}^2 + \frac{\sigma}{2}\norms{x^{k-1} - x^{\star}}^2.
\end{array}
}
From this inequality and Young's inequality, for any $c \in [0, 1]$, we can show that
\myeqn{
\arraycolsep=0.2em
\begin{array}{lcl}
\Tc_{[1]} &:= & 4\beta r^2\big[ \iprods{Gx^{k-1}, x^{k-1} - x^{\star}}  - \beta \norms{Gx^{k-1}}^2 \big] \vspace{1ex}\\
& \geq &  \frac{(1-c) r}{ s - 1}  4r\beta(s-1)\big[ \iprods{Gx^{k-1}, x^{k-1} - x^{\star}}  - \beta \norms{Gx^{k-1}}^2 \big] \vspace{1ex}\\
&& + {~}  4 c \beta r^2  \big(\frac{1}{2L} - \beta\big)\norms{Gx^{k-1}}^2 + 2c \beta \sigma r^2 \norms{x^{k-1} - x^{\star}}^2 \vspace{1ex}\\
& \geq &  \frac{(1-c) r}{ s - 1} 4 r \beta(s-1) \big[ \iprods{Gx^{k-1}, x^{k-1} - x^{\star}}  - \beta \norms{Gx^{k-1}}^2 \big] \vspace{1ex} \\
&& + {~}  \frac{(1-c) r}{ s - 1} \big[ \norms{r(x^{k-1} - x^{\star}) + s(x^k - x^{k-1})}^2 + r\norms{x^{k-1} - x^{\star}}^2 \big] \vspace{1ex}\\
&& + {~} 4 c \beta r^2 \big(\frac{1}{2L} - \beta\big)\norms{Gx^{k-1}}^2 +  r^2 \big[  2c \beta \sigma  - \frac{(1-c)(2r+1)}{s-1}  \big]  \norms{x^{k-1} - x^{\star}}^2 \vspace{1ex}\\
&& - {~}  \frac{2r s^2(1-c)}{ s - 1}\norms{x^k - x^{k-1}}^2 \vspace{1ex}\\
&\overset{\tiny\eqref{eq:AVFR4SNE_lm3_proof1}}{ = } &  \frac{(1-c) r }{ s - 1} \Lc_{k-1} +  4 c \beta r^2 \big(\frac{1}{2L} - \beta\big)\norms{Gx^{k-1}}^2  -  \frac{2r s^2(1-c)}{ s - 1}\norms{x^k - x^{k-1}}^2 \vspace{1ex}\\
&& + {~} r^2 \big[ 2c \beta \sigma  - \frac{(1-c) (2r  + 1)}{s-1} \big] \norms{x^{k-1} - x^{\star}}^2.
\end{array}
}
On the other hand, from Condition~\eqref{eq:co-coerciveness_of_G} and Jensen's inequality in $\myeqc{1}$, we also have
\myeqn{
\arraycolsep=0.2em
\begin{array}{lcl}
\iprods{Gx^k - Gx^{k-1}, x^k - x^{k-1}} & \overset{\tiny \eqref{eq:co-coerciveness_of_G} }{ \geq } & \frac{1}{Ln}\sum_{i=1}^n\norms{G_ix^k - G_ix^{k-1}}^2 = \frac{1}{L}U_k, \vspace{0.5ex} \\
- \norms{Gx^k - Gx^{k-1}}^2 & \overset{\tiny\myeqc{1}}{ \geq } & -\frac{1}{n}\sum_{i=1}^n \norms{G_ix^k - G_ix^{k-1}}^2 = -U_k.
\end{array}
}
Substituting $\Tc_{[1]}$ and the last two inequalities into \eqref{eq:AVFR4SNE_lm3_proof2}, we obtain
\myeqn{
\arraycolsep=0.2em
\begin{array}{lcl}
\Lc_{k-1} - \Expsk{k}{\Lc_k } & \geq &   \frac{(1-c) r}{ s - 1}  \Lc_{k-1}  + \big[ 2s  - r - 1 -  \frac{2r s^2(1-c)}{s-1} \big] \norms{x^k - x^{k-1}}^2    \vspace{1ex}\\ 
&&  +  {~}  r^2 \big[ 2c \beta \sigma  - \frac{(1-c) (2r  + 1)}{s-1} \big] \norms{x^{k-1} - x^{\star}}^2   - s^2 \eta^2 \Expsk{k}{ \norms{e^k}^2 }\vspace{1ex}\\
&& + {~} 4\beta(s-1) (s - r - 1)  \big( \frac{1}{L} -  \beta  \big) U_k  + 4c \beta r^2  \big(\frac{1}{2L} - \beta\big)\norms{Gx^{k-1}}^2.
\end{array}
}
We now impose the following two conditions:
\myeq{eq:AVFR4SNE_lm3_proof4}{
(s-1)(2s  - r - 1) -  2r s^2(1-c)   \geq 0, \quad \text{and} \quad  2c \beta\sigma (s-1)  - (1 - c)(2r + 1) \geq 0.
}
Clearly, if  we choose $c := \frac{2r + 1}{2r + 1 + 2\beta\sigma (s-1)}$, then the second condition of \eqref{eq:AVFR4SNE_lm3_proof4} holds with equality, while the first one becomes
\myeqn{
\arraycolsep=0.2em
\begin{array}{lll}
\beta\sigma \leq \frac{(2r + 1)(2s-r-1)}{4r s^2 - 2(s-1)(2s-r-1)},
\end{array}
}
provided that $2rs^2 - (s-1)(2s-r-1) > 0$.
This is the first condition of \eqref{eq:AVFR4SNE_para_condition2}.

Under the condition \eqref{eq:AVFR4SNE_para_condition2}, if we denote $\omega :=  \frac{(1-c) r }{ s - 1}  = \frac{2r \beta\sigma}{2r + 1 + 2\beta\sigma(s-1)}$, then  the last inequality reduces to the following inequality:
\myeq{eq:AVFR4SNE_lm3_proof6}{
\hspace{-2ex}
\arraycolsep=0.2em
\begin{array}{lcl}
(1-\omega) \Lc_{k-1} - \Expsk{k}{\Lc_k } & \geq &  4\beta(s-1) (s - r -  1)  \big( \frac{1}{L} -  \beta  \big) U_k - 4\beta^2(s-1)^2  \Expsk{k}{ \norms{e^k}^2 } \vspace{0.5ex}\\
&& + {~} \frac{4\beta r^2(2r+1)}{2r+1 +  2\beta\sigma(s-1)} \big(\frac{1}{2L} - \beta\big)\norms{Gx^{k-1}}^2.
\end{array}
\hspace{-4ex}
}
Next, from the second condition of \eqref{eq:AVFR4SNE_para_condition2}, we have $0 < \omega < \rho$.
Then,  from \eqref{eq:ub_SG_estimator}, using $\gamma_k = \gamma > 0$ and $\rho \in (0, 1)$, we can show that
\myeqn{
\arraycolsep=0.2em 
\left\{\begin{array}{llcl}
& \Expsn{k}{ \norms{ e^k}^2 }  & \leq &  \Expsk{k}{ \Delta_k },   \vspace{1ex}\\ 
& \Expsn{k}{ \Delta_k } & \leq &  \frac{1 - \rho}{\rho - \omega} \big[ (1 - \omega) \Delta_{k-1}  - \Expsn{k}{\Delta_k} \big] +  \frac{\hat{\Theta}}{\rho - \omega}\big[ (1 - \omega) U_{k-1} - U_k \big] + \frac{(1 - \omega)\Theta + \hat{\Theta}}{\rho - \omega }  U_k.
\end{array}\right.
}
Substituting these inequalities into \eqref{eq:AVFR4SNE_lm3_proof6} and noting that $(1-\omega)\Theta + \hat{\Theta} \leq \Theta + \hat{\Theta}$, we get
\myeqn{
\hspace{-2ex}
\arraycolsep=0.2em
\begin{array}{lcl}
(1-\omega) \Lc_{k-1} - \Expsn{k}{\Lc_k } & \geq &   \frac{4\beta r^2(2r+1)}{2r+1 +  2\beta\sigma(s-1)} \big(\frac{1}{2L} - \beta\big) \norms{Gx^{k-1}}^2 \\
&& - {~}     \frac{4 \beta^2(s-1)^2(1 - \rho )}{\rho - \omega} \big[ (1 - \omega) \Delta_{k-1}  - \Expsn{k}{ \Delta_k } \big]   \vspace{1ex} \\
&& - {~} \frac{4  \beta^2(s-1)^2\hat{\Theta}}{\rho - \omega}\big[ (1 - \omega) U_{k-1} - U_k \big] \vspace{1ex} \\
&& + {~} 4\beta(s-1)(s - r-1) \Big[  \frac{1}{L} -  \left( 1   + \frac{(s-1)(\Theta + \hat{\Theta})}{(s-r-1)(\rho - \omega) } \right)  \beta \Big]  U_k.
\end{array}
\hspace{-4ex}
}
Rearranging this inequality, and using $\Ec_k$ from \eqref{eq:AVFR4SNE_Lfunc} and $\omega \in (0, 1)$,  we obtain \eqref{eq:AVFR4SNE_key_estimate}.
\Eproof
\endproof

\vspace{1ex}
\beforesubsec
\subsection*{\textbf{Proof of Corollary~\ref{co:complexity_for_AVFR4SNE}.}}\label{apdx:subsec:proof_of_Corollary4}
\aftersubsec
Since $\beta := \frac{1}{2}\bar{\beta}$ with $\bar{\beta}$ defined by \eqref{eq:AVFR4SNE_par_conditions}, we have
\myeq{eq:AVFR4SNE_co10_proof1}{
\arraycolsep=0.2em
\begin{array}{lll}
\frac{1}{\omega} = 2 + \frac{3}{2\beta\sigma} = 2 + \max\set{5, 6\kappa, \frac{2(1-2\rho)}{\rho}, \frac{N + \sqrt{N^2 + 12\rho M}}{2\rho} }.
\end{array}
}
For SVRG in \eqref{eq:SVRG_estimator} with $\mbf{p} = \frac{1}{n^{1/3}}$ and $b = \lfloor n^{2/3}\rfloor$, we have $\rho = \frac{\mbf{p}}{2} = \Theta(n^{-1/3})$ and $\Gamma \leq \frac{\mbf{p}}{2} + \frac{16}{b\mbf{p}} = \BigOs{n^{-1/3}}$.
Hence, $N = 3\Gamma\kappa + 2(1-2\rho) = \BigOs{1+\kappa n^{-1/3}}$ and $M = 2(2\Gamma-1)\kappa = \BigOs{\kappa}$.
If $M\leq 0$, then $\sqrt{N^2+12\rho M}\leq N$; otherwise, $\sqrt{N^2+12\rho M}\leq N+\sqrt{12\rho M}$.
In both cases, we can show that
\myeqn{
\arraycolsep=0.2em
\begin{array}{lcl}
\frac{N + \sqrt{N^2 + 12\rho M}}{2\rho} & \leq & \frac{N}{\rho} + \sqrt{\frac{3M_+}{\rho}} = \BigOs{\kappa+n^{1/3}+\sqrt{\kappa n^{1/3}}} = \BigOs{\kappa+n^{1/3}},
\end{array}
}
where $M_+ := \max\set{M,0}$.
Using this bound in \eqref{eq:AVFR4SNE_co10_proof1} and noting that $k \geq \BigOs{ \frac{1}{\omega}\log(\frac{1}{\epsilon})}$, we get $k = \BigOs{(\kappa + n^{1/3}) \log(\frac{1}{\epsilon})}$.
The expected total oracle complexity of \eqref{eq:VRKM4ME} is
\myeqn{
\arraycolsep=0.2em
\begin{array}{lcl}
\Expn{ \Tc_{G_i}} = n + (n\mbf{p} + 3b)k \leq n + 4n^{2/3}k = \BigOs{n + \max\sets{n, n^{2/3}\kappa} \log(\frac{1}{\epsilon})}.
\end{array}
}
For SAGA in \eqref{eq:SAGA_estimator} with $b = \lfloor n^{2/3} \rfloor$, we also have $\rho = \frac{b}{2n} = \Theta(n^{-1/3})$ and $\Gamma = \rho + 2(\Theta + \hat{\Theta}) = \BigOs{n^{-1/3}}$.
Hence, the remainder of the proof is analogous to the SVRG case, and we omit it here.
\Eproof
\end{APPENDICES}

\bibliographystyle{informs2014}

\end{document}